\documentclass[11pt]{article}
\usepackage[margin=0.75in]{geometry}

\usepackage{amsthm}
\usepackage{amssymb}
\usepackage{amsmath}
\usepackage{comment}
\usepackage{thm-restate}
\usepackage{url} 
\usepackage{hyperref}
\usepackage[noabbrev,capitalise]{cleveref}

\usepackage[utf8]{inputenc}

\usepackage{color}
\usepackage[normalem]{ulem}

\newcounter{i}
\theoremstyle{plain}
\newtheorem{thm}{Theorem}[section]
\newtheorem{lem}[thm]{Lemma}
\newtheorem{claim}{Claim}[thm]
\newtheorem{proposition}[thm]{Proposition}
\newtheorem{cor}[thm]{Corollary}

\newtheorem{remark}[thm]{Remark}
\newtheorem{conj}[thm]{Conjecture}
\newtheorem{question}[thm]{Question}

\newenvironment{proofclaim}[1][]%
{\noindent \emph{Proof.} {}{#1}{}}{\hfill
	$\Diamond$\vspace{1em}}

\theoremstyle{plain} % just in case the style had changed
\newcommand{\thistheoremname}{}
\newtheorem{genericthm}{\thistheoremname}

\theoremstyle{definition}
\newtheorem{definition}[thm]{Definition}

\newcommand{\Prob}[1]{\ensuremath{%
\mathbb P\left[#1\right]
}}

\newcommand{\Expect}[1]{\ensuremath{%
\mathbb E\left[#1\right]
}}

\DeclareMathOperator{\Med}{Med}

\title{Finding an almost perfect matching in a hypergraph avoiding forbidden submatchings}

\author{
Michelle Delcourt
\thanks{Department of Mathematics, Toronto Metropolitan University (formerly named Ryerson University),
Toronto, Ontario M5B 2K3, Canada {\tt mdelcourt@ryerson.ca}. Research supported by NSERC under Discovery Grant No. 2019-04269 and a Sloan Research Fellowship.}
\and
Luke Postle
\thanks{Combinatorics and Optimization Department,
University of Waterloo, Waterloo, Ontario N2L 3G1, Canada {\tt lpostle@uwaterloo.ca}. Partially supported by NSERC
under Discovery Grant No. 2026-04610.}}
\date{\today}

\begin{document}

\maketitle

\begin{abstract} 
In 1973, Erd\H{o}s conjectured the existence of high girth $(n,3,2)$-Steiner systems. Recently, Glock, K\"{u}hn, Lo, and Osthus and independently Bohman and Warnke proved the approximate version of Erd\H{o}s' conjecture. Recently Kwan, Sah, Sawhney, and Simkin proved Erd\H{o}s' conjecture. As for Steiner systems with more general parameters, Glock, K\"{u}hn, Lo, and Osthus conjectured the existence of high girth $(n,q,r)$-Steiner systems. We prove the approximate version of their conjecture.

This result follows from our general main results which concern finding perfect or almost perfect matchings in a hypergraph $G$ avoiding a given set of submatchings (which we view as a hypergraph $H$ where $V(H)=E(G)$). Our first main result is a common generalization of the classical theorems of Pippenger (for finding an almost perfect matching) and Ajtai, Koml\'os, Pintz, Spencer, and Szemer\'edi (for finding an independent set in girth five hypergraphs). More generally, we prove this for coloring and even list coloring, and also generalize this further to when $H$ is a hypergraph with small codegrees (for which high girth designs is a specific instance). Indeed, the coloring version of our result even yields an almost partition of $K_n^r$ into approximate high girth $(n,q,r)$-Steiner systems. 

Our main results also imply the existence of a perfect matching in a bipartite hypergraph where the parts have slightly unbalanced degrees. This has a number of other applications; for example, it proves the existence of $\Delta$ pairwise disjoint list colorings in the setting of Kahn's theorem for list coloring the edges of a hypergraph; it also proves asymptotic versions of various rainbow matching results in the sparse setting (where the number of times a color appears could be much smaller than the number of colors) and even the existence of many pairwise disjoint rainbow matchings in such circumstances. 
\end{abstract}

\section{Introduction}

\subsection{High Girth Steiner Systems}

A \emph{(multi)-hypergraph} consists of a pair $(V, E)$ where $V$ is a set whose elements are called \emph{vertices} and $E$ is a
(multi)-set of subsets of $V$ called \emph{edges}. For an integer $r\ge 1$, a hypergraph is \emph{$r$-bounded} if every edge has size at most $r$ and is \emph{$r$-uniform} if every edge has size exactly $r$. 

Given $n\ge q > r\ge 2$, a \emph{partial $(n,q,r)$-Steiner system} is a set $S$ of $q$-element subsets of an $n$-element set $V$ such that every $r$-element subset of $V$ is contained in at most one $q$-element set in $S$. 
An \emph{$(n,q,r)$-Steiner system} is a partial $(n,q,r)$-Steiner system $S$ with $|S| = \binom{n}{r}/\binom{q}{r}$, i.e.~every $r$-element subset of $V$ is in exactly one $q$-element set of $S$. For fixed $q$ and $r$ we call $n$ \emph{admissible} if $\binom{q-i}{r-i}~|~\binom{n-i}{r-i}$ for all $0\le i \le r-1$. 

The notorious Existence Conjecture dating from the mid-1800's asserts that for sufficiently large $n$, there exists an $(n,q,r)$-Steiner system whenever $n$ is admissible.  In 1847, Kirkman~\cite{K47} proved this when $q=3$ and $r=2$. In the 1970s, Wilson~\cite{W75} proved the Existence Conjecture for graphs, i.e.~when $r=2$. In 1985, R\"{o}dl~\cite{R85} introduced his celebrated ``nibble method'' to prove that there exists a partial $(n,q,r)$-Steiner system $S$ with $|S|=(1-o(1))\binom{n}{r}/\binom{q}{r}$, settling the approximate version of the Existence Conjecture (known as the Erd\H{o}s-Hanani conjecture~\cite{EH63}). Recently, Keevash~\cite{K14} proved the Existence Conjecture using randomized algebraic constructions. Thereafter, Glock, K\"{u}hn, Lo, and Osthus~\cite{GKLO16} gave a purely combinatorial proof of the Existence Conjecture via iterative absorption.

In a partial $(n,q,r)$-Steiner system, a \emph{$(j, i)$-configuration} is a set of $i$ $q$-element subsets spanning at most $j$ vertices. It is easy to see that an $(n,3,2)$-Steiner system contains an $(i+3,i)$-configuration for every fixed $i$.  In 1973, Erd\H{o}s~\cite{E73} conjectured the following. 
%``locally sparse'' or ``high girth''.

\begin{conj}[Erd\H{o}s~\cite{E73}]\label{conj:Erdos}
For every integer $g\geq2$, there exists $n_g$ such that for all admissible
$n \geq n_g$, there exists an $(n,3,2)$-Steiner system with no
$(i+2,i)$-configuration for all $2 \leq i\leq g$.
\end{conj}

In 1993, Lefmann, Phelps, and R\"odl~\cite{LPR93} proved that for every integer $g\geq2$, there exists a constant $c_g>0$ such that for all $n\geq 3$ there exists a partial $(n,3,2)$-Steiner system $S$ with $|S|\ge c_g \cdot n^2$ and no $(i+2,i)$-configuration for all $2 \leq i\leq g$. Recently, Glock, K\"{u}hn, Lo, and Osthus~\cite{GKLO20} and independently Bohman and Warnke~\cite{BW19} proved there exists a partial $(n,3,2)$-Steiner system $S$ with $|S|=(1-o(1))n^2/6$ and no $(i+2,i)$-configuration for all $2 \leq i\leq g$, thus settling the approximate version of Conjecture~\ref{conj:Erdos}. Recently Kwan, Sah, Sawhney, and Simkin~\cite{KSSS22}, building on these approximate versions while incorporating iterative absorption, impressively proved Conjecture~\ref{conj:Erdos} in full. 

Glock, K\"{u}hn, Lo, and Osthus~\cite{GKLO20} observed that every $(n,q,r)$-Steiner system contains an $(i(q-r)+r+1,i)$-configuration for every fixed $i\ge 2$. They also gave a short proof that there exists a partial $(n,q,r)$-Steiner system $S$ with $|S|=(1-o(1))\binom{n}{r}/\binom{q}{r}$ with no $(i(q-r)+2r-q, i)$-configurations for all $2\le i \le g$. Thus they made the following common generalization of Erd\H{o}s' conjecture and the Existence Conjecture.

\begin{conj}[Glock, K\"{u}hn, Lo, and Osthus~\cite{GKLO20}]\label{conj:HighGirthAllUniformities}
For all integers $q > r \geq 2$ and every integer $g\ge 2$, there exists $n_0$ such that for all admissible $n\ge n_0$, there exists an $(n,q,r)$-Steiner system with no $(i(q-r)+r,i)$-configurations for all $2\le i \le g$.
\end{conj}

For $r=2$ and all $q\geq 3$, this was asked by F\"uredi and Ruszink\'o~\cite{FR13} in 2013 while Conjecture~\ref{conj:HighGirthAllUniformities} was reiterated by Keevash and Long~\cite{KL20} in 2020. 

Our first main result of this paper is that we prove the approximate version of this conjecture as follows.

\begin{thm}\label{thm:HighGirthSteiner}
For all integers $q > r \geq 2$ and every integer $g\ge 2$, there exists $n_0$ and $\beta\in (0,1)$ such that for all $n\ge n_0$, there exists a partial $(n,q,r)$-Steiner system $S$ with $|S|\ge (1-n^{-\beta})\binom{n}{r}/\binom{q}{r}$ and no $(i(q-r)+r,i)$-configurations for all $2\le i \le g$.
\end{thm}

Indeed, our original motivation for this paper was to prove Theorem~\ref{thm:HighGirthSteiner}; however, during the course of our research, we discovered a surprising yet much more general phenomenon that unites two prominent streams of research: matchings in hypergraphs and independent sets in hypergraphs of girth at least five. Our core main result, Theorem~\ref{thm:GirthFive}, unites the seminal, yet classical theorems of Pippenger, Theorem~\ref{thm:Pippenger} below (albeit with stronger codegree assumptions), and that of Ajtai, Koml\'os, Pintz, Spencer, and Szemer\'edi~\cite{AKPSS82}, Theorem~\ref{thm:AKPSS} below. 

Crucially we prove our core main result, Theorem~\ref{thm:GirthFive}, via the nibble method and so we are able to add a number of bells and whistles to the proofs, that is various tweaks or additional clever reductions. This allows us to prove coloring and list coloring versions while also allowing much less restrictive codegree assumptions. Theorem~\ref{thm:HighGirthSteiner} then follows easily from our most general theorem, Theorem~\ref{thm:SmallCodegree}; additionally, its coloring version even implies an almost partition of $K_n^r$ into high girth partial $(n,q,r)$-Steiner systems, each with $(1-n^{-\beta})\binom{n}{r}/\binom{q}{r}$ $q$-element subsets.

We believe that Theorem~\ref{thm:SmallCodegree} (or variations of it) has potential for even more applications; already, it immediately yields many similar consequences for other high girth design problems. Namely, for problems where R\"odl's nibble approach works to yield approximate decompositions, we are able to derive high girth versions. We detail a selection of these applications in Section~\ref{ss:Designs}. Before stating our main results, we briefly discuss these two streams of research.

\vskip.15in
\noindent {\bf Matchings in Hypergraphs.} In 1985, Frankl and R\"{o}dl~\cite{FR85} generalized R\"{o}dl's result to matchings in hypergraphs, namely to regular hypergraphs of small codegree. Recall that in a hypergraph $G$, the \emph{codegree} of two distinct vertices $u,v$ of $G$ is the number of edges of $G$ containing both $u$ and $v$. The classical result of Pippenger (unpublished) generalizes their result further as follows, where we say a hypergraph is \emph{$(1\pm \gamma)D$-regular} if every vertex has degree at least $(1-\gamma)D$ and at most $(1+\gamma)D$.

\begin{thm}[Pippenger]\label{thm:Pippenger}
For every $r, \varepsilon > 0$, there exists $\gamma > 0$ such that the following holds. If $G$ is an $r$-uniform $(1 \pm \gamma)D$-regular hypergraph on $n$ vertices with codegrees at most $\gamma D$, then there is a matching in $G$ covering all but at most $\varepsilon n$ vertices.
\end{thm}

In 1989, Pippenger and Spencer~\cite{PS89} generalized Theorem~\ref{thm:Pippenger} further by proving that $G$ has chromatic index $(1+o(1))D$. In 1996, Kahn~\cite{K96} generalized this even further by proving $G$ has list chromatic index $(1+o(1))D$.

\vskip.15in
\noindent {\bf Independent Sets in Girth Five Hypergraphs.} For a hypergraph $J$, a \emph{coloring} of $J$ is an assignment of colors to the vertices of $J$ so that no edge contains only vertices of the same color. The \emph{chromatic number} of $J$, denoted $\chi(J)$ is the minimum number of colors needed to color the vertices of $J$.  An easy application of the Lov\'asz Local Lemma yields a bound of $O\left(\Delta^{1/(r-1)}\right)$ for the chromatic number of $r$-uniform hypergraphs of maximum degree $\Delta$. However, in the case of forbidding small cycles better bounds are known as follows.

Let $i\ge 2$ be an integer. An \emph{$i$-cycle} in a hypergraph $H$ is an alternating sequence of distinct edges and vertices $e_1,v_1,\ldots, e_i, v_i$ where $v_j \in e_j\cap e_{(j+1)\mod i}$ for all $j\in [i]$. A hypergraph is called \emph{linear} if it has no $2$-cycles. The \emph{girth} of a hypergraph $H$ is the minimum $i$ such that $H$ has an $i$-cycle. 

The classical result of Ajtai, Koml\'os, Pintz, Spencer, and Szemer\'edi~\cite{AKPSS82} from 1982 improves the trivial lower bound for independence number by a $\log \Delta$ factor for hypergraphs of girth at least five.

\begin{thm}[Ajtai, Koml\'os, Pintz, Spencer, Szemer\'edi~\cite{AKPSS82}]\label{thm:AKPSS}
If $H$ is an $r$-uniform hypergraph on $n$ vertices of girth at least five and maximum degree $\Delta$, then $H$ has an independent set of size at least $\Omega\left(n\cdot \left(\frac{\log \Delta}{\Delta}\right)^{1/(r-1)}\right)$.
\end{thm}

We note that the original statement only requires the weaker assumption of average degree $\Delta$ instead of maximum degree but the two results are equivalent up to a constant (of $4$) since every hypergraph $H$ on $n$ vertices with average degree at most $\Delta$ has a subset $S\subseteq V(H)$ with $|S|\ge n/2$ such that $H[S]$ has maximum degree at most $2\Delta$ (since at most half the vertices have twice the average degree). We use the maximum degree formulation here to compare this to more general results about coloring hypergraphs where the maximum degree condition is indeed required.

We also note that implicit in the proof of Theorem~\ref{thm:AKPSS} by Ajtai et al.~\cite{AKPSS82} (due to the inductive nature of the proof) is a mixed uniformity version of their result as follows (see~\cite{LL20} for a formal proof). The \emph{$i$-degree} of a vertex $v$ of $H$, denoted $d_{H,i}(v)$, is the number of edges of $H$ of size $i$ containing $v$. The \emph{maximum $i$-degree} of $H$, denoted $\Delta_i(H)$, is the maximum of $d_{H,i}(v)$ over all vertices $v$ of $H$. Their proof yields that if $H$ is $r$-bounded and for some $D>0$, we have $\Delta_i(H) \le O(D^{i-1}\log D)$ for all $2\le i\le r$, then $H$ has an independent set of size at least $n/D$. Our results also generalize this mixed uniformity version and indeed mixed uniformities are necessary for the applications to high girth Steiner systems.

Returning to the uniform case, Theorem~\ref{thm:AKPSS} was generalized by Duke, Lefmann, and R\"{o}dl~\cite{DLR95} to linear hypergraphs (i.e.~girth at least three). Frieze and Mubayi~\cite{FM13} generalized this further to coloring by proving that linear hypergraphs of maximum degree $\Delta$ have chromatic number at most $O\left(\left(\Delta/\log \Delta\right)^{\frac{1}{r-1}}\right)$. Cooper and Mubayi~\cite{CM16} generalized this even further by relaxing the linearity condition to a bounded codegree condition (see Theorem~\ref{thm:CopperMubayi} for a precise statement). Very recently, Li and the second author~\cite{LP22} proved the mixed uniformity version of their codegree result.

\vskip.15in
\noindent {\bf So now we ask a far-reaching question: Can we combine these two streams of research?}\\ This paper answers this question in the affirmative.

\subsection{Main Results: Combining Two Streams of Research}

To that end, we introduce the following key definition.

\begin{definition}[Configuration Hypergraph]
Let $G$ be a (multi)-hypergraph. We say a hypergraph $H$ is a \emph{configuration hypergraph} for $G$ if $V(H)=E(G)$ and $E(H)$ consists of a set of matchings of $G$ of size at least two. We say a matching of $G$ is \emph{$H$-avoiding} if it spans no edge of $H$.
\end{definition}

Note that when attempting to find an $H$-avoiding matching of $G$, one may assume without loss of generality that no edge of $H$ contains another edge of $H$. Recall that the \emph{line graph} of a hypergraph $G$, denoted $L(G)$, is the graph where $V(L(G)) = E(G)$ and $E(L(G)) = \{uv: u,v \in E(G), u\cap v\ne \emptyset\}$. Thus we ask the following question.

\begin{question}[Edge+Vertex Phenomenon]
Let $G$ be a hypergraph, $H$ a configuration hypergraph of $G$ and $J=L(G)\cup H$. Under what conditions does it hold that: $J$ has independence number (resp.~chromatic number) of size almost the minimum (resp.~maximum) of the independence numbers (resp.~chromatic numbers) of $H$ and $L(G)$?
\end{question}

The core main result of this paper then is that we can find a common generalization of Theorem~\ref{thm:Pippenger} (albeit with stronger codegree assumptions) for a hypergraph $G$ and the mixed uniformity version of Theorem~\ref{thm:AKPSS} for its configuration hypergraph $H$. To that end, it is natural to make an additional assumption on the `codegree' of $G$ with $H$, namely to assume it is small compared to $D$. Hence we introduce the following definition.

\begin{definition}
Let $G$ be a hypergraph and $H$ be a configuration hypergraph of $G$. We define the \emph{codegree} of a vertex $v\in V(G)$ and $e\in E(G)=V(H)$ with $v\notin e$ as the number of edges of $H$ that contain $e$ and an edge incident with $v$. We then define the \emph{maximum codegree} of $G$ with $H$ as the maximum codegree over vertices $v\in V(G)$ and edges $e\in E(G)=V(H)$ with $v\notin e$. 
\end{definition}

Here is our core main result.

\begin{thm}\label{thm:GirthFive}
For all integers $r,g \ge 2$ and real $\beta \in \left(0,\frac{1}{8r}\right]$, there exists an integer $D_{\beta}\ge 0$ such that the following holds for all $D\ge D_{\beta}$: 
\vskip.05in
Let $G$ be an $r$-uniform $(1\pm D^{-\beta})D$-regular (multi)-hypergraph with codegrees at most $\log^2 D$ and let $H$ be a $g$-bounded configuration hypergraph of $G$ of girth at least five with $\Delta_i(H) \le \frac{\beta}{4g^2} \cdot D^{i-1}\log D$ for all $2\le i\le g$. If the maximum codegree of $G$ with $H$ is at most $\log^{2(g-1)} D$, then there exists an $H$-avoiding matching of $G$ of size at least $\frac{v(G)}{r} \cdot (1-D^{-\frac{\beta}{16r}})$.
\end{thm}

The proof of Theorem~\ref{thm:GirthFive} uses R\"odl's nibble approach while the proofs of the approximate and exact versions of Conjecture~\ref{conj:Erdos} analyzed a ``high-girth" process (which is equivalent to the random greedy independent set process on $L(G)\cup H$). Thus our proof provides independent proofs of the approximate versions of Conjecture~\ref{conj:Erdos} while also providing a new perspective on why exactly Conjecture~\ref{conj:Erdos} is true.

We should mention here that Glock, Joos, Kim, K\"{u}hn, and Lichev~\cite{GJKKL22} have independently obtained some results of a similar flavor. First, they also prove Conjecture~\ref{conj:HighGirthAllUniformities}. Second, they also develop a general theory of finding almost perfect matchings avoiding forbidden submatchings (what they call \emph{conflict-free hypergraph matchings}). Their general main result is to prove the independent set version of our Theorem~\ref{thm:SmallCodegree} (by independent set version we mean in the sense of Theorem~\ref{thm:GirthFive} but with the less restrictive codegree assumptions of Theorem~\ref{thm:SmallCodegree} below which we discuss next). See Section~\ref{ss:OtherWork} for more discussion on the differences in the results and proofs.

Theorem~\ref{thm:GirthFive} is too restrictive to prove Theorem~\ref{thm:HighGirthSteiner} given that Theorem~\ref{thm:GirthFive} assumes the codegrees of $G$ and the codegrees of $G$ with $H$ are quite small and that $H$ has girth at least five. We will remedy this deficiency in due course by allowing much weaker codegree assumptions, but first we discuss other avenues of generalizing Theorem~\ref{thm:GirthFive}. For example, it is quite natural given the previous work in the two research streams to wonder if the coloring or even the list coloring version of Theorem~\ref{thm:GirthFive} holds. Both of these will actually follow from our first major bell, which is to generalize Theorem~\ref{thm:GirthFive} further to finding a perfect matching in bipartite hypergraphs as follows, but first a definition.

\begin{definition}
We say a hypergraph $G=(A,B)$ is \emph{bipartite with parts $A$ and $B$} if $V(G)=A\cup B$ and  every edge of $G$ contains exactly one vertex from $A$. We say a matching of $G$ is \emph{$A$-perfect} if every vertex of $A$ is in an edge of the matching.
\end{definition}

The study of matchings in bipartite hypergraphs has also garnered much attention over the years. Whereas in graphs the classical Hall's theorem characterizes when a bipartite graph $G=(A,B)$ has an $A$-perfect matching, even in $3$-uniform bipartite hypergraphs it is NP-complete to determine if there exists an $A$-perfect matching. Nevertheless, there are many results establishing sufficient conditions on a bipartite hypergraph $G=(A,B)$ to guarantee the existence of an $A$-perfect matching, for example Aharoni and Kessler~\cite{AK90}, Haxell~\cite{H95}, Aharoni and Haxell~\cite{AH00}, and Meshulam~\cite{M01} to name a few.

Our next main result thus finds an $H$-avoiding $A$-perfect matching in a bipartite hypergraph provided that the degrees of $A$ are slightly larger than the degrees of $B$ as follows. 

\begin{thm}\label{thm:GirthFiveBipartite}
For all integers $r,g \ge 2$, there exists an integer $D_{r,g}\ge 0$ such that the following holds for all $D\ge D_{r,g}$ where we let $\alpha = \frac{1}{20r} > 0$: 
\vskip.05in
Let $G=(A,B)$ be a bipartite $r$-bounded (multi)-hypergraph with codegrees at most $\log^2 D$ such that every vertex in $A$ has degree at least $(1+D^{-\alpha})D$ and every vertex in $B$ has degree at most $D$. Let $H$ be a $g$-bounded configuration hypergraph of $G$ of girth at least five with $\Delta_i(H) \le \frac{\alpha}{g^2} \cdot D^{i-1}\log D$ for all $2\le i\le g$. If the maximum codegree of $G$ with $H$ is at most $\log^{2(g-1)} D$, then there exists an $H$-avoiding $A$-perfect matching of $G$.
\end{thm}

The proof of Theorem~\ref{thm:GirthFiveBipartite} is the same nibble process as in the proof of Theorem~\ref{thm:GirthFive} with a few minor tweaks coupled with a finishing lemma to extend the matching to an $A$-perfect matching. While at first glance, Theorem~\ref{thm:GirthFiveBipartite} may not appear to be a generalization of Theorem~\ref{thm:GirthFive}, it in fact implies the coloring and even list coloring generalization of Theorem~\ref{thm:GirthFive}. 

Before that though, we add our first whistle which is the existence of many disjoint matchings as follows. While one may think that this would require much more work than proving Theorem~\ref{thm:GirthFiveBipartite}, it is actually a direct consequence of Theorem~\ref{thm:GirthFiveBipartite}; namely by applying Theorem~\ref{thm:GirthFiveBipartite} not to $G$ itself but to an auxiliary hypergraph of $G$.

\begin{thm}\label{thm:GirthFiveBipartiteMany}
Under the assumptions of Theorem~\ref{thm:GirthFiveBipartite} (except for setting $\alpha = \frac{1}{20(r+1)}$ instead) there exists a set of $D$ disjoint $H$-avoiding $A$-perfect matchings of $G$.
\end{thm}
\begin{proof}[Proof (assuming Theorem~\ref{thm:GirthFiveBipartite}).]
For each edge $e$ of $G$ and $i\in [D]$, define a set $e_{i} := \{(v,i):v\in e\}\cup \{e\}$. We define an auxiliary hypergraph $G' = (A', B')$ as:
$$A' := A \times [D],~~B' := (B \times [D]) \cup E(G)$$
and
$$E(G') := \{ e_i: e\in E(G),~i\in [D]\}.$$
Define a configuration hypergraph $H'$ of $G'$ where
$$E(H') := \left\{ \{e_{i}: e\in S\}: S\in E(H),~i\in [D]\right\}.$$

Now an $H'$-avoiding $A'$-perfect matching of $G'$ is precisely a set of $D$ disjoint $H$-avoiding $A$-perfect matchings of $G$. Moreover, $G'$ and $H'$ satisfy the hypotheses of Theorem~\ref{thm:GirthFiveBipartite}. Hence by Theorem~\ref{thm:GirthFiveBipartite}, there exists an $H'$-avoiding $A'$-perfect matching of $G'$ as desired.
\end{proof}

Similarly, a short argument yields the following corollary which shows that the coloring, list coloring, and even the many disjoint list coloring versions of Theorem~\ref{thm:GirthFive} hold; namely, we apply Theorem~\ref{thm:GirthFiveBipartiteMany} to an auxiliary hypergraph of the graph in Theorem~\ref{thm:GirthFive}. 

\begin{cor}\label{cor:GirthFiveList}
For all integers $r,g \ge 2$ and real $\beta \in \left(0,\frac{1}{8r}\right]$, there exists an integer $D_{\beta}\ge 0$ such that the following holds for all $D\ge D_{\beta}$,where we let $\alpha = \frac{1}{20r} > 0$: 
\vskip.05in
Let $G$ be an $r$-uniform $(1\pm D^{-\beta})D$-regular (multi)-hypergraph with codegrees at most $\log^2 D$ and let $H$ be a $g$-bounded configuration hypergraph of $G$ of girth at least five with $\Delta_i(H) \le \frac{\beta}{4g^2} \cdot D^{i-1}\log D$ for all $2\le i\le g$. If the maximum codegree of $G$ with $H$ is at most $\log^{2(g-1)} D$, then 
$$\chi(L(G)\cup H) \le \chi_{\ell}(L(G)\cup H) \le D(1+D^{-\alpha}).$$
Furthermore if $L$ is a list assignment of the edges of $G$ such that $|L(e)|\ge L_0 \ge D(1+D^{-\alpha})$ for each edge $e$ of $G$, then there exists a set of $L_0-D^{1-\alpha}$ (which is $\ge D$) pairwise disjoint $L$-colorings of $L(G)\cup H$.
\end{cor}
\begin{proof}[Proof (assuming Theorem~\ref{thm:GirthFiveBipartiteMany}).]
Let $L$ be an $L_0$-list-assignment of $E(G)$ and let $C:=\bigcup_{e\in E(G)} L(e)$. For each edge $e$ of $G$ and $c\in L(e)$, define a set $e_c := \{(v,c):v\in e\}\cup \{e\}$. Define an auxiliary $(r+1)$-uniform bipartite hypergraph $G_L=(A_L,B_L)$ where 
$$A_L := E(G),~~B_L := \bigcup\{(v,c):v\in V(G),~c\in C\}$$ and
$$E(G_L) := \{ e_c : e\in E(G),~c\in L(e)\}.$$
Define a configuration hypergraph $H_L$ of $G_L$ where
$$E(H_L) := \left\{ \{e_c: e\in S\}: S\in E(H),~c\in \bigcap_{e\in S} L(e)\right\}.$$

Now an $H_L$-avoiding $A_L$-perfect matching of $G_L$ is precisely an $L$-coloring of $E(G)$ where each color class is an $H$-avoiding matching of $G$. Moreover, $G_L$ and $H_L$ satisfy the hypotheses of Theorem~\ref{thm:GirthFiveBipartiteMany} since $L_0 \ge D(1+D^{-\alpha})$. Hence by Theorem~\ref{thm:GirthFiveBipartiteMany}, there exists $L_0-D^{1-\alpha}$ disjoint $H_L$-avoiding $A_L$-perfect matchings of $G_L$ as desired.
\end{proof}

For our final bell, we generalize all of the above results to where each of the various codegree conditions are allowed to be much less restrictive, but first some definitions.

\begin{definition}
Let $H$ be a hypergraph. The \emph{maximum $(k,\ell)$-codegree} of $H$ is $$\Delta_{k,\ell}(H) := \max_{S\in \binom{V(H)}{\ell}} |\{e\in E(H): S\subseteq e, |e|=k\}|.$$
We define the \emph{common $2$-degree} of distinct vertices $u, v\in V(H)$ as $|\{w\in V(H): uw, vw\in E(H)\}|$. Similarly, we define the \emph{maximum common $2$-degree} of $H$ is the maximum of the common $2$-degree of $u$ and $v$ over all distinct pairs of vertices $u,v$ of $H$.
\end{definition}

\begin{definition}
Let $G$ be a hypergraph and let $H$ be a configuration hypergraph of $G$. We define the \emph{$i$-codegree} of a vertex $v\in V(G)$ and $e\in E(G)=V(H)$ with $v\notin e$ as the number of edges of $H$ of size $i$ that contain $e$ and an edge incident with $v$. We then define the \emph{maximum $i$-codegree} of $G$ with $H$ as the maximum $i$-codegree over vertices $v\in V(G)$ and edges $e\in E(G)=V(H)$ with $v\notin e$. 
\end{definition}

Here is our most general result which generalizes Theorem~\ref{thm:GirthFiveBipartiteMany} by allowing the codegrees of $G$ as well as the maximum codegree of $G$ with $H$ to be much larger, while $H$ no longer need be girth five or even linear but rather have small codegrees and small common $2$-degrees.

\begin{thm}\label{thm:SmallCodegree} 
For all integers $r,g \ge 2$ and real $\beta \in (0,1)$, there exist an integer $D_{\beta}\ge 0$ and real $\alpha > 0$ such that the following holds for all $D\ge D_{\beta}$: 
\vskip.05in
Let $G=(A,B)$ be a bipartite $r$-bounded (multi)-hypergraph with codegrees at most $D^{1-\beta}$ such that every vertex in $A$ has degree at least $(1+D^{-\alpha})D$ and every vertex in $B$ has degree at most $D$. Let $H$ be a $g$-bounded configuration hypergraph of $G$ with $\Delta_i(H) \le \alpha \cdot D^{i-1}\log D$ for all $2\le i\le g$ and $\Delta_{k,\ell}(H) \le D^{k-\ell-\beta}$ for all $2\le \ell< k\le g$. If the maximum $2$-codegree of $G$ with $H$ and the maximum common $2$-degree of $H$ are both at most $D^{1-\beta}$, then there exists an $H$-avoiding $A$-perfect matching of $G$ and indeed even a set of $D$ disjoint $H$-avoiding $A$-perfect matchings of $G$.
\end{thm}

Theorem~\ref{thm:SmallCodegree} is derived from Theorem~\ref{thm:GirthFiveBipartiteMany} via a random sparsification trick. Namely, we choose each edge of $G$ independently with some small probability $p$ and delete any edges of $G$ that still lie in cycles of $H$ of length at most four. While the concept is straightforward, the execution is quite technical and intricate, requiring layering many applications of the exceptional outcomes version of our linear Talagrand's inequality, Theorem~\ref{exceptional talagrand's observations}.

Identical to how Corollary~\ref{cor:GirthFiveList} follows from Theorem~\ref{thm:GirthFiveBipartiteMany}, we have that Theorem~\ref{thm:SmallCodegree} yields the following corollary which is a generalization of Theorems~\ref{thm:GirthFive} and~\ref{cor:GirthFiveList} to much weaker codegree conditions.

\begin{cor}\label{cor:SmallCodegreeColoring} 
For all integers $r,g \ge 2$ and real $\beta > 0$, there exist an integer $D_{\beta}\ge 0$ and real $\alpha > 0$ such that the following holds for all $D\ge D_{\beta}$: 
\vskip.05in
Let $G$ be an $r$-bounded (multi)-hypergraph with $\Delta(G)\le D$ and codegrees at most $D^{1-\beta}$. Let $H$ be a $g$-bounded configuration hypergraph of $G$, $\Delta_i(H) \le \alpha \cdot D^{i-1}\log D$ for all $2\le i\le g$ and $\Delta_{k,\ell}(H) \le D^{k-\ell-\beta}$ for all $2\le \ell< k\le g$. 
\vskip.05in
If the maximum $2$-codegree of $G$ with $H$ and the maximum common $2$-degree of $H$ are both at most $D^{1-\beta}$, then $$\chi(L(G)\cup H) \le \chi_{\ell}(L(G)\cup H) \le D(1+D^{-\alpha}),$$
and hence there exists an $H$-avoiding matching of $G$ of size at least $\frac{e(G)}{D}\cdot (1-D^{-\alpha})$.
\vskip.05in
Furthermore if $L$ is a list assignment of the edges of $G$ such that $|L(e)|\ge L_0 \ge D(1+D^{-\alpha})$ for each edge $e$ of $G$, then there exists a set of $L_0-D^{1-\alpha}$ (which is $\ge D$) pairwise disjoint $L$-colorings of $L(G)\cup H$.
\end{cor}

In the case where $G$ is uniform and mostly regular, we can also find many disjoint large matchings as follows.

\begin{cor}\label{cor:SmallCodegreeMatchings} 
Under the assumptions of Corollary~\ref{cor:SmallCodegreeColoring}, if $G$ is $r$-uniform and has minimum degree at least $D(1-D^{-\beta})$, then there exists a set of $D(1-D^{-\alpha/2}-D^{-\beta/2})$ disjoint $H$-avoiding matchings of $G$ each of size at least $\frac{v(G)}{r}\cdot (1-D^{-\alpha/2}-D^{-\beta/2})$.
\end{cor}
\begin{proof}
By Corollary~\ref{cor:SmallCodegreeColoring}, we have that $\chi(L(G)\cup H) \le D(1+D^{-\alpha})$.  We fix a coloring of $L(G)\cup H$ using exactly $D(1+D^{-\alpha})$ colors.  Since $G$ is $r$-uniform, each color class has size at most $\frac{v(G)}{r}$. Since $G$ has minimum degree at least $D(1-D^{-\beta})$, we have that $e(G) \ge \frac{D(1-D^{-\beta}) \cdot v(G)}{r}$. Thus the average number of vertices not covered by a color is at most $\frac{D^{-\beta}+D^{-\alpha}}{1+D^{-\alpha}}\cdot v(G)$. It follows that there exists at most $D^{1-\beta/2}+D^{1-\alpha/2}$ colors with at least $(D^{-\beta/2}+D^{-\alpha/2})\cdot v(G)$ uncovered vertices, and hence a set of at least $D(1-D^{-\alpha/2}-D^{-\beta/2})$ disjoint $H$-avoiding matchings of $G$, each of size at least $\frac{v(G)}{r}\cdot (1-D^{-\alpha/2}-D^{-\beta/2})$ as desired.
\end{proof}

Returning to designs, we now derive our approximate solution of Conjecture~\ref{conj:HighGirthAllUniformities}, Theorem~\ref{thm:HighGirthSteiner}, from Corollary~\ref{cor:SmallCodegreeColoring}.

\begin{proof}[Proof of Theorem~\ref{thm:HighGirthSteiner}]
Let $r' = \binom{q}{r}$. We let $G$ be the $r'$-uniform hypergraph where $V(G) :=\binom{[n]}{r}$ and $E(G) := \{ \binom{S}{r} : S\in \binom{[n]}{q}\}$. We let $H$ be the configuration hypergraph of $G$ where $E(H)$ consists of the $(i(q-r)+r,i)$-configurations for all $3\le i \le g$ that do not contain any $(i'(q-r)+r,i')$-configuration for $2\le i' < i$.

Let $D:= \binom{n-r}{q-r}$. Note $D\le n^{q-r}$ and $D = \Omega(n^{q-r})$. Note that $G$ is $D$-regular and has codegrees at most $\binom{n-r-1}{q-r-1} = O(n^{q-r-1}) = O\left(D^{1- \frac{1}{q-r}}\right)$. Moreover for each $i$ with $3\le i \le g$, we have $\Delta_i(H) \le O\left(n^{(i-1)(q-r)}\right) = O(D^{i-1})$. 

Meanwhile, for $2\le \ell < k \le g$, we have that
$$\Delta_{k,\ell}(H) \le O\left(n^{(k-\ell)(q-r)-1}\right) = O\left(D^{k-\ell - \frac{1}{q-r}}\right),$$
since $\ell$ edges of $G$ contained in an edge of $H$ of size $k$ span at least $\ell(q-r)+r+1$ vertices of $G$ by construction. Note that since $H$ has no edges of size two, we have that the maximum $2$-codegree of $G$ with $H$ and the maximum common $2$-degree of $H$ are $0 \le D^{1-\beta}$.

Thus if we fix $\beta < \frac{1}{q-r}$ (say $\beta = \frac{1}{2(q-r)}$), we find that $G$ and $H$ satisfy the conditions of Corollary~\ref{cor:SmallCodegreeColoring} provided that $D$ is large enough (or equivalently that $n$ is large enough).

Hence by Corollary~\ref{cor:SmallCodegreeColoring}, there exists an $H$-avoiding matching $M$ of $G$ of size at least $$\frac{e(G)}{D}\cdot (1-D^{-\alpha}) = \frac{\binom{n}{q}}{\binom{n-r}{q-r}} \cdot (1-D^{-\alpha}) = \frac{\binom{n}{r}}{\binom{q}{r}}\cdot \left(1-O\left(n^{-\alpha(q-r)}\right)\right).$$ 
But then $M$ is a partial $(n,q,r)$-Steiner system $S$ with $$|S|\ge \frac{\binom{n}{r}}{\binom{q}{r}}\cdot \left(1-O\left(n^{-\alpha(q-r)}\right)\right)$$ and no $(i(q-r)+r,i)$-configurations for all $2\le i \le g$ as desired.
\end{proof}

In fact using Corollary~\ref{cor:SmallCodegreeMatchings} in the proof above yields that for some $\gamma > 0$ and $n$ large enough, there exists a set of at least $\binom{n-r}{q-r} \cdot \left(1-n^{-\gamma}\right)$ disjoint partial $(n,q,r)$-Steiner systems each of size at least $\frac{\binom{n}{r}}{\binom{q}{r}}\cdot \left(1-n^{-\gamma}\right)$ with no $(i(q-r)+r,i)$-configurations for all $2\le i \le g$.

\subsection{Relation to Other Work}\label{ss:OtherWork}

Here we comment on the differences of our work to that of Glock, Joos, Kim, K\"{u}hn, and Lichev~\cite{GJKKL22}. Their proof uses a random greedy process instead of the nibble method and so is complementary to ours. However, this has the drawback of requiring that the degree is poly-logarithmic in the number of vertices; that is, it does not apply to sparse hypergraphs. Similarly, their proof does not seem to immediately carry over to the coloring/list coloring/bipartite $A$-perfect settings as our proof does (in particular, Theorem~\ref{thm:SmallCodegree} and Corollary~\ref{cor:SmallCodegreeColoring});  hence their results do not imply our new results in Sections~\ref{ss:ListSmall},~\ref{ss:Matchings},~\ref{ss:RainbowMatchings} and~\ref{ss:RainbowMatchingApplications}. 

On the other hand, their proof has the benefit of yielding rather exact counting results while our nibble is a bit wasteful; so even if our nibble could be optimally converted to a counting argument, it would not yield as precise of a count as their random greedy process without tightening up the nibble procedure. They also prove their theorem with the additional gadgetry of test systems (that is a system of various other variables that one desires to track and concentrate throughout the process). These quasi-random properties can be quite useful for future applications. That said, given the robustness of our Linear Talagrand's Inequality, we believe our proof also adapts to this setting and beyond; namely we could add to our nibble any set of variables that could be concentrated via Linear Talagrand's Inequality as long as the Local Lemma still applies. 

\subsection{Notation}
 Following convention, for example as in Glock, K\"{u}hn, Lo, and Osthus,~\cite{GKLO20}, equations containing $\pm$ should always be read from left to right.  For example, $\alpha = \beta \pm \gamma$ if $\beta-\gamma\leq\alpha\leq \beta+\gamma$ and $\beta_1 \pm \gamma_1 = \beta_2 \pm \gamma_2$ means that $\beta_1 - \gamma_1 \geq \beta_2 - \gamma_2$ and $\beta_1 + \gamma_1 \leq \beta_2 + \gamma_2$.  
Throughout, if $n$ is a positive integer, we let $[n]:=\left\{1,2,\ldots,n\right\}$. Similarly if $S$ is a set and $r$ is a nonnegative integer, we let $\binom{S}{r} := \{T\subseteq S: |T|=r\}$.

If $J$ is a hypergraph 
and $v\in V(J)$, we let $N_J'(v):= \{e\in E(J): v\in e\}$. Throughout this paper, $G$ will usually be an $r$-uniform (multi)-hypergraph and $H$ will be a $g$-bounded configuration hypergraph of $G$. 

For a hypergraph $J$, a \emph{coloring} of $J$ is an assignment of colors to the vertices of $J$ so that no edge contains only vertices of the same color. The \emph{chromatic number} of $J$, denoted $\chi(J)$ is the minimum number of colors needed to color the vertices of $J$. A \emph{list assignment} $L$ of the vertices of $J$ is simply a collection of lists $(L(v): v\in V(J))$. An \emph{$L$-coloring} of $J$ is a coloring where each vertex $v$ receives a color from its list $L(v)$. The \emph{list chromatic number} of $J$ is the minimum integer $k$ such that for every list assignment $L$ with $|L(v)|\ge k$ for all $v\in V(J)$, there exists an $L$-coloring of $J$.

Additionally, in this paper, all logarithms are natural. Also, we assume familiarity with the Lov\'asz Local Lemma~\cite{EL73} (see Alon and Spencer~\cite{AS16} for more discussion and applications).  

\subsection{Outline of Paper}

We first proceed in Section~\ref{s:Applications} to describe more applications of our main results to designs, list coloring hypergraphs of small codegree, edge colorings, perfect matchings in bipartite hypergraphs and rainbow matchings. 

In Section~\ref{s:Outline}, we outline the proofs of our main results at a high level.  We state our key iterative lemma, Lemma~\ref{lem:Iterative}, overview its proof and derive Theorem~\ref{thm:GirthFive} from it; similarly we state our key iterative lemma for the bipartite result, Lemma~\ref{lem:IterativeBip}, as well as a finishing lemma, Lemma~\ref{lem:Finish}, and derive Theorem~\ref{thm:GirthFiveBipartite} from them. We also prove Theorem~\ref{thm:SmallCodegree} assuming a random sparsification lemma, Lemma~\ref{lem:Sparsifying}.

In Section~\ref{s:Talagrand}, we state our linear version of Talagrand's Inequality. In Section~\ref{s:ProofIter}, we prove our key iterative lemma, Lemma~\ref{lem:Iterative}. In Section~\ref{s:Split}, we prove the auxiliary results for the bipartite result, namely a regularization lemma, Lemma~\ref{lem:reg}, as well as Lemmas~\ref{lem:IterativeBip} and~\ref{lem:Finish}. 

In Section~\ref{s:Generalizations}, we prove Lemma~\ref{lem:Sparsifying}. In Section~\ref{s:Rainbow}, we derive our other rainbow matching results, Theorems~\ref{thm:SparseGrinblat},~\ref{thm:CKAverage} and~\ref{thm:CKBipartite},  from Theorem~\ref{thm:KahnBipartiteRainbow}.

\section{History and Applications}\label{s:Applications}

Here we mention a number of applications of our main results while also discussing more related history.

\subsection{More Applications to Designs}\label{ss:Designs}

Our main results have more applications to designs. Here, we mention just a few notable examples. First, we can generalize Theorem~\ref{thm:HighGirthSteiner} to random graphs. To that end, let $H_r(n,p)$ consist of the random $r$-uniform hypergraph on $n$ vertices where each $r$-element set is selected independently with probability $p$. Note when $r=2$, $H_r(n,p)$ is just the Erd\H{o}s-R\'enyi random graph model $G(n,p)$. Applying Corollary~\ref{cor:SmallCodegreeColoring} to the appropriate auxiliary hypergraph now yields the following.
 
\begin{thm}\label{thm:HighGirthSteinerRandom}
For all integers $q > r \geq 2$, every integer $g\ge 2$ and constant $p\in (0,1]$, there exists $n_0$ and $\beta\in (0,1)$ such that for all $n\ge n_0$, $H_r(n,p)$ contains a partial $(n,q,r)$-Steiner system $S$ with $|S|\ge (1-n^{-\beta})\cdot p\cdot \binom{n}{r}/\binom{q}{r}$ and no $(i(q-r)+r,i)$-configurations for all $2\le i \le g$.
\end{thm}

Let $G$ and $F$ be hypergraphs. An \emph{$F$ packing} of $G$ is a set of edge-disjoint copies of $F$ in $G$. An \emph{$F$ decomposition} of $G$ is a partition of $E(G)$ into copies of $F$.  We let $K_{n*q}^r$ denote the \emph{complete $q$-partite $r$-graph with parts of size $n$}; this is the $q$-partite $r$-graph with parts of size $n$ and maximum number of edges.  When $r=2$, a $K_q$ decomposition of $K_{q*n}$ is referred to as a high-dimensional permutation (where the dimension is $q-1$).  As a special case, when $r=2$ and $q=3$ this is referred to as a \emph{Latin square}.  

In a recent follow-up paper, Kwan, Sah, Sawhney, and Simkin~\cite{KSSS22B} proved that high girth Latin squares exist. Our main result vastly generalizes this in the approximate sense, for example it proves the existence of high girth approximate high dimensional permutations.  Applying Corollary~\ref{cor:SmallCodegreeColoring} to the appropriate auxiliary hypergraph now yields the existence of approximate high girth $K_q^r$ decompositions of $K_{q*n}^r$ as follows.

\begin{thm}\label{thm:HighGirthLatinSquares}
For all $g > q > r \geq 2$, there exist $n_0$ and $\beta \in (0,1)$ such that for all $n\ge n_0$, there exists a $K_q^r$-packing $S$ of $K_{q*n}^r$ with $|S|\ge (1 - n^{-\beta})n^r/\binom{q}{r}$ and no $(i(q-r)+r,i)$-configurations for all $2\le i \le g$.
\end{thm}

We omit the proofs of Theorems~\ref{thm:HighGirthSteinerRandom} and~\ref{thm:HighGirthLatinSquares} as they follow similarly from Corollary~\ref{cor:SmallCodegreeColoring} as Theorem~\ref{thm:HighGirthSteiner} did. As in the case of $(n,q,r)$-Steiner systems, Theorem~\ref{thm:HighGirthLatinSquares} also generalizes to the random hypergraph setting, we omit its obvious statement. Also, in all cases, Corollary~\ref{cor:SmallCodegreeMatchings} implies the existence of many disjoint such objects covering nearly all the $q$-element sets. 

\subsection{List Coloring Hypergraphs of Small Codegree}\label{ss:ListSmall}

Here is Cooper and Mubayi's precise codegree theorem.

\begin{thm}[Cooper and Mubayi~\cite{CM16}]\label{thm:CopperMubayi}
Fix $k\ge 3$. Let $H$ be a $k$-uniform hypergraph with maximum degree $\Delta$. If for all $2\le \ell < k$, we have
%$$\Delta_{k,\ell}(H) \le\frac{\Delta^{\frac{k-\ell}{k-1}}}{f},$$
$\Delta_{k,\ell}(H) \le {\Delta^{\frac{k-\ell}{k-1}}}/{f},$
then 
%$$\chi(H) = O\left(\left(\frac{\Delta}{\logf}\right)^{\frac{1}{k-1}}\right).$$
$\chi(H) = O\left(\left({\Delta}/{\log f}\right)^{\frac{1}{k-1}}\right).$
\end{thm}

Li and the second author~\cite{LP22} proved a non-uniform generalization of Theorem~\ref{thm:CopperMubayi}. Both theorems however only hold for ordinary coloring.

Our Corollary~\ref{cor:SmallCodegreeColoring} implies a generalization of Theorem~\ref{thm:CopperMubayi} to non-uniform hypergraphs as well as to list coloring (albeit with stronger codegree assumptions), namely by taking $G$ to be a matching whose edges correspond to the vertices of $H$. Here is that statement as follows.

\begin{thm}\label{thm:ListCooperMubayi}
Let $g\ge 3$ be an integer and $\beta > 0$ be real. There exists $D_{\beta}$ such that the following holds for all $D\ge D_{\beta}$:

Let $H$ be a $g$-bounded hypergraph such that every edge of $H$ has size at least two and $\Delta_i(H) \le O(D^{i-1}\log D)$ for all $2\le i \le g$. If the maximum common $2$-degree of $H$ is at most $D^{1-\beta}$ and for all $2\le \ell < k \le g$, we have that
$\Delta_{k,\ell}(H) \le D^{k-\ell-\beta}$,
then 
$$\chi(H) \le \chi_{\ell}(H) \le D.$$
\end{thm}

\subsection{Matchings and Edge Colorings of Hypergraphs}\label{ss:Matchings}

The \emph{nibble} or \emph{semi-random method}, in which a combinatorial object is constructed via a series of small random steps, has had tremendous influence on combinatorics in recent decades. The results by R\"odl~\cite{R85}, Frankl and R\"odl~\cite{FR85}, Pippenger (unpublished), Pippenger and Spencer~\cite{PS89}, and Kahn~\cite{K96} were all proved using variations of the nibble method. We refer the reader to the recent survey by Kang, Kelly, K\"{u}hn, Methuku and Osthus~\cite{KKKAD21} for more information on the nibble method, its influence and uses.

Finding perfect or almost perfect matchings in hypergraphs has many applications. For example, R\"{o}dl's result on approximate designs follows from Theorem~\ref{thm:Pippenger} by applying it to an auxiliary hypergraph whose vertices are the edges of $K_n^r$ and whose edges are the copies of $K_q^r$ in $K_n^r$. For more background on finding matchings in hypergraphs and their applications, we refer the reader to the survey by Keevash~\cite{K18}. 

Indeed, Corollary~\ref{thm:SmallCodegree} implies (by setting $H$ to be empty) the following generalization of Kahn's theorem (albeit with stronger codegree assumptions) for finding many disjoint list colorings.

\begin{thm}\label{thm:KahnKahn}
For all integers $r \ge 2$ and real $\beta > 0$, there exist an integer $D_{r,\beta}\ge 0$ and real $\alpha > 0$ such that the following holds for all $D\ge D_{r,\beta}$: 

If $G$ is an $r$-bounded hypergraph of maximum degree at most $D$ and maximum codegree at most $D^{1-\beta}$ and $L$ is a list assignment to the edges of $G$ such that $|L(e)|\ge D(1+D^{-\alpha})$ for each edge $e$ of $G$, then $G$ has $D$ pairwise disjoint $L$-colorings, that is there exist $L$-colorings $\phi_1,\ldots, \phi_D$ such that $\phi_i(e)\ne \phi_j(e)$ for every $i\ne j \in [D]$ and $e\in E(G)$.
\end{thm}

Similarly, setting $H$ to be empty in Theorem~\ref{thm:SmallCodegree} yields a generalization of Theorem~\ref{thm:KahnKahn} about finding $A$-perfect matchings in bipartite hypergraphs. Namely, the theorem guarantees the existence of an $A$-perfect matching provided that the vertices of $A$ have slightly larger degree than the vertices in $B$. As far as we know, this theorem does not appear in the literature to date.

\begin{thm}\label{thm:KahnBipartite}
For all integers $r \ge 2$ and real $\beta > 0$, there exist an integer $D_{r,\beta}\ge 0$ and real $\alpha > 0$ such that the following holds for all $D\ge D_{r,\beta}$:
\vskip.1in
\noindent Let $G=(A,B)$ be a bipartite $r$-bounded (multi)-hypergraph satisfying
\vskip.1in
\begin{enumerate}
\item[(G1)] every vertex in $A$ has degree at least $(1+D^{-\alpha})D$ and every vertex in $B$ has degree at most $D$, and
\item[(G2)] every pair of vertices of $G$ has codegree at most $D^{1-\beta}$,
\end{enumerate}
then there exists an $A$-perfect matching of $G$ and indeed even a set of $D$ disjoint $A$-perfect matchings of $G$.
\end{thm}

Note that Theorem~\ref{thm:KahnKahn} is in fact a corollary of Theorem~\ref{thm:KahnBipartite} in the same manner that Corollary~\ref{cor:SmallCodegreeColoring} is a corollary of Theorem~\ref{thm:SmallCodegree}. Moreover, Theorem~\ref{thm:KahnBipartite} has applications to rainbow matchings which we discuss in the next two subsections.

We also note that our proof of Theorem~\ref{thm:SmallCodegree} and hence of Theorem~\ref{thm:KahnBipartite} uses that the gap between the degrees of $A$ and the degrees of $B$ widens as the nibble process iterates. Kahn's proof instead uses his trick of ``reserve" colors; namely before the nibble process, every vertex reserves each color independently with some small probability $p$ and then the colors of an edge that are reserved at both ends are used at the end to complete the coloring. 

We note then that Theorem~\ref{thm:KahnBipartite} admits an alternate proof wherein each vertex of $B$ is `reserved' independently with probability $p$, the reserved vertices are deleted, and edges whose $B$ vertices are all reserved are used to complete the matching to an $A$-perfect matching at the end of the nibble process. Thus Kahn's ``reserve colors" are a special case of this more general ``reserve trick" when working with the specific list-coloring based bipartite hypergraphs from the proof of Corollary~\ref{cor:GirthFiveList}. We opted not to use the ``reserve trick" for the proof of Theorem~\ref{thm:SmallCodegree} since the complications created by configuration hypergraphs seemed needlessly difficult.

\subsection{Rainbow Matchings in Hypergraphs}\label{ss:RainbowMatchings}

Rainbow matchings in graphs and hypergraphs have also received much attention from researchers over the decades with well over a hundred papers on the subject. Recall that a matching $M$ of a (not necessarily properly) edge colored hypergraph $G$ is \emph{rainbow} if every edge of $M$ is colored differently and a rainbow matching is \emph{full} if every color of the coloring appears on some edge of $M$.

A typical example of a rainbow matching conjecture is the Aharoni-Berger Conjecture as follows; see Aharoni-Charbit-Howard~\cite{ACH15} for more discussion on the history.

\begin{conj}[Aharoni and Berger (unpublished)]\label{conj:AB}
If $G$ is a bipartite multigraph properly edge colored with $q$ colors where every color appears at least $q+1$ times, then there exists a full rainbow matching.
\end{conj}

Here is a version for non-bipartite graphs (see~\cite{CPS21} and~\cite{GRWW21}).

\begin{conj}\label{conj:GRWW}
If $G$ is a multigraph properly edge colored with $q$ colors where every color appears at least $q+2$ times, then there exists a full rainbow matching.
\end{conj}

There are three natural ways to weaken these conjectures. The {\bf first} is to {\bf find a slightly smaller rainbow matching}  (i.e.~$(1-o(1))q$ in the conjectures above), what Munh\'{a} Correia, Pokrovskiy, and Sudakov~\cite{CPS21} called the \emph{`weak asymptotic'}. The {\bf second}, which implies the first, is to instead {\bf assume that each color appears slightly more times}, (i.e.~$(1+o(1))q$ times above), what Munh\'{a} Correia, Pokrovskiy, and Sudakov~\cite{CPS21} called the \emph{`strong asymptotic'}. There is a {\bf third} direction, also implying the first, which is to instead {\bf assume the number of colors is slightly more}, (i.e.~$(1+o(1))q$ colors above). Since many of the conjectures on rainbow matchings are notoriously difficult, these three directions have garnered much attention over the years.

For Conjecture~\ref{conj:AB}, the weak asymptotic was proved by Bar\'at, Gy\'arf\'as and S\'ark\"ozy~\cite{BGS17}; Pokrovskiy~\cite{P18} proved the strong asymptotic and subsequently Munh\'{a} Correia, Pokrovskiy, and Sudakov~\cite{CPS21} gave an incredibly short proof via their sampling trick which they used to reduce the problem to the weak asymptotic. For Conjecture~\ref{conj:GRWW}, the weak and strong asymptotic were both proved by Munh\'{a} Correia, Pokrovskiy, and Sudakov~\cite{CPS21}.

However, in many rainbow conjectures and their various weakenings (such as Conjecture~\ref{conj:AB} and~\ref{conj:GRWW}), the number of colors and the desired size of a rainbow matching are assumed to be on the order of the number of times a color appears; we refer to this as the \emph{`dense setting'}. 

A natural generalization then is to study all these conjectures and weakenings in what we call the \emph{`sparse setting'}, where the number of colors can be much larger than the number of times a color appears and rather the number of times a color appears is related to the degree of the graph. So for example, here are sparse setting versions of Conjectures~\ref{conj:AB} and~\ref{conj:GRWW}.

\begin{conj}\label{conj:SparseAB}
If $G$ is a bipartite multigraph properly edge colored where every color appears at least $\Delta(G)+1$ times, then there exists a full rainbow matching.
\end{conj}

\begin{conj}\label{conj:SparseGRWW}
If $G$ is a multigraph properly edge colored where every color appears at least $\Delta(G)+2$ times, then there exists a full rainbow matching.
\end{conj}

These sparse setting variants also admit the three natural weakenings; {\bf first} to require {\bf $(1-o(1))$ the desired size of a matching}; {\bf second} to assume that {\bf each color appears at least $(1+o(1))$ the number of conjectured times}; and {\bf third}, that {\bf a matching of size $(1-o(1))$ exists even if the colors only appear $(1-o(1))$ the conjectured number of times}.

There are also other avenues of generalization such as: analogous results in hypergraphs, allowing not necessarily proper edge colorings, or finding many disjoint full rainbow matchings. On the other hand, another natural weakening is to assume the graph is simple or of bounded multiplicity (or bounded codegrees in the hypergraph case).

As noted by various researchers (e.g., Aharoni and Berger~\cite{AB09} as well as Gao, Ramadurai, Wanless, and Wormald~\cite{GRWW21}), matchings of bipartite hypergraphs correspond to rainbow matching of hypergraphs as follows. 

For a hypergraph $G$ whose edges are colored by a not necessarily proper coloring $\phi$, define a new hypergraph, the \emph{rainbow hypergraph of $G$ with respect to $\phi$} as Rainbow$(G,\phi):=(A,B)$ where $A=\bigcup_{e\in E(G)} \phi(e)$ is the set of colors and $B=V(G)$ is the set of vertices and then extend every edge $e$ of $G$ to include its color $\phi(e)$, that is $E({\rm Rainbow}(G,\phi)) = \{ e\cup \phi(e): e\in E(G)\}$. Then a rainbow matching of $G$ is precisely a matching in ${\rm Rainbow}(G,\phi)$. Vice versa, declaring the vertices of $A$ in a bipartite hypergraph to be colors of the projected edges in $B$, the reverse correspondence also holds. 

Thus Theorem~\ref{thm:KahnBipartite} has the following equivalent formulation in terms of rainbow matchings. 
\begin{thm}\label{thm:KahnBipartiteRainbow}
For all integers $r \ge 2$ and real $\beta > 0$, there exist an integer $D_{\beta}\ge 0$ and real $\alpha > 0$ such that the following holds for all $D\ge D_{\beta}$: 
\vskip.1in
\noindent Let $G$ be an $r$-bounded (multi)-hypergraph that is (not necessarily properly) edge colored satisfying
\vskip.1in
\begin{enumerate}
\item[(G1)] $\Delta(G)\le D$ and every color appears at least $(1+D^{-\alpha})D$ times, and
\item[(G2)] every pair of vertices of $G$ has codegree at most $D^{1-\beta}$ and every color appears at most $D^{1-\beta}$ times around a vertex,
\end{enumerate}
then there exists a full rainbow matching of $G$ and indeed even a set of $D$ disjoint full rainbow matchings of $G$.
\end{thm}

As one might imagine, Theorem~\ref{thm:KahnBipartiteRainbow} has applications to rainbow matching results, especially in the sparse setting where the number of times a color appears is much smaller than the number of colors. We detail some of these applications in the next subsection.

Already though, Theorem~\ref{thm:KahnBipartiteRainbow} generalizes the main results of Gao, Ramadurai, Wanless, and Wormald~\cite{GRWW21} (their Theorems 1.4 and 1.7) to allowing any number of colors (the sparse setting) while also allowing larger multiplicity; it thus almost proves a conjecture of Aharoni and Berger~\cite{AB09} (their Conjecture 2.5) for bipartite $3$-uniform hypergraphs but with an additional multiplicity assumption (the conjecture is actually not true without some additional assumption) and also generalizes it to all uniformities. Similarly, Theorem~\ref{thm:KahnBipartiteRainbow} generalizes two of the main results of Chakraborti and Loh~\cite{CL20} (their Theorems 1.10 and 1.12) to the sparse setting, hypergraphs and with better multiplicity assumptions. 

We also note that the ``reserve trick" mentioned in the previous subsection has an equivalent formulation for rainbow matchings in hypergraphs, namely reserving every vertex of $B$ in ${\rm Rainbow}(G,\phi)=(A,B)$ independently with probability $p$ is equivalent to reserving every vertex of $G$ independently with probability $p$. For graphs, this is precisely the ``sampling trick" of Munh\'{a} Correia, Pokrovskiy, and Sudakov~\cite{CPS21} which they recently introduced and used to great effect to provide short proofs of various rainbow matching results. Hence Kahn's ``reserve colors" and their ``sampling trick" can both be viewed as shades of the same more general ``reserve trick" in bipartite hypergraphs where one reserves vertices of $B$ independently with some small probability $p$.

\subsection{More Applications to Rainbow Matchings}\label{ss:RainbowMatchingApplications}

Here is another well-studied rainbow matching conjecture.

\begin{conj}[Alspach~\cite{A88}]\label{conj:Alspach}
If $G$ is a simple $2d$-regular graph that is edge colored such that each color class is a $2$-factor, then $G$ has a full rainbow matching. 
\end{conj}

Munh\'{a} Correia, Pokrovskiy, and Sudakov~\cite{CPS21} proved the strong asymptotic version of Conjecture~\ref{conj:Alspach}. We note though that our Theorem~\ref{thm:KahnBipartiteRainbow} immediately implies the strong asymptotic version of Conjecture~\ref{conj:Alspach} in the sparse setting.

Here is another rainbow matching conjecture, originally motivated by equivalence classes in algebras, that has received much attention.

\begin{conj}[Grinblat~\cite{G02}]\label{conj:Grinblat}
If $G$ is a multigraph that is (not necessarily properly) edge colored with $n$ colors where each color class is the disjoint union of non-trivial cliques and spans at least $3n-2$ vertices, then $G$ has a rainbow matching of size $n$.
\end{conj}

The strong asymptotic version of Conjecture~\ref{conj:Grinblat} was proved by Clemens, Ehrenm\"{u}ller, and Pokrovskiy~\cite{CEP17} (see Munh\'{a} Correia, Pokrovskiy, and Sudakov~\cite{CPS21} for a short proof). Conjecture~\ref{conj:Grinblat} was then fully proved by Munh\'{a} Correia and Sudakov~\cite{CS21}. 

Better bounds are known though when assuming the graph has bounded multiplicity. For example, Munh\'{a} Correia and Yepremyan~\cite{CY20} proved the version where if $G$ has multiplicity at most $\sqrt{n}/\log^2 n$ and each color class spans at least $2n+o(n)$ vertices, then $G$ has full rainbow matching. Munh\'{a} Correia, Pokrovskiy, and Sudakov~\cite{CPS21} improved this further by showing that for multiplicity $m$, each color class spanning $2n+2m+O\left((n/\log n)^{1/4}\right)$ suffices. 

In this paper, we use a simple deletion trick combined with Theorem~\ref{thm:KahnBipartiteRainbow} to prove a bounded multiplicity strong asymptotic version of Conjecture~\ref{conj:Grinblat} for hypergraphs in the sparse setting that finds many disjoint full rainbow matchings as follows. See Section~\ref{s:Rainbow} for its proof.

\begin{thm}\label{thm:SparseGrinblat}
For all integers $r\ge 2$ and real numbers $\beta > 0$, there exist an integer $D_{\beta}$ and real $\alpha > 0$ such that the following holds for all $D\ge D_{\beta}$:

If $G$ is an  $r$-uniform (multi)-hypergraph with codegrees at most $D^{1-\beta}$ that is (not necessarily properly) edge colored where each color class is the disjoint union of non-trivial cliques and spans at least $rD(1+D^{-\alpha})$ vertices and each vertex is incident with at most $D$ colors, then $G$ has a full rainbow matching and indeed even a set of $D$ disjoint full rainbow matchings.
\end{thm}

Finally, we turn to the third weakening direction where the number of colors is allowed to be slightly larger than as conjectured. Aharoni and Berger~\cite{AB09} and in a slightly weaker form Drisko~\cite{D98} proved that if $G$ is a bipartite multigraph that is properly edge colored with $2n-1$ colors each appearing exactly $n$ times, then $G$ has a rainbow matching of size $n$. Bar\'at, Gy\'arf\'as, and S\'ark\"ozy~\cite{BGS17} conjectured this should hold even for non-bipartite multigraphs when $n$ is odd (and with $2n$ colors when $n$ is even). Aharoni, Berger, Chudnovsky, Howard, and Seymour~\cite{ABCSP2019} proved $3n-2$ colors suffice, and this was improved to $3n-3$ by Aharoni, Briggs, Kim, and Kim~\cite{ABKK19}. 

What about when $G$ is simple though? Chakraborti and Loh~\cite{CL20} proved the following.

\begin{thm}[Chakraborti and Loh~\cite{CL20}]\label{thm:CK}
For every real $\varepsilon > 0$, there exists $q_0$ such that for all integers $q\ge q_0$ the following holds: If $G$ is a simple graph that is properly edge colored with $(2+\varepsilon)q$ colors such that each color appears exactly $q$ times, then $G$ has a rainbow matching of size $q$. 
\end{thm}

We can derive Theorem~\ref{thm:CK} from Theorem~\ref{thm:KahnBipartiteRainbow} rather straightforwardly; indeed we prove the following stronger statement. See Section~\ref{s:Rainbow} for its proof.

\begin{thm}\label{thm:CKAverage}
For every real $\varepsilon > 0$, there exists $q_0$ such that for all integers $q\ge q_0$ the following holds: Let $r\le q$ be a positive integer. If $G$ is a simple graph with at least $(2+\varepsilon)qr$ edges that is properly edge colored with $(2+\varepsilon)q$ colors such that no color appears more than $q$ times, then $G$ has a rainbow matching of size $r$. 
\end{thm}

Chakraborti and Loh also conjectured the following.

\begin{conj}[Chakraborti and Loh~\cite{CL20}]
For every real $\varepsilon > 0$, there exists $q_0$ such that for all integers $q\ge q_0$ the following holds: If $G$ is a simple bipartite graph that is properly edge colored with $(1+\varepsilon)q$ colors such that each color appears exactly $q$ times, then $G$ has a rainbow matching of size $q$. 
\end{conj}

Once again we can use Theorem~\ref{thm:KahnBipartiteRainbow} to prove their conjecture with the following stronger statement as follows. See Section~\ref{s:Rainbow} for its proof.

\begin{thm}\label{thm:CKBipartite}
For every real $\varepsilon > 0$, there exists $q_0$ such that for all integers $q\ge q_0$ the following holds: Let $r\le q$ be a positive integer. If $G$ is a simple bipartite graph with at least $(1+\varepsilon)qr$ edges that is properly edge colored with $(1+\varepsilon)q$ colors such that no color appears more than $q$ times, then $G$ has a rainbow matching of size $r$. 
\end{thm}

Lastly, we mention one more application. Let $G$ be a simple graph of minimum degree $d$. In 2011, Wang~\cite{W11} asked the natural question of how large does $v(G)$ need to be as a function of $d$ to guarantee that a properly
edge-colored graph $G$ has a rainbow matching of size at least $d$? Independently, Gy\'{a}rf\'{a}s and S\'{a}rk\"{o}zy~\cite{GS14} as well as Lo and Tan~\cite{LT14} showed that having $4d-3$ vertices suffices.  The best known result for general graphs is by Lo~\cite{L15} who showed that having $3.5\cdot d+2$ vertices suffices. As for when $G$ is a bipartite graph, Lo~\cite{L15} showed that having $(3+o(1))d$ vertices suffices. We note that at least $2d$ vertices are needed to guarantee a (rainbow) matching of size $d$. As a corollary of Theorem~\ref{thm:CKBipartite}, we achieve this asymptotically in the bipartite case as follows. See Section~\ref{s:Rainbow} for its proof.

\begin{thm}\label{thm:MCCBipartite}
For every real $\varepsilon > 0$, there exists $d_0$ such that for all integers $d\ge d_0$ the following holds: If $G$ is a simple bipartite graph with minimum degree $d$ and at least $(2+\varepsilon)d$ vertices that is properly edge colored, then $G$ has a rainbow matching of size $d$. 
\end{thm}

\section{Outline of Proof}\label{s:Outline}

In this section, we overview the proofs of Theorems~\ref{thm:GirthFive},~\ref{thm:GirthFiveBipartite} and~\ref{thm:SmallCodegree}.

To prove Theorem~\ref{thm:GirthFive}, we will prove the following iterative lemma, but first a definition.

\begin{definition}
Let $H$ be a hypergraph and let $D\geq 1$ be a real number. We define the \emph{$D$-weighted degree} of a vertex $v\in V(H)$, denoted 
$$w_{D}(H,v):=\sum_{i\ge 2} d_{H,i}(v) \cdot \frac{i-1}{D^{i-1}}.$$ 
We define the \emph{maximum $D$-weighted degree} of $H$, denoted 
$$w_D(H) := \max_{v\in V(H)} w_D(H,v).$$
\end{definition}

Here then is our key iterative lemma.

\begin{lem}\label{lem:Iterative}
For all integers $r,g \ge 2$, there exists an integer $D_{r,g}\ge 0$ such that the following holds for all $D\ge D_{r,g}$ and for all $\varepsilon \ge D^{-1/6r}$, $\gamma \le \frac{1}{24rg}$, $w \ge D^{-1/6}$ such that $\gamma \ge rg(5r+2w+w^2)\varepsilon$.
\vskip.1in
Let $G$ be an $r$-uniform (multi)-hypergraph on $n$ vertices and let $H$ be a $g$-bounded configuration hypergraph of $G$.
\vskip.1in
\noindent If $G$ satisfies
\vskip.1in
\begin{enumerate}
\item[(G1)] $G$ is $D(1\pm \gamma)$-regular,
\item[(G2)] every pair of vertices of $G$ has codegree at most $4\log^2 D$,
\end{enumerate}
and $H$ satisfies
\vskip.1in
\begin{enumerate}
\item[(H1)] $w_D(H) \le w$,
\item[(H2)] $H$ has girth at least five,
\end{enumerate}
and $G$ and $H$ satisfy
\begin{enumerate}
\item[(G+H)] the maximum codegree of $G$ with $H$ is at most $2^{2(g-1)}\cdot \log^{2(g-1)} D$,
\end{enumerate}
then there exists an $H$-avoiding matching $E'$ of $G$ and a subhypergraph $G'$ of $G$ such that $V(G')= V(G)-\bigcup_{e\in E'} e$ and $E(G')\subseteq E(G[V(G')])$, and a configuration hypergraph $H'$ of $G'$ of girth at least five where $E(H'):= \{S\setminus E': S\in E(H),~S\setminus E' \subseteq E(G')\}$ that satisfy the all of the following, where we let
$$D':= De^{-\varepsilon(r-1 + w)}:$$

\begin{itemize}
\item[(1)] if $M$ is an $H'$-avoiding matching of $G'$, then $M\cup E'$ is an $H$-avoiding matching of $G$,
\item[(2)] $v(G\setminus V(G')) \ge \frac{\varepsilon}{2} \cdot v(G)$,
\item[(3)] $G'$ is $D'(1\pm \gamma(1+6r\varepsilon))$-regular,
\item[(4)]$w_{D'}(H') \le w \left(1-\frac{\varepsilon}{2}\right)$.
\end{itemize}
\end{lem}

Assuming Lemma~\ref{lem:Iterative}, we are now prepared to prove Theorem~\ref{thm:GirthFive} as follows.

\begin{proof}[Proof of Theorem~\ref{thm:GirthFive}]
We choose $D_{\beta}$ to be equal to $D_{r,g}^{1/(1-\beta)}$ where $D_{r,g}$ is as in Lemma~\ref{lem:Iterative}. Recall $D\ge D_{\beta}$ by assumption.

Let $\varepsilon := D_0^{-(7/6)\beta}$ and let $T:= \frac{\beta\log D_0}{8r\varepsilon}$. Let $D_0:= D$, $w_0 := \frac{\beta\log D_0}{4}$, and $\gamma_0:= D_0^{-\beta}$. 

As $\Delta_i(H) \le \frac{\beta}{4g^2} \cdot D^{i-1}\log D$ for all $i\in \{2,\ldots, g\}$ by assumption, we find that $w_D(H) \le \frac{\beta \log D_0}{4} = w_0$. Note that $w_0\ge 1 \ge D_0^{-1/6}$ since $D$ is large enough.

For each integer $i$ with $1\le i\le T+1$, define the following parameters:
$$D_i := D_0\cdot e^{-\varepsilon\left(i(r-1)+w_0\cdot \sum_{j=0}^{i-1} \left(1-\frac{\varepsilon}{2}\right)^j\right)},$$
$$\gamma_i := \gamma_0\cdot (1+6r\varepsilon)^i,$$
and
$$w_i := w_0 \cdot \left(1-\frac{\varepsilon}{2}\right)^i,$$

\begin{claim}\label{claim:iterative}
For each integer $0\le i \le T+1$, there exists an $H$-avoiding matching $E_i$ of $G$, a subhypergraph $G_i$ of $G$ with $V(G_i)=V(G)\setminus V(E_i)$ and a configuration hypergraph $H_i$ of $G_i$ of girth at least five where $E(H_i):= \{S\setminus E_i: S\in E(H),~S\setminus E_i \subseteq E(G_i)\}$ that satisfy all of the following:
\begin{enumerate}
\item[(1)] if $M$ is an $H_i$-avoiding matching of $G_i$, then $M\cup E_i$ is an $H$-avoiding matching of $G$, 
\item[(2)] $v(G_i)\le v(G) \cdot e^{-i\varepsilon/2}$,
\item[(3)] $G_i$ is $D_i(1\pm \gamma_i)$-regular,
\item[(4)] $w_{D_i}(H_i)\le w_i$.
\end{enumerate}
\end{claim}
\begin{proof}
We proceed by induction on $i$. When $i=0$, then $G_0:=G$ and $H_0:= H$ satisfy the conclusion as desired. So we may assume that $i \ge 1$. By induction, there exists $G_{i-1}$ and $H_{i-1}$ satisfying (1)-(4).

Note that $D_i = D_{i-1}\cdot e^{-\varepsilon(r-1+w_{i-1})}$, $\gamma_i = \gamma_{i-1}(1+6r\varepsilon)$ and $w_i = w_{i-1} \cdot \left(1-\frac{\varepsilon}{2}\right)$. Also note that 
$$D_{i-1} \ge D_0 \cdot e^{-(i-1)(r-1)\varepsilon-2w_0} \ge D_0 \cdot e^{-Tr\varepsilon-2w_0} \ge  D_0^{1-(5/8)\beta}$$
since $Tr\varepsilon = \frac{\beta \log D_0}{8}$ and $w_0 = \frac{\beta \log D_0}{4}$.
Hence 
$$\varepsilon \ge D_0^{-(7/6)\beta} \ge D_{i-1}^{ \frac{-(7/6)\beta}{1-(5/8)\beta}} \ge D_{i-1}^{-1/6r}$$ 
since $\frac{(7/6)\beta}{1-(5/8)\beta} \le \frac{1}{6r}$ as $\beta \le \frac{1}{8r}$.

On the other hand, 
$$\gamma_{i-1} \le \gamma_0\cdot e^{6r\varepsilon (i-1)} \le D_0^{-\beta} \cdot e^{6r\varepsilon T} \le D_0^{-\beta/4} \le \frac{1}{24rg}$$ 
since $D_0$ is large enough. Since $w_{i-1}\le w_0$, we also have that 
$$\gamma_{i-1}\ge \gamma_0 \ge rg(5r+2w_0+w_0^2)\cdot \varepsilon \ge rg(5r+2w_{i-1}+w_{i-1}^2)\cdot \varepsilon,$$
where we used that $w_0 = \frac{\beta \log D_0}{4}$, $\gamma_0 = D_0^{\beta/6}\cdot \varepsilon$ and that $D=D_0$ is large enough.
Moreover, using that $1-x \ge e^{-2x}$ for all $x\in [0,1/2]$, we find that
$$w_{i-1} = w_0 \cdot \left(1-\frac{\varepsilon}{2}\right)^{i-1} \ge e^{-T\varepsilon} = D_0^{-\frac{\beta}{8r}} \ge D_{i-1}^{-1/6},$$
where we used that $w_0\ge 1$, $i-1\le T$ and $D_{i-1} \ge D_0^{1-\beta} \ge D_0^{\frac{6\beta}{8r}}$ as $\beta \le \frac{1}{2}$ by assumption.

The codegree of pairs of vertices of $G_{i-1}$ is at most the codegree of pairs of vertices of $G_0$ which is at most $\log^2 D_0 \le 4\log^2 D_{i-1}$ since $D_{i-1} \ge D_0^{1/2}$.  Similarly the codegree of $G_i$ with $H_i$ is at most the codegree of $G_0$ with $H_0$ which is at most $\log^{2(g-1)} D_0 \le 2^{2(g-1)}\cdot \log^{2(g-1)} D_{i-1}$ since $D_{i-1} \ge D_0^{1/2}$. 

 Thus we have that $G_{i-1}, H_{i-1}, D_{i-1}, \gamma_{i-1}, w_{i-1}$ and $\varepsilon$ satisfy the the conditions of Lemma~\ref{lem:Iterative}. Hence by Lemma~\ref{lem:Iterative} applied to $G_{i-1}$ and $H_{i-1}$, there exists $G_i$, $H_i$ and $E_i'$ satisfying (1)-(4) of Lemma~\ref{lem:Iterative} and hence they also satisfy (1)-(4) of this claim as desired where we let $E_i = E_i'\cup E_{i-1}$.
\end{proof}

Let $T^* = \lfloor T \rfloor + 1$. By Claim~\ref{claim:iterative}, there exists $G_{T^*}, H_{T^*}$ satisfying (1)-(4) of Claim~\ref{claim:iterative}.  By Claim~\ref{claim:iterative}(2), we have that 
$$v(G_{T^*}) \le v(G) \cdot e^{-T\varepsilon/2} \le v(G) \cdot D^{-\frac{\beta}{16r}}.$$
Hence $E_{T^*}$ is an $H$-avoiding matching of $G$ with size at least
$$\frac{v(G)-v(G_{T^*})}{r} \ge \frac{v(G)}{r} \cdot (1-D^{-\frac{\beta}{16r}}),$$
as desired.
\end{proof}

In the next subsection, we provide a rough overview of the proof of Lemma~\ref{lem:Iterative}; the detailed proof may be found in Section~\ref{s:ProofIter}.

\subsection{Rough Overview of Proof of Lemma~\ref{lem:Iterative}}\label{ss:IterOverview}

Our random procedure chooses each edge of $G$ independently with probability $\varepsilon/D$. The matching $E'$ then consists of chosen edges which are not incident with any other chosen edges nor in a configuration (i.e.~an edge of $H$) where all edges are chosen. Note this means $E'$ is an $H$-avoiding matching of $G$ by construction.

We let the hypergraph $G'$ consist of the edges on the unmatched vertices, except we delete any edge in a nearly completed configuration; however, since $w_D(H,e)$ may not equal $w$, we also perform an equalizing coinflip to ensure each edge is lost with probability exactly $\varepsilon w$ due to this configuration loss. See Section~\ref{ss:Procedure} for the formal description of the random procedure. 

What follows is a rough guide to the behavior of the variables in the proof. Ignoring lower order terms from codegree, concentration, etc., we have $\Prob{v\in V(G')}\approx e^{-\varepsilon}$ and hence $\Prob{V(e)\subseteq V(G')~|~ v\in V(G')} \approx e^{-(r-1)\varepsilon}$. Since whether an edge is in a completed configuration is nearly independent of the previous event, we find that $\Prob{e\in E(G')~|~v\in V(G')} \approx e^{-(r-1+w)\varepsilon}$ and hence $\Expect{d_{G'}(v)} \approx D'$.

Similarly, for an edge $f$ of $G$, $\Prob{f\in E(G')}\approx e^{-(r+w)\varepsilon}$. Hence, for an edge $S$ of $H$ containing an edge $e$ of $G$, 
$$\Prob{S\in E(H')~|~e\in E(G')} \approx e^{-(|S|-1)(r+w)\varepsilon} = \left(\frac{D'}{D}\right)^{|S|-1}e^{-\varepsilon(|S|-1)},$$ 
since these events are nearly independent as $H$ has girth at least four. This event meanwhile is independent of whether the edges are chosen (and hence end up in $E'$). So essentially the term $d_i(H,e) \cdot (i-1)/D^{i-1}$ in $w_D(H,e)$ is replaced in $w_{D'}(H',e)$ by the following two terms: a first term that counts the edges of size $i$ of $H$ that remain size $i$ in $H'$ and a second term for those that shrink to size $i-1$ (further shrinkings are negligible as they have an $\varepsilon^2$ factor) 
\begin{align*}
&d_i(H,e) \cdot \frac{i-1}{(D')^{i-1}} \cdot \left(\frac{D'}{D}\right)^{i-1}e^{-\varepsilon(i-1)} + d_i(H,e) \cdot \frac{i-2}{(D')^{i-2}} \cdot \left(\frac{D'}{D}\right)^{i-1}e^{-\varepsilon(i-1)} \cdot \frac{\varepsilon}{D}\cdot (i-1)\\
&\approx d_i(H,e) \cdot \frac{i-1}{D^{i-1}}\cdot (1-(i-1)\varepsilon+(i-2)\varepsilon)\\
&= d_i(H,e) \cdot \frac{i-1}{D^{i-1}}\cdot (1-\varepsilon).
\end{align*}
Since we have this additional $(1-\varepsilon)$ factor for every $i$, we find that
$$\Expect{w_{D'}(H',e)} \approx w(1-\varepsilon).$$ 
Note the girth five assumption on $H$ is used to concentrate the above variable. Of course, to formally prove the key iterative lemma, Lemma~\ref{lem:Iterative}, we have to be more careful with the probabilities due to lower-order terms, then use Talagrand's Inequality to concentrate the variables and finally use the Lov\'asz Local Lemma to show all of these concentrations can happen simultaneously. Hence the proof grows much longer and fills all of Section~\ref{s:ProofIter}.

\subsection{Proof Overview for Bipartite Result}\label{ss:BipOverview}

To prove Theorem~\ref{thm:GirthFiveBipartite}, we instead use the following iterative lemma.

\begin{lem}\label{lem:IterativeBip}
For all integers $r,g \ge 2$, there exists an integer $D_{r,g}\ge 0$ such that the following holds for all $D\ge D_{r,g}$ and all $D_A$ with $D\log D \ge D_A \ge D$, and for all $\varepsilon \ge D^{-1/6r}$, $\gamma \le \frac{1}{24rg}$, $w \ge D^{-1/6}$ such that $\gamma \ge 16\cdot \frac{D_A}{D} \cdot rg\left(5r\cdot \frac{D_A}{D}+2w+w^2\right)\varepsilon$.

\vskip.1in
Let $G=(A,B)$ be a bipartite $r$-uniform (multi)-hypergraph on $n$ vertices and let $H$ be a $g$-bounded configuration hypergraph of $G$.
\vskip.1in
\noindent If $G$ satisfies
\vskip.1in
\begin{enumerate}
\item[(G1A)] every vertex in $A$ has degree at least $D_A$, and 
\item[(G1B)] every vertex in $B$ has degree at most $D$,
\item[(G2)] every pair of vertices of $G$ has codegree at most $4\log^2 D$,
\end{enumerate}
and $H$ satisfies
\vskip.1in
\begin{enumerate}
\item[(H1)] $w_D(H) \le w$,
\item[(H2)] $H$ has girth at least five,
\end{enumerate}
and $G$ and $H$ satisfy
\begin{enumerate}
\item[(G+H)] the maximum codegree of $G$ with $H$ is at most $2^{2(g-1)}\cdot \log^{2(g-1)} D$,
\end{enumerate}
then there exists an $H$-avoiding matching $E'$ of $G$ and a subhypergraph $G'$ of $G$ such that $V(G')= V(G)-\bigcup_{e\in E'} e$ and $E(G')\subseteq E(G[V(G')])$, and a configuration hypergraph $H'$ of $G'$ of girth at least five where $E(H'):= \{S\setminus E': S\in E(H),~S\setminus E' \subseteq E(G')\}$ that satisfy the all of the following, where we let
$$D_A' = D_A \cdot e^{-\varepsilon(r-1+w+\gamma)},$$
and
$$D':= D \cdot e^{-\varepsilon\left(r-2 + \frac{D_A}{D} + w - \gamma\right)}:$$
\begin{itemize}
\item[(1)] if $M$ is an $H'$-avoiding matching of $G'$, then $M\cup E'$ is an $H$-avoiding matching of $G$,
%\item[(2)] $v(G\setminus V(G')) \ge \frac{\varepsilon}{2} \cdot v(G)$,
\item[(3a)] $\forall a\in A\cap V(G')$, $d_{G'}(a) \ge D_A'$,
\item[(3b)] $\forall b\in B\cap V(G')$, $d_{G'}(b) \le D'$, and
\item[(4)]$w_{D'}(H') \le w \left(1-\frac{\varepsilon}{2}\right)$.
\end{itemize}
\end{lem}

We prove Lemma~\ref{lem:IterativeBip} in Section~\ref{s:Split} by explaining how to tweak the proof of Lemma~\ref{lem:Iterative}. First we can assume every vertex in $A$ has degree exactly $D_A$ by deleting edges, and then that every vertex in $B$ has degree exactly $D$ via a regularization lemma, Lemma~\ref{lem:reg}. Then we use the same procedure as for Lemma~\ref{lem:Iterative}. The main difference is that for $b\in B$,  $\Prob{b\in V(G')} \approx e^{-\varepsilon}$ while for $a\in A$, $\Prob{a\in V(G')} \approx e^{-\varepsilon \cdot \frac{D_A}{D}}$. This difference then means that $\Expect{d_{G'}(b)} \approx D\cdot e^{-\varepsilon(r-2 + (D_A/D) + w)}\approx D'$ while $\Expect{d_{G'}(a)} \approx D_A \cdot e^{-\varepsilon(r-1+w)} \approx D_A'$.

To prove Theorem~\ref{thm:GirthFiveBipartite}, we also need a finishing lemma as follows.

\begin{lem}\label{lem:Finish}
For every integer $r\ge 1$, there exists $D_0$ such that for all integers $D\ge D_0$ the following holds: 

Let $G=(A,B)$ be an $r$-uniform bipartite hypergraph such that $d_G(a) \ge 300rD$ for each $a\in A$ and $d_G(b)\le D$ for each $b\in B$. If $H$ is a configuration hypergraph of $G$ with $w_D(H) \le 1$, then there exists an $H$-avoiding $A$-perfect matching of $G$.
\end{lem} 

We also prove Lemma~\ref{lem:Finish} in Section~\ref{s:Split}. Assuming Lemmas~\ref{lem:IterativeBip} and~\ref{lem:Finish}, we are now prepared to prove Theorem~\ref{thm:GirthFiveBipartite} as follows.

\begin{proof}[Proof of Theorem~\ref{thm:GirthFiveBipartite}]
Recall $D\ge D_{r,g}$ by assumption. We choose $D_{r,g} \ge \max\{ D_{r,g}')^2, D_0^2\}$ large enough, where $D_{r,g}'$ is the value of $D_{r,g}$ in Lemma~\ref{lem:IterativeBip} and $D_0$ is the value in Lemma~\ref{lem:Finish}.

We may assume without loss of generality that $G$ is $r$-uniform by adding dummy vertices to edges as necessary. 

Let $\varepsilon := D_0^{-(7/6)\alpha}$ and $\gamma:= \frac{D_0^{-\alpha}}{4}$. Let $D_0:= D$, $D_{A,0}:= D(1+D^{-\alpha})$ and $w_0 := \frac{\alpha\log D_0}{2}$. Choose $T$ such that $$\left(1+\frac{\varepsilon}{2}\right)^{T} = 300r\cdot D^{\alpha}.$$

As $\Delta_i(H) \le \frac{\alpha}{g^2} \cdot D^{i-1}\log D$ for all $i\in \{2,\ldots, g\}$ by assumption, we find that $w_D(H) \le \frac{\alpha \log D_0}{2} = w_0$. Note that $w_0\ge 1 \ge D_0^{-1/6}$ since $D$ is large enough. 

For each integer $i$ with $1\le i\le T+1$, define the following parameters:
$$D_{A,i} := D_{A,0} \cdot e^{-\varepsilon\left(i(r-1+\gamma)+w_0\cdot \sum_{j=0}^{i-1} \left(1-\frac{\varepsilon}{2}\right)^j\right)},$$
%this was a typo in the submitted version, namely had $D_0$ instead of $D_{A,0}$
let $D_i$ satisfy
$$\frac{D_{A,i}}{D_i}-1 = \left(\frac{D_{A,0}}{D_{0}}-1\right) \cdot \left(1+\frac{\varepsilon}{2}\right)^i,$$
and
$$w_i := w_0 \cdot \left(1-\frac{\varepsilon}{2}\right)^i.$$

\begin{claim}\label{claim:iterativebip}
For each integer $0\le i \le T+1$, there exists an $H$-avoiding matching $E_i$ of $G$, a subhypergraph $G_i=(A_i,B_i)$ of $G$ with $V(G_i)=V(G)\setminus V(E_i)$ and a configuration hypergraph $H_i$ of $G_i$ of girth at least five where $E(H_i):= \{S\setminus E_i: S\in E(H),~S\setminus E_i \subseteq E(G_i)\}$ that satisfy all of the following:
\begin{enumerate}
\item[(1)] if $M$ is an $H_i$-avoiding matching of $G_i$, then $M_i\cup E_i$ is an $H$-avoiding matching of $G$, 
\item[(2)] every vertex in $A_i$ has degree at least $D_{A,i}$ in $G$,
\item[(3)] every vertex in $B_i$ has degree at most $D_i$ in $G_i$,
\item[(4)] $w_{D_i}(H_i)\le w_i$.
\end{enumerate}
\end{claim}
\begin{proof}
We proceed by induction on $i$. When $i=0$, then $G_0:=G$ and $H_0:= H$ satisfy the conclusion as desired. So we may assume that $i \ge 1$. By induction, there exists $G_{i-1}=(A_{i-1},B_{i-1})$ and $H_{i-1}$ satisfying (1)-(4).

First note that $$D_{A,i} = D_{A,i-1}\cdot e^{-\varepsilon(r-1+w_{i-1}+\gamma)},$$
$$\frac{D_{A,i}}{D_i}-1 = \left(\frac{D_{A,i-1}}{D_{i-1}}-1\right) \cdot \left(1+\frac{\varepsilon}{2}\right),$$
and 
$$w_i = w_{i-1} \cdot \left(1-\frac{\varepsilon}{2}\right).$$ 
In addition, 
$$\gamma = \frac{D_0^{-\alpha}}{4} \le \frac{1}{24rg}$$ 
since $\alpha > 0$ and $D_0$ is large enough.

Let 
$$D_{i-1}' := D_{i-1} \cdot e^{-\varepsilon\left(r-2 + \frac{D_{A,i-1}}{D_{i-1}} + w_{i-1} - \gamma\right)}.$$
First we show that $D_{i-1}'\le D_i$ as follows. To that end, note that $\frac{D_{A,0}}{D_0} = 1+4\gamma$ and hence $\frac{D_{A,i}}{D_i} = 1 + 4\gamma (1+\frac{\varepsilon}{2})^i$. Thus we first calculate that
$$\frac{D_{A,i}/D_i}{D_{A,i-1}/D_{i-1}} = \frac{(1+4\gamma (1+\frac{\varepsilon}{2})^i)}{(1+4\gamma (1+\frac{\varepsilon}{2})^{i-1})} \le 1+ \frac{\varepsilon}{2}\cdot 4\gamma \left(1+\frac{\varepsilon}{2}\right)^{i-1}= 1 + \frac{\varepsilon}{2}\cdot \left(\frac{D_{A,i-1}}{D_{i-1}}-1\right) \le e^{\frac{\varepsilon}{2} \left(\frac{D_{A,i-1}}{D_{i-1}}-1\right)},$$
where the final inequality follows since $D_{A,i-1}/D_{i-1} \ge 1$. Since $\frac{D_{A,i}}{D_{A,i-1}} = e^{-\varepsilon(r-1+w_{i-1}+\gamma)}$, we find by substitution that 
$$D_{i-1}' = D_{i-1} \cdot \frac{D_{A,i}}{D_{A,i-1}} \cdot e^{-\varepsilon\left(\frac{D_{A,i-1}}{D_{i-1}}-1\right)}\le D_i\cdot \frac{D_{A,i}/D_i}{D_{A,i-1}/D_{i-1}}\cdot e^{-\varepsilon\left(\frac{D_{A,i-1}}{D_{i-1}}-1\right)} \le D_i\cdot e^{-\frac{\varepsilon}{2}\left(\frac{D_{A,i-1}}{D_{i-1}}-1\right)} \le D_i.$$

Next we find that
$$\frac{D_{A,i-1}}{D_{i-1}} - 1 = \left( \frac{D_{A,0}}{D_0} -1 \right) \left(1+\frac{\varepsilon}{2}\right)^{i-1} \le \left( \frac{D_{A,0}}{D_0} -1 \right) \left(1+\frac{\varepsilon}{2}\right)^{T} \le D^{-\alpha} \cdot (300r \cdot D^{\alpha}) = 300r,$$
and hence
$$D_{i-1} \ge \frac{D_{A,i-1}}{300r+1} \ge \frac{D_{A,i-1}}{301r}.$$
Also note that 
\begin{align*}
D_{A,i-1} &\ge D_{A,0} \cdot e^{-(i-1)(r-1+\gamma)\varepsilon-2w_0} \ge D_{A,0} \cdot e^{-Tr\varepsilon-2w_0} \\
&\ge  D_{A,0} \cdot (300r)^{-4r} \cdot D_0^{-(4r+1)\alpha} \ge (301r) D_{A,0}^{1-(4r+2)\alpha} \ge (301r) \cdot D_{A,0}^{2/3}
\end{align*}
where we used that $T\varepsilon = 4\log(D_0^{\alpha}\cdot 300r)$, $\gamma \le 1$, $w_0 = \frac{\alpha \log D_0}{2}$, $D_{A,0}\ge D_0$, $r\ge 2$, $\alpha \le \frac{1}{15r}$ and that $D_{A,0}^{\alpha} \ge (300r)^{4r}\cdot 301r$ since $D$ is large enough.
Thus
$$\frac{D_{A,i-1}}{D_{i-1}} \le 301r \le \log D_{i-1}$$
since $D_{i-1} \ge \frac{D_{A,i-1}}{301r} \ge \frac{D_{A,0}^{2/3}}{301r}$ and $D$ is large enough. In addition,
$$D_{i-1} \ge \frac{D_{A,i-1}}{301r} \ge D_{A,0}^{2/3} \ge D_0^{2/3}.$$
Hence 
$$\varepsilon \ge D_0^{-(7/6)\alpha} \ge D_{i-1}^{ -(7/4)\alpha} \ge (D_{i-1})^{-1/6r}$$ 
since $\alpha \le \frac{2}{21r}$.

Since $w_{i-1}\le w_0$, we also have that 
$$\gamma = \frac{D_0^{-\alpha}}{4} \ge 16\cdot \frac{D_{A,i-1}}{D_{i-1}} \cdot rg\left(5r\cdot \frac{D_{A,i-1}}{D_{i-1}}+2w_{i-1}+w_{i-1}^2\right)\varepsilon,$$
since $\frac{D_{A,i-1}}{D_{i-1}} \le \log D_{i-1} \le \log D_0$, $w_{i-1}\le w_0 \le \log D_0$, $\gamma \ge \frac{\varepsilon}{4}\cdot D_0^{\alpha/6}$ and $D$ is large enough.
Moreover, using that $1-x \ge e^{-2x}$ for all $x\in [0,1/2]$, we find that
$$w_{i-1} = w_0 \cdot \left(1-\frac{\varepsilon}{2}\right)^{i-1} \ge e^{-T\varepsilon} = (300r)^{-4} D_0^{-4\alpha} \ge D_{i-1}^{-1/6},$$
where we used that $w_0\ge 1$, $i-1\le T$ and $D_{i-1} \ge D_0^{2/3} \ge \left(300r \cdot D_0^{\alpha}\right)^{24}$ as $\alpha \le \frac{1}{20r} \le \frac{1}{40}$ by assumption since $r\ge 2$ where we used that $D$ is large enough.

The codegree of pairs of vertices of $G_{i-1}$ is at most the codegree of pairs of vertices of $G_0$ which is at most $\log^2 D_0 \le 4\log^2 D_{i-1}$ since $D_{i-1} \ge D_0^{1/2}$.  Similarly the codegree of $G_i$ with $H_i$ is at most the codegree of $G_0$ with $H_0$ which is at most $\log^{2(g-1)} D_0 \le 2^{2(g-1)}\cdot \log^{2(g-1)} D_{i-1}$ since $D_{i-1} \ge D_0^{1/2}$. 

 Thus we have that $G_{i-1}, H_{i-1}, D_{i-1}, D_{A,i-1}, w_{i-1}, \gamma$ and $\varepsilon$ satisfy the conditions of Lemma~\ref{lem:IterativeBip}. Hence by Lemma~\ref{lem:IterativeBip} applied to $G_{i-1}$ and $H_{i-1}$, there exists $G_i=(A_i,B_i)$, $H_i$ and $E_i'$ satisfying (1), (3a), (3b) and (4) of Lemma~\ref{lem:IterativeBip}.
 
 Now Lemma~\ref{lem:IterativeBip}(3b) yields that $d_{G_i}(b) \le D_{i-1}'$ for every $b\in B_i$ and hence $d_{G_i}(b) \le D_i$ for every $b\in B_i$, that is (3) holds. Similarly Lemma~\ref{lem:IterativeBip}(4) yields that $w_{D_{i-1}'}(H_i)\le w_{i-1} \cdot \left(1-\frac{\varepsilon}{2}\right) = w_i$. Since $D_{i-1}' \le D_i$, then by the definition of $w_D(H)$, we find that $w_{D_i}(H_i) \le w_{D_{i-1}'}(H_i)$ and hence $w_{D_i}(H_i) \le w_i$, that is (4) holds. Hence $G_i$ and $D_i$ also satisfy (1)-(4) of this claim as desired where we let $E_i = E_i'\cup E_{i-1}$.
\end{proof}

Let $T^* = \lfloor T \rfloor + 1$. By Claim~\ref{claim:iterativebip}, there exists $G_{T^*}, H_{T^*}$ satisfying (1)-(4) of Claim~\ref{claim:iterativebip}. Hence $E_{T^*}$ is an $H_{T^*}$-avoiding matching of $G_{T^*}$. We note that if $A_i=\emptyset$ for any $i\le T+1$, then $E_{T^*}$ is $A$-perfect and hence $E_{T^*}$ is as desired. So we may assume that $A_{T^*}\ne \emptyset$.

Recall that by definition of $T$, we have
$$\left(1+\frac{\varepsilon}{2}\right)^T = 300r \cdot D^{\alpha}.$$
Thus, we find that
$$\frac{D_{A,T^*}}{D_{T^*}} - 1 \ge \left(\frac{D_{A,0}}{D_0}-1\right) \left(1+\frac{\varepsilon}{2}\right)^T \ge D^{-\alpha} \cdot (300r \cdot D^{\alpha}) = 300r,$$
and hence
$$D_{A,T^*} \ge 300r \cdot D_{T^*}.$$
Meanwhile,
$$w_{T+1} = w_0 \cdot \left(1-\frac{\varepsilon}{2}\right)^T \le w_0 \cdot e^{-\varepsilon T /2} = w_0 \cdot D^{-2\alpha} \cdot \frac{1}{(300r)^2} \le 1$$
where for the last inequality we used that $w_0 \le \log D_0$ and $D$ is large enough.

Thus by Lemma~\ref{lem:Finish}, there exists an $H_{T^*}$-avoiding $A_{T^*}$-perfect matching $E_{T^*}'$ of $G_{T^*}$. But then $E_{T^*}\cup E_{T^*}'$ is an $H$-avoiding $A$-perfect matching of $G$ as desired.
\end{proof}

\subsection{Overview of Small Codegree Result}

To prove Theorem~\ref{thm:SmallCodegree}, we use random sparsification; namely, we choose each edge of $G$ independently with probability $p=D^{\beta/4g -1}$ and show the resulting graph satisfies the conditions of Theorem~\ref{thm:GirthFive} (or Theorem~\ref{thm:GirthFiveBipartite} as needed) with positive probability. Here is a lemma that encapsulates this.

\begin{lem}\label{lem:Sparsifying}
For all integer $r, g\ge 2$ and real  $\beta \in (0,1)$, there exists an integer $D_{\beta}$ such that the following holds for all $D\ge D_{\beta}$:

Let $G$ be an $r$-bounded hypergraph of maximum degree at most $2D$. Let $H$ be a $g$-bounded configuration hypergraph of $G$ with $\Delta_i(H) \le \alpha \cdot D^{i-1}\log D$ for all $2\le i\le g$ and $\Delta_{k,\ell}(H) \le D^{k-\ell-\beta}$ for all $2\le \ell< k\le g$. 

If the maximum $2$-codegree of $G$ with $H$ and the maximum common $2$-degree of $H$ are both at most $D^{1-\beta}$,
then letting $p=D^{\beta/4g -1}$, there exists a subhypergraph $G'$ of $G$ with $V(G')=V(G)$ and a configuration hypergraph $H' = H[E(G')]$ such that all of the following hold:
\begin{itemize}
    \item[(1)] for each $v\in V(G')$, $d_{G'}(v) = p\cdot d_G(v)(1 \pm D^{-\beta/18g}) \pm D^{\beta/6g}$,
    \item[(2)] the codegree of every pair of vertices of $G'$ is at most $\log^2(pD)$,
    \item[(3)] $w_{pD}(H') \le w_D(H)+1$,
    \item[(4)] the maximum codegree of $G'$ with $H'$ is at most $\log^{2(g-1)}(pD)$,
    \item[(5)] $H'$ has girth at least five.
\end{itemize}
\end{lem}

The proof of Lemma~\ref{lem:Sparsifying} is quite technical and utilizes layering many applications of our exceptional outcomes version of Talagrand's Inequality. The proof can be found in Section~\ref{s:Generalizations}.

Assuming Lemma~\ref{lem:Sparsifying}, we are now prepared to prove Theorem~\ref{thm:SmallCodegree} as follows.

\begin{proof}[Proof of Theorem~\ref{thm:SmallCodegree}]
Apply Lemma~\ref{lem:Sparsifying} and then apply Theorem~\ref{thm:GirthFiveBipartiteMany} to the resulting hypergraph and configuration hypergraph. \end{proof}

\section{Linear Talagrand's Inequality}\label{s:Talagrand}

To prove that variables are concentrated around their expectations with high probability, we need a version of Talagrand's Inequality~\cite{T95} where the Lipschitz constant is linear. Surprisingly, this is possible by a simple arithmetic trick; one that is commonly used in the literature when applying Freedman's Inequality on martingales. However, martingale inequalities tend to only be useful in the `dense' case whereas Talagrand's Inequality is often needed for sparser applications such as coloring bounded degree graphs. So here we introduce this new linear version of Talagrand's Inequality. We think it will have numerous applications beyond its use here; we also think that many of the proofs using Freedman's Inequality could probably be proved instead with this version of Talagrand's Inequality. 

To prove our new version, we require a non-uniform generalization of a version of Talagrand's Inequality with exceptional events by Kelly and the second author (Theorem 6.3 in~\cite{KP20}). First a definition.

\begin{definition}
  Let $((\Omega_i, \Sigma_i, \mathbb P_i))_{i=1}^n$ be probability spaces, let $(\Omega, \Sigma, \mathbb P)$ be their product space, let $\Omega^* \subseteq \Omega$ be a set of \textit{exceptional outcomes}, and let $X : \Omega \rightarrow \mathbb R_{\geq0}$ be a non-negative random variable.  Let $b > 0$.
  \begin{itemize}
  \item If $\omega = (\omega_1, \dots, \omega_n) \in \Omega$ and $s > 0$, a \textit{$b$-certificate} for $X, \omega, s$, and $\Omega^*$ is an index set $I\subseteq\{1, \dots, n\}$ and a vector $(c_i: i\in I)$ with $\sum_{i\in I} c_i^2 \le bs$ such that for all $I'\subseteq I$, we have that
  \begin{equation*}
    X(\omega') \geq s - \sum_{i\in I'} c_i,
  \end{equation*}
  for all $\omega' = (\omega'_1, \dots, \omega'_n)\in\Omega\setminus\Omega^*$  such that $\omega_i=\omega_i'$ for all $i\in I\setminus I'$.
  \item If for every $\omega\in\Omega\setminus\Omega^*$, there exists a $b$-certificate for $X, \omega, s:= X(\omega)$, and $\Omega^*$, then $X$ is \textit{$b$-certifiable} with respect to $\Omega^*$.
  \end{itemize}
\end{definition}

\begin{thm}\label{exceptional talagrand's}
  Let $((\Omega_i, \Sigma_i, \mathbb P_i))_{i=1}^n$ be probability spaces, let $(\Omega, \Sigma, \mathbb P)$ be their product space, let $\Omega^* \subseteq \Omega$ be a set of exceptional outcomes, and let $X : \Omega \rightarrow \mathbb R_{\geq0}$ be a non-negative random variable.  Let $b > 0$.
  
  If $X$ is $b$-certifiable with respect to $\Omega^*$, 
then for any $t > 96\sqrt{b\Expect{X}} +  128b + 8\Prob{\Omega^*}(\sup X),$
\begin{equation*}
\Prob{|X - \Expect{X}| > t} \leq 4\exp\left({\frac{-t^2}{2b(4\Expect{X} + t)}}\right) + 4\Prob{\Omega^*}.
\end{equation*}
\end{thm}

We note that the proof is a straightforward generalization of the proof of Theorem 6.3 in~\cite{KP20} but include a proof in Appendix~\ref{s:Appendix} for completeness. Our Linear Talagrand's Inequality will be a corollary of Theorem~\ref{exceptional talagrand's}. But first some definitions.

\begin{definition}
  Let $((\Omega_i, \Sigma_i, \mathbb P_i))_{i=1}^n$ be probability spaces, let $(\Omega, \Sigma, \mathbb P)$ be their product space, let $\Omega^* \subseteq \Omega$ be a set of \textit{exceptional outcomes}, and let $X : \Omega \rightarrow \mathbb R_{\geq0}$ be a non-negative random variable.  Let $r,d> 0$.
  \begin{itemize}
  \item If $\omega = (\omega_1, \dots, \omega_n) \in \Omega$ and $s \ge 0$, an \textit{$(r,d)$-certificate} for $X, \omega, s$, and $\Omega^*$ is an index set $I\subseteq\{1, \dots, n\}$ and a vector $(c_i: i\in I)$ with $\sum_{i\in I} c_i \le rs$ and $\max_{i\in I} c_i \le d$ such that for all $I'\subseteq I$, we have that
  \begin{equation*}
    X(\omega') \geq s - \sum_{i\in I'} c_i,
  \end{equation*}
  for all $\omega' = (\omega'_1, \dots, \omega'_n)\in\Omega\setminus\Omega^*$  such that $\omega_i=\omega_i'$ for all $i\in I\setminus I'$.
  \item If for every $\omega\in\Omega\setminus\Omega^*$, there exists an $(r,d)$-certificate for $X, \omega, s:=X(\omega)$, and $\Omega^*$, then $X$ is \textit{$(r,d)$-observable} with respect to $\Omega^*$.
  \end{itemize}
\end{definition}

We are now ready to state our Linear Talagrand's Inequality.

\begin{thm}[Linear Talagrand's Inequality]\label{exceptional talagrand's observations}
  Let $((\Omega_i, \Sigma_i, \mathbb P_i))_{i=1}^n$ be probability spaces, let $(\Omega, \Sigma, \mathbb P)$ be their product space, let $\Omega^* \subseteq \Omega$ be a set of exceptional outcomes, and let $X : \Omega \rightarrow \mathbb R_{\geq0}$ be a non-negative random variable.  Let $r,d \geq 0$.
  
 If $X$ is $(r,d)$-observable with respect to $\Omega^*$, 
then for any $t > 96\sqrt{rd\Expect{X}} +  128rd + 8\Prob{\Omega^*}(\sup X),$
\begin{equation*}
\Prob{|X - \Expect{X}| > t} \leq 4\exp\left({\frac{-t^2}{8rd(4\Expect{X} + t)}}\right) + 4\Prob{\Omega^*}.
\end{equation*}
\end{thm}
\begin{proof}
By Theorem~\ref{exceptional talagrand's}, it suffices to prove that $X$ is $b$-certifiable with respect to $\Omega^*$ where $b=rd$.

Let $s > 0$ and $\omega\in\Omega\setminus\Omega^*$ such that $X(\omega) \geq s$. By definition of $(r,d)$-observable, there exists an $(r,d)$-certificate for $X, \omega, s$, and $\Omega^*$. Recall that such an $(r,d)$-certificate is an index set $I\subseteq\{1, \dots, n\}$ and a vector $(c_i: i\in I)$ with $\sum_{i\in I} c_i \le rs$ and $\max_{i\in I} c_i \le d$ such that for all $I'\subseteq I$, we have that
  \begin{equation*}
    X(\omega') \geq s - \sum_{i\in I'} c_i,
  \end{equation*}
  for all $\omega' = (\omega'_1, \dots, \omega'_n)\in\Omega\setminus\Omega^*$  such that $\omega_i=\omega_i'$ for all $i\in I\setminus I'$. But then
\begin{align*}
\sum_{i\in I} c_i^2 &\le \sum_{i\in I}c_i \cdot \max_{i\in I} \{c_i\} \\
&= \left(\max_{i\in I} \{c_i\}\right) \cdot \sum_{i\in I}c_i\\
&\le (d)\cdot (rs) = (rd) \cdot s = b\cdot s.
\end{align*}

\noindent Hence $I,(c_i: i\in I)$ is a $b$-certificate for $X, \omega, s$, and $\Omega^*$. Thus $X$ is $b$-certifiable with respect to $\Omega^*$ as desired.
\end{proof}

We note Theorem~\ref{exceptional talagrand's observations} is quite user-friendly; for many applications, we have that the random variable $X$ is the sum of $\{0,1\}$ random variables $X_i$ which are all functions of some set of independent trials; then $X$ being $(r,d)$-observable would mean that each $X_i$ can be verified to equal $1$ by observing a set of at most $r$ trials, which we think of as an \emph{observation} for $X_i$, (the specific observation set may change depending on the outcome of the trials) with the promise that each trial is used in at most $d$ observations for any particular outcome (or in the exceptional outcomes version, by at most $d$ observations in any particular non-exceptional outcome). We refer to this $d$ as \emph{observation degree}.

For many applications, the variable $X$ is not itself $(r,d)$-observable but rather can be written as the difference of a number of $(r,d)$-observable variables; this setup also admits concentration provided these variables have bounded expectation as follows. 

\begin{thm}[Linear Talagrand's Difference Inequality]\label{exceptional talagrand's difference}
  Let $((\Omega_i, \Sigma_i, \mathbb P_i))_{i=1}^n$ be probability spaces, let $(\Omega, \Sigma, \mathbb P)$ be their product space, let $\Omega^* \subseteq \Omega$ be a set of exceptional outcomes.  Let $r,d \geq 0$.
  
 Suppose that $X=\sum_{i=1}^m \varepsilon_i X_i$ where for each $i\in [m]$,  $\varepsilon_i \in \{-1,1\}$ and $X_i: \Omega \rightarrow \mathbb R_{\geq0}$ is a non-negative random variable that is $(r,d)$-observable with respect to $\Omega^*$ such that $\Expect{X_i} \le M$, 
then for any $t > m(96\sqrt{rdM} +  128rd + 8\Prob{\Omega^*}(\max_{i\in [m]}\sup X_i)),$
\begin{equation*}
\Prob{|X - \Expect{X}| > t} \leq 4m\exp\left({\frac{-t^2}{8rdm^2(4M + t)}}\right) + 4m\Prob{\Omega^*}.
\end{equation*}
\end{thm}
\begin{proof}
Apply Theorem~\ref{exceptional talagrand's observations} separately to each $X_i$ with $t_i := t/m$ and use the fact that $\Expect{X_i} \le M$.
\end{proof}
 
Here is a useful proposition showing how scaling a random variable scales its observability.

\begin{proposition}\label{prop:scaling}
 Let $((\Omega_i, \Sigma_i, \mathbb P_i))_{i=1}^n$ be probability spaces, let $(\Omega, \Sigma, \mathbb P)$ be their product space, let $\Omega^* \subseteq \Omega$ be a set of exceptional outcomes.  Let $r,d \geq 0$.
 
 If $a>0$ is a real number and $X$ is $(r,d)$-observable with respect to $\Omega^*$, then $aX$ is $(r,ad)$-observable with respect to $\Omega^*$.
\end{proposition} 
\begin{proof}
Let $s > 0$ and $\omega\in\Omega\setminus\Omega^*$ such that $aX(\omega) \geq s$. Hence $X(\omega) > s/a > 0$. Since $X$ is $(r,d)$-observable, we have by definition that there exists an $(r,d)$-certificate for $X,\omega, s/a, \Omega^*$, that is there exists an index set $I\subseteq\{1, \dots, n\}$ and a vector $(c_i: i\in I)$ with $\sum_{i\in I} c_i \le r(s/a)$ and $\max_{i\in I} c_i \le d$ such that for all $I'\subseteq I$, we have that
  \begin{equation*}
    X(\omega') \geq (s/a) - \sum_{i\in I'} c_i,
  \end{equation*}
  for all $\omega' = (\omega'_1, \dots, \omega'_n)\in\Omega\setminus\Omega^*$  such that $\omega_i=\omega_i'$ for all $i\in I\setminus I'$.
For each $i\in I$, let $c_i' = a \cdot c_i$. Then $\sum_{i\in I} c_i' = a \sum_{i\in I} c_i \le rs$ and $\max_{i\in I} c_i' = a \max_{i\in I} c_i = ad$ while for all $I'\subseteq I$, we have that
  \begin{equation*}
    aX(\omega') \geq s - \sum_{i\in I'} c_i',
  \end{equation*}
  for all $\omega' = (\omega'_1, \dots, \omega'_n)\in\Omega\setminus\Omega^*$  such that $\omega_i=\omega_i'$ for all $i\in I\setminus I'$.
Hence $I$ and the vector $(c_i':i\in I)$ are an $(r,ad)$-certificate for $aX, \omega, s, \Omega^*$ and hence $aX$ is $(r,ad)$-observable with respect to $\Omega^*$.
\end{proof}
 
\section{Proof of Key Iterative Lemma}\label{s:ProofIter}

In this section, we prove Lemma~\ref{lem:Iterative}. In Section~\ref{ss:Procedure}, we define our random procedure as well as the matching $E'$, hypergraph $G'$ and configuration hypergraph $H'$ in such a way that (1) from Lemma~\ref{lem:Iterative} holds. In Section~\ref{ss:Bad}, we define our bad events related to variables for proving that (2), (3) and (4) hold. In Section~\ref{ss:Exp}, we calculate the expectations of these variables via various auxiliary variables. In Section~\ref{ss:Conc}, we define exceptional outcomes and then show these variables are concentrated via the exceptional outcomes version of Talagrand's Inequality. In Section~\ref{ss:LocalLemma}, we use the results of the previous subsections and the Lov\'asz Local Lemma to finish the proof.

\subsection{The Procedure}\label{ss:Procedure}

Let $E_0$ be obtained by choosing each edge of $G$ independently with probability $p:= \varepsilon / D$. 

Note if $e\in E(G)$ and $S\in E(H)$ such that $e\in S$, then
$$\Prob{S\setminus e\subseteq E_0} = p^{|S|-1} = \left(\frac{\varepsilon}{ D}\right)^{|S|-1}.$$
By the union bound
$$\Prob{\exists~S \in E(H): S\setminus e \subseteq E_0  } \le \sum_{i\ge 2} d_i(H,e) \left(\frac{\varepsilon}{ D}\right)^{i-1} \le \varepsilon \cdot \sum_{i\ge 2} \frac{d_i(H,e)}{D^{i-1}} \le \varepsilon \cdot w,$$
where we used that $\varepsilon \le 1$ by assumption. Recall that $\varepsilon w \le \gamma \le 1$ by assumption.

For each edge $e$ of $G$, perform an equalizing coin flip, specifically a $\{0,1\}$-variable $\gamma_e$ where $\gamma_e = 1$ with probability
$$\frac{1-\varepsilon \cdot w}{1-\Prob{\exists~S \in N_H'(e): S\setminus e \subseteq E_0  } }.$$
Note that this quantity is at most $1$ and hence is a valid probability.

Let $F_0$ be the set of edges $e$ of $G$ such that at least one of the following hold: 
\begin{enumerate}
\item[(i)] there exists $S\in E(H)$ such that $S\setminus e$ is a subset of $E_0$, or
\item[(ii)] $\gamma_e = 0$.
\end{enumerate} 
Note then by definition of $\gamma_e$ that $\Prob{e\in F_0} = \varepsilon\cdot w$ for every edge $e\in E(G)$. 

We also calculate the probability of (i) not holding for an edge $e$ as
$$\Prob{S\setminus e \text{ is not a subset of } E_0 : \forall S \in N_H'(e)} =  \prod_{S\in N_H'(e)} \Prob{S\setminus e \text{ is not a subset of } E_0 } = \prod_{S\in N_H'(e)} \left(1-p^{|S|-1}\right)$$ 
where the first equality follows since as $H$ has girth at least three, we have that $(S\setminus e)\cap (S'\setminus e)=\emptyset$ for all $S,S'\in N_H'(e)$. Hence by the definition of $F_0$ (as properties (i) and (ii) are independent), we find that 
$$\Prob{e\not\in F_0} = \Prob{\gamma_e=1} \cdot \Prob{ S\setminus e \text{ is not a subset of } E_0 : \forall S \in N_H'(e)} = 1 -\varepsilon\cdot w.$$

Let $E_1$ be the set of edges of $E_0$ that are not incident to another edge of $E_0$ and let $E_2:= E_0\setminus E_1$ be the set of edges of $E_0$ incident with at least one other edge of $E_0$. 

We let 
\begin{itemize}
    \item $E':=E_1\setminus F_0$, 
    \item $V':=V(G)-\bigcup_{e\in E'} e$,
    \item $G':=G[V'] \setminus F_0$.
\end{itemize}

Recall that by definition
$$E(H'):= \{S\setminus E': S\in E(H), (S\setminus E') \subseteq E(G')\}.$$
Since $E(G')\cap F_0=\emptyset$, we have that every edge of $H'$ has size at least two and hence $H'$ is indeed a configuration hypergraph of $G'$. Meanwhile, if $H'$ has an $i$-cycle $e_1,v_1,\ldots, e_i,v_i$ for some $i\ge 2$, then $H$ has an $i$-cycle since each edge $e_j$ in $H'$ extends to an edge $e_j'$ in $H$. Hence since $H$ has girth at least five, we have that $H'$ has girth at least five. 

Since $E'\subseteq E_1$, we have that $E'$ is a matching of $G$. By property (i) of the definition of $F_0$ and since $E'\cap F_0=\emptyset$, it follows that $E'$ is $H$-avoiding. Moreover, it follows from the definition of $H'$ that if $M$ is an $H'$-avoiding matching of $G'$, then $M\cup E'$ is an $H$-avoiding matching of $G$ and hence (1) holds.

\subsection{Bad Events}\label{ss:Bad}

It remains to show that all of (2), (3) and (4) hold simultaneously for $E'$, $G'$ and $H'$ with some positive probability. To that end, we use the Lov\'asz Local Lemma. However, before we define our bad events, we need a lower bound on $n$ in terms of $D$ as follows.

Since $G$ has codegrees at most $4\log^2 D$ by assumption (G2), it follows that $G$ has multiplicity at most $4\log^2 D$. Hence 
$$e(G) \le \binom{n}{r} \cdot 4\log^2 D \le \frac{n^r}{r} \cdot 4\log^2 D.$$ 
Yet 
$$e(G) \ge \frac{Dn}{r} \cdot (1 - \gamma) \ge \frac{Dn}{2r}$$ 
and hence, combining the above, it follows that 
$$n \ge \left(\frac{D}{8\log^2 D}\right)^{1/(r-1)}.$$
Since $D_0$ is large enough, we have that
$$n \ge D^{1/r}.$$

\noindent {\bf Bad Events for (2):} Partition $V(G)$ into sets $X_1,X_2,\ldots, X_m$ of size between $D^{1/r}$ and $2\cdot D^{1/r}$; this is possible since $v(G) \ge D^{1/r}$. For each $i\in [m]$, let $A_i$ be the event that $|X_i\setminus V'| < \frac{\varepsilon}{2} \cdot |X_i|$.
\vskip.25in

\noindent {\bf Bad Events for (3):} For each $v\in V(G)$, let 
$$D_v := \{e\in E(G): v\in e,~e\not\in F_0,~e\setminus v \subseteq V(G')\}.$$
Let $A_v$ be the event that $|D_v| \ne D'(1\pm \gamma(1+6r\varepsilon))$. Note that when $v\in V(G')$, then $d_{G'}(v)=|D_v|$.
\vskip.25in

\noindent {\bf Bad Events for (4):} For each $e\in V(H)$,
let $F_{0,e}$ be the set of edges $f$ of $G$ such that at least one of the following hold: 
\begin{enumerate}
\item[(i)] there exists $S\in E(H)$ with $e\not\in S$ such that $S\setminus f$ is a subset of $E_0$, or
\item[(ii)] $\gamma_f = 0$.
\end{enumerate} 
Note that $F_{0,e} \subseteq F_{0}$. Moreover if $S\in N_H(e)$ and $f\in S\setminus e$, then since $H$ has girth at least three, it follows that
$$\Prob{f\not\in F_{0,e}}\cdot \left(1-p^{|S|-1}\right) = \Prob{f\not\in F_0} = 1- \varepsilon\cdot w.$$
For each $e\in V(H)$, let
$$D_{e,i,j} := \{ T\in N_H'(e): V(T\setminus (E'\cup e))\subseteq V(G'), (T\setminus e)\cap F_{0,e} = \emptyset, |T|=i, |T\setminus (E'\setminus e)|=j\}.$$

Note that if $e\in E(G')$, then
$$d_{H',~j}(e) \le \sum_{i\ge j} |D_{e,i,j}|.$$
Hence if $e\in E(G')$, then
\begin{align*}
w_{D'}(H',e) &\le \sum_{j\ge 2} d_{H',~j}(e) \cdot \frac{j-1}{(D')^{j-1}}\\
&\le \sum_{j\ge 2} \sum_{i\ge j} |D_{e,i,j}| \cdot \frac{j-1}{(D')^{j-1}}.
\end{align*}
For each $e\in V(H)$, let 
$$w(e) := \sum_{j\ge 2} \sum_{i\ge j} |D_{e,i,j}| \cdot \frac{j-1}{(D')^{j-1}}$$
and let $A_e$ be the event that 
$$w(e) > w \cdot \left(1-\frac{\varepsilon}{2}\right).$$
\vskip.25in

Let 
$$\mathcal{E} = \bigcup_{i\in [m]} A_i \cup \bigcup_{v\in V(G)} A_v \cup \bigcup_{e\in V(H)} A_e$$
be the set of all bad events. Note that if none of the $A_i$ happen for all $i\in [m]$, then
$$v(G\setminus V') = \sum_{i\in [m]} |X_i\setminus V'| \ge \sum_{i\in [m]} |X_i| \cdot \varepsilon/2 = v(G) \cdot \varepsilon/2,$$
and hence (2) holds as desired. Similarly if none of the $A_v$ happen for all $v\in V(G)$, then (3) holds as desired, and if none of the $A_e$ happen for all $e\in E(G)$, then (4) holds as desired. Hence it remains to show that with positive probability none of the events in $\mathcal{E}$ happen.

\subsection{Expectations}\label{ss:Exp}

We say an edge $e$ is \emph{unearthed} if $e\in E_2 \cup (E_0\cap F_0)$. We say a vertex $v$ is \emph{unearthed} if it is incident with an unearthed edge. Note then that every unearthed vertex is in $V(G')$. We say a vertex $v$ is \emph{uncovered} if $v$ is not incident with an edge of $E_0$ and \emph{covered} otherwise. Hence $V(G')$ consists precisely of the unearthed and uncovered vertices. Let $V_0$ denote the set of uncovered vertices, $V_1$ the set of covered vertices and $V_2$ the set of unearthed vertices. 

\subsubsection{Basic Calculations}

Before proceeding with the expectations and concentrations of the variables in the bad events, we perform some basic calculations. 

For each $v\in V(G)$, using $\gamma \le 1$, we upper bound the probability a vertex is uncovered as
\begin{align*}
\Prob{v\in V_0} &= (1-p)^{d_G(v)} \le e^{-p\cdot d_G(v)} \le e^{-\varepsilon(1-\gamma)} \le 1 - \varepsilon (1-\gamma) + \varepsilon^2(1-\gamma)^2\\
&\le 1 -\varepsilon(1-\gamma)+\varepsilon^2,
\end{align*}
where we used that $p=\frac{\varepsilon}{D}$ and $d_G(v)\ge D(1-\gamma)$. Since $V_1 = V(G)\setminus V_0$, we lower bound the probability a vertex is covered as
$$\Prob{v\in V_1} \ge \varepsilon (1- \gamma) - \varepsilon^2.$$
Meanwhile for each edge $e\in E(G)$,
$$\Prob{e\in E_2} = \Prob{e\in E_0}\cdot \Prob{N(e)\cap E_0\ne \emptyset} \le rD(1+\gamma)p^2 \le  \frac{2r\varepsilon^2}{D},$$
where we used that $\gamma \le 1$, while 
$$\Prob{e\in E_0\cap F_0} = \Prob{e\in E_0}\cdot \Prob{e\in F_0} = \frac{\varepsilon^2w}{D}.$$
Hence we upper bound the probability that an edge is unearthed as
\begin{align*}
\Prob{e\in E_2\cup (E_0\cap F_0)} &\le \Prob{e\in E_2} + \Prob{e\in E_0\cap F_0} \\
&\le \frac{\varepsilon^2}{D} \cdot \left(2r+w\right).
\end{align*}
Thus by the union bound (and using $d_G(v) \le 2D$), we upper bound the probability a vertex is unearthed as
\begin{align*}
\Prob{v\in V_2} &\le  \Prob{N'_G(v)\cap (E_2\cup (E_0\cap F_0)) \ne \emptyset} \\
&\le d_G(v)\cdot \frac{\varepsilon^2}{D} \cdot \left(2r+w\right)
 \\
&\le \varepsilon^2 \cdot \left(4r+2w\right).
\end{align*}
Combining, we lower bound (using $r\ge 1$) the probability that a vertex is covered but not unearthed as (and hence is precisely in $V(G)\setminus V(G')$)
$$\Prob{v\in V_1\setminus V_2} \ge \varepsilon (1-\gamma) - \varepsilon^2  \cdot \left(5r+2w\right).$$

\subsubsection{Expectation of $|X_i\setminus V'|$}

Recall that the sets $X_i$ form a partition of $V(G)$ with size between $D^{1/r}$ and $2\cdot D^{1/r}$. Note that $X_i\setminus V' = X_i \cap (V_1\setminus V_2)$. Hence by linearity of expectation, we have by the calculations above that 
\begin{align*}
\Expect{|X_i\cap (V_1\setminus V_2)|} &\ge  |X_i|\cdot (\varepsilon (1- \gamma) - \varepsilon^2  \cdot \left(5r+2w\right) )\\
&\ge |X_i| \cdot \left(\frac{3}{4} \cdot \varepsilon\right),
\end{align*}
since $\varepsilon(5r+2w) \le \gamma \le \frac{1}{8}$ by assumption.

\subsubsection{Expectation of $d_{G'}(v)$}

Recall that
$$D_v := \{e\in E(G): v\in e,~e\not\in F_0,~e\setminus v \subseteq V(G')\},$$
and
$$D':= De^{-\varepsilon(r-1 + w)}.$$
Now we relate the variable $D_v$ to other variables that are easier to calculate as follows. Let
\begin{align*}
&D_{0,v} := \{e\in N'_G(v): e\setminus v \subseteq V_0\}, \\
&D_{0,v}' := D_{0,v} \setminus F_0, \\
&D_{2,v} := \{e\in N'_G(v): (e\setminus v)\cap V_2 \ne \emptyset\}.
\end{align*}
Since $D_{0,v}' \subseteq D_v \subseteq D_{0,v}'\cup D_{2,v}$, we find that
$$|D_{0,v}'| \le |D_v| \le |D_{0,v}'| + |D_{2,v}|.$$

Now we proceed to calculate upper (and for some also lower) bounds on the probability that $e\in N'_G(v)$ is in each of $D_{0,v}, D_{0,v}', D_{2,v}$ as follows. 

Fix an edge $e\in N'_G(v)$ and let $B_e := \bigcup_{u\in e\setminus v} N'_G(u)$. Now we have that
$$\Prob{e\in D_{0,v}} = \Prob{B_e \cap E_0 =\emptyset} = (1-p)^{|B_e|}.$$
Yet using assumptions $(G1)$ and $(G2)$, we find that
\begin{align*}
(r-1)D(1+\gamma) \ge |B_e| &\ge (r-1)D(1-\gamma) - \binom{r-1}{2}(4\log^2 D)\\
&\ge D(r-1- r\gamma),
\end{align*}
since $\gamma \ge 4r^2 \frac{\log^2 D}{D}$ as $\gamma \ge \varepsilon \ge D^{-1/6r}$ by assumption and $D$ is large enough.

Note that the events $e\in F_0$ and $e\in D_{0,v}$ are independent and hence
$$\Prob{e\in D_{0,v}'} = (1-\varepsilon w) \cdot (1-p)^{|B_e|}.$$

Thus for an upper bound, we have
\begin{align*}
\Prob{e\in D_{0,v}'} &\le e^{-\varepsilon w} \cdot e^{-p|B_e|}\\
&\le e^{-\varepsilon(r-1-r\gamma + w)}\\
&= \frac{D'}{D} \cdot e^{r\gamma\varepsilon}.
\end{align*}
Recall that for $x\in [0,1/2]$, $1-x \ge e^{-x-x^2}$. Thus for a lower bound, we have
\begin{align*}
\Prob{e\in D_{0,v}'} &\ge e^{-\varepsilon w - \varepsilon^2 w^2} \cdot \left(e^{-p- p^2}\right)^{|B_e|} \\
&\ge e^{-\varepsilon((r-1)(1+\gamma)+ w) - \varepsilon^2 (w^2 +(r-1)(1+\gamma)) }\\
&\ge e^{-\varepsilon(r-1+ w) - r\gamma\varepsilon}\\
&= \frac{D'}{D} \cdot e^{-r\gamma\varepsilon},
\end{align*}
where we used that $\gamma \le 1/2$ and $\varepsilon(w^2+2(r-1)) \le \gamma$ by assumption.

Meanwhile, by the union bound, for each $e$ with $v\in e$, we have that
$$\Prob{e\in D_{2,v}} = \Prob{ (e\setminus v)\cap V_2 \ne \emptyset} \le (r-1) \varepsilon^2 \cdot \left(4r+2w\right) \le \gamma \varepsilon,$$
since $(r-1)(4r+2w)\varepsilon \le \gamma$ by assumption. 

By linearity of expectation, for a lower bound we have
$$\Expect{|D_{0,v}'|} \ge d_G(v)\cdot \frac{D'}{D} \cdot e^{-r\gamma\varepsilon} \ge D'(1-\gamma(1+r\varepsilon)).$$
Similarly by linearity of expectation, for an upper bound we have
\begin{align*}
\Expect{|D_{0,v}'|} &\le d_G(v)\cdot \frac{D'}{D} \cdot e^{r\gamma\varepsilon} \le D'(1+\gamma(1+r\varepsilon(1+r\gamma\varepsilon))) \\
&\le  D'(1+\gamma(1+2r\varepsilon)),
\end{align*}
since $r\gamma \varepsilon \le r \gamma \le 1$ by assumption, where we used that for $x\in [0,1]$, we have $e^{x} \le 1 + x + x^2$.
Similarly, by linearity of expectation, we have
$$\Expect{|D_{2,v}|} \le \gamma \varepsilon\cdot d_G(v) \le \gamma \varepsilon D (1+\gamma) \le 2D' \gamma\varepsilon,$$
since $D(1+\gamma) \le 2D'$ as $\varepsilon(r-1 + w) \le \gamma \le 1/8$ by assumption.

Combining, we have that
$$\Expect{|D_v|} \ge \Expect{|D_{0,v}'|} \ge D'(1-\gamma(1+r\varepsilon)),$$
and
$$\Expect{|D_v|} \le \Expect{|D_{0,v}'|} + \Expect{|D_{2,v}|} \le D'(1+\gamma(1+4r\varepsilon)).$$

\subsubsection{Expectation of $w_{D'}(H',e)$}

For $S\in E(H)$, let $V_G(S):= \bigcup_{e\in S} e$. Now we relate $D_{e,i,j}$  to other variables that are easier to calculate as follows.

For each $2\le j \le i\le g$, let 
\begin{align*}
D_{0,e,i} &:= \left\{ S\in N'_H(e): \left(\bigcup_{u\in V_G(S)\setminus e} N_G'(u)\setminus (S\setminus e)\right) \cap E_0 = \emptyset, |S|=i \right\},\\
D_{0,e,i}' &:= \{ S\in D_{0,e,i}: (S\setminus e) \cap F_{0,e} = \emptyset \},\\
D_{0,e,i,j}' &:= \{ S\in D_{0,e,i}': |(S\setminus e)\cap E_0|=i-j \}.\\
D_{2,e,i} &:= \{ S\in N'_H(e): (V_G(S)\setminus e) \cap V_2 \ne \emptyset, |S|=i \},\\
D_{2,e,i,j} &:= \{ S\in D_{2,e,i}: |(S\setminus e)\cap E'|=i-j \}.
\end{align*}

Here we provide some intuitive explanation of what these variables represent.

\begin{remark}
$D_{0,e,i}$ is the set of configurations of size $i$ containing $e$ all of whose vertices $u$ (but not those in $e$) are not incident with any `chosen' edges (i.e.~in $E_0$), except that in order to make this agnostic to whether the edges in $S\setminus e$ are chosen we do not consider the edges of $S\setminus e$ in this definition). So this essentially counts the configurations of size $i$ containing $e$ none of whose edges are deleted by incident chosen edges. 
\end{remark}

\begin{remark}
 $D_{0,e,i}'$ counts the subset of $D_{0,e,i}$ not in $F_{0,e}$ where we recall that $F_{0,e}$ would be the edges $f$ that are in nearly complete configuration without $e$ (and so would be deleted for being `dangerous') or that lost their equalizing coin flip (and so would also be deleted). So this essentially counts the configurations of size $i$ containing $e$ that survive except that we do not yet know which edges of $S\setminus e$ have been chosen.
\end{remark}

\begin{remark}
Thus $D_{0,e,i,j}'$ counts the number of configurations from $D_{0,e,i}'$ that will then shrink to size $j$.
\end{remark}

\begin{remark}
$D_{2,e,i}$ is the set of configurations of size $i$ containing $e$ such that some vertex $u$ of $S$ (but not one in $e$)  is in $V_2$ (the set of unearthed vertices those that were covered but remain due to edges being `unchosen' due to being incident with another chosen edge or being dangerous or losing their equalizing coin flip). 
\end{remark}

\begin{remark}
Thus $D_{2,e,i,j}$ is the subset of $D_{2,e,i}$ where precisely $i-j$ of the edges are in $E'$ (the actual final set of chosen edges) and hence shrink to size $j$.     
\end{remark}

We now proceed with the proof. Let
$$w'(e) := \sum_{i\ge 2} \sum_{2\le j\le i} \left(|D_{0,e,i,j}'|+|D_{2,e,i,j}|\right) \cdot \frac{j-1}{(D')^{j-1}}.$$
Note that for all $i$ and $j$ with $2\le j \le i\le g$, we have
$$D_{e,i,j} \subseteq D_{0,e,i,j}'\cup D_{2,e,i,j},$$ 
and hence
$$w(e)\le w'(e).$$

Fix $S\in E(H)$ with $e\in S$ and $|S|=i$. We now proceed to calculate upper bounds on the probability that $S$ is in each of $D_{0,e,i,j}$, $D_{0,e,i,j}'$ and $D_{2,e,i,j}$ as follows. 

\begin{claim}\label{cl:D0Prob}
\begin{align*}
\Prob{S\in D_{0,e,i,j}'} &\le \binom{i-1}{i-j} p^{i-j} \cdot\left( \frac{D'}{D} e^{-\varepsilon}\right)^{i-1} \cdot e^{2rg\gamma\varepsilon}. 
\end{align*}
    
\end{claim}
\begin{proofclaim}
Note we have that
$$\Prob{S\in D_{0,e,i}} = (1-p)^{\left|\bigcup_{u\in V_G(S)\setminus e} N_G'(u)\setminus (S\setminus e)\right|}.$$
Yet since $G$ has codegrees at most $4\log^2 D$ by condition (G2), we find that
\begin{align*}
\left|\bigcup_{u\in V_G(S)\setminus e} N_G'(u)\setminus (S\setminus e)\right| &\ge (i-1)rD(1-\gamma) - \binom{i-1}{2}(4\log^2 D),\\ 
&\ge (i-1)rD - rg\gamma D,
\end{align*}
since $4g^2 \log^2 D \le \gamma D$ as $\gamma \ge \varepsilon \ge D^{-1/6r}$ by assumption and $D$ is large enough. Thus
\begin{align*}
\Prob{S\in D_{0,e,i}} &\le (1-p)^{(i-1)rD -rg\gamma D} \\
&\le e^{-\varepsilon(i-1)r+rg\gamma\varepsilon}.    
\end{align*}

Fix $u\in S\setminus e$. Let 
$$T_u := \{T\in N_H'(u): V_G(T)\cap (V_G(S)\setminus (e\cup u)) \ne \emptyset\} \cup \{S\}.$$ Now using condition $(G+H)$ and $|S|\le g$, we find that
$$|T_u| \le 2^{2(g-1)}\cdot rg^2\log^{2(g-1)} D,$$
and hence
$$p\cdot |T_u| \le \frac{\varepsilon}{D}\cdot 2^{2(g-1)}\cdot rg^2 \log^{2(g-1)} D \le \frac{\gamma\varepsilon}{3} \le \frac{1}{3},$$
since $3\cdot 2^{2(g-1)}\cdot rg^2 \log^{2(g-1)} D \le \gamma D$ as $\gamma \ge \varepsilon \ge D^{-1/6r}$ by assumption and $D$ is large enough, and we also used that $\gamma \varepsilon \le 1$ by assumption. Thus
\begin{align*}
\Prob{u\not\in F_{0,e}~|~S\in D_{0,e,i}} &= \Prob{\gamma_u =1} \cdot \Prob{\nexists T\in N_H'(u)\setminus T_u: T\setminus u\subseteq E_0}\\
&= \Prob{\gamma_u =1} \cdot \prod_{T\in N_H'(u)\setminus T_u} \left(1-p^{|T|-1}\right)\\
&= \frac{\Prob{u\not\in F_0}}{\prod_{T\in T_u} \left(1-p^{|T|-1}\right)}\\
&\le \frac{1-\varepsilon w}{1-p\cdot |T_u|}\\
&\le 1 - \varepsilon w + \frac{\gamma\varepsilon/3}{1-(\gamma\varepsilon/3)}\\
&\le  e^{- \varepsilon w} + (\gamma \varepsilon/2)\\
&\le  e^{- \varepsilon w + \gamma\varepsilon},
\end{align*}
where we used that $p\cdot |T_u| \le \gamma\varepsilon/3 \le 1/3$ from above and $e^{-\varepsilon w}\ge 1/2$ since $\varepsilon w \le \gamma \le 1/2$ by assumption.

As $H$ has girth at least four, we have that the events $$( (u\not\in F_{0,e}~|~S\in D_{0,e,i}): u\in S\setminus e)$$ are independent and hence
\begin{align*}
\Prob{S\in D_{0,e,i}'} &= \Prob{S\in D_{0,e,i}} \cdot \prod_{u\in S\setminus e} \Prob{u\not\in F_{0,e}~|~S\in D_{0,e,i}} \\
&\le e^{-\varepsilon(i-1)r+rg\gamma\varepsilon} \cdot e^{- \varepsilon(i-1) \cdot w+(g-1)\gamma\varepsilon}\\
&= e^{-\varepsilon(i-1)(r+w)+2rg\gamma\varepsilon}\\
&= \left( \frac{D'}{D} e^{-\varepsilon}\right)^{i-1} \cdot e^{2rg\gamma\varepsilon}. 
\end{align*}

Since the events $(u\in E_0: u\in S\setminus e)$ are independent of each other and the event $S\in D_{0,e,i}'$, we find that
\begin{align*}
\Prob{S\in D_{0,e,i,j}'} &\le \binom{i-1}{i-j} p^{i-j} \cdot (1-p)^{j-1} \cdot \Prob{S\in D_{0,e,i}}\\
&\le
\binom{i-1}{i-j} p^{i-j} \cdot\left( \frac{D'}{D} e^{-\varepsilon}\right)^{i-1} \cdot e^{2rg\gamma\varepsilon}. 
\end{align*}
as desired.
\end{proofclaim}

\begin{claim}\label{cl:D2Prob}
\begin{align*}
\Prob{S\in D_{2,e,i,j}} &\le  \binom{i-1}{i-j} p^{i-j}  \cdot \gamma\varepsilon.
\end{align*}    
\end{claim}
\begin{proofclaim}
By the union bound,
\begin{align*}
\Prob{S\in D_{2,e,i}} \le \sum_{u\in V(S)\setminus e} \Prob{u\in V_2} \le r(i-1)\varepsilon^2 \cdot (4r+2w) \le \gamma\varepsilon,    
\end{align*}
since $rg(4r+2w)\varepsilon \le \gamma$ by assumption.
Note if $f\in (S\setminus e)\cap E'$, then $f\in E_0$ and $f\cap V_2=\emptyset$. Hence the events $f\in (S\setminus e)\cap E'$ and $f\cap V_2\ne \emptyset$ are disjoint (and hence negatively correlated). Thus 
\begin{align*}
\Prob{S\in D_{2,e,i,j}} &\le \Prob{|(S\setminus e)\cap E_0|= i-j} \cdot \Prob{S\in D_{2,e,i}}\\
&\le \binom{i-1}{i-j} p^{i-j}  \cdot \gamma\varepsilon.
\end{align*}
as desired.
\end{proofclaim}

\begin{claim}
\begin{align*}
\Expect{|D_{0,e,i,j}'| + |D_{2,e,i,j}|} &\le d_{H,i}(e) \cdot \binom{i-1}{i-j} p^{i-j} \cdot \left( \frac{D'}{D} e^{-\varepsilon}\right)^{i-1} \cdot e^{3rg\gamma\varepsilon}
\end{align*}    
\end{claim}
\begin{proofclaim}
Using Claims~\ref{cl:D0Prob} and~\ref{cl:D2Prob}, we have by the union bound that
\begin{align*}
\Prob{S\in D_{0,e,i,j}'\cup D_{2,e,i,j}} &\le \binom{i-1}{i-j} p^{i-j} \cdot \Bigg(\left( \frac{D'}{D} e^{-\varepsilon}\right)^{i-1} \cdot e^{2rg\gamma\varepsilon} +  \gamma\varepsilon\Bigg). \\
&\le \binom{i-1}{i-j} p^{i-j} \cdot \left( \frac{D'}{D} e^{-\varepsilon}\right)^{i-1} \cdot e^{3rg\gamma\varepsilon},
\end{align*}
since $\left( \frac{D'}{D} e^{-\varepsilon}\right)^{i-1} \ge 1/2$ which follows since $g\varepsilon(r+w) \le \gamma \le 1/2$ by assumption and we also used that $g\ge 2$.
Thus by linearity of expectation, we have
\begin{align*}
\Expect{|D_{0,e,i,j}'| + |D_{2,e,i,j}|} &\le d_{H,i}(e) \cdot \binom{i-1}{i-j} p^{i-j} \cdot \left( \frac{D'}{D} e^{-\varepsilon}\right)^{i-1} \cdot e^{3rg\gamma\varepsilon}
\end{align*}
as desired.
\end{proofclaim}

Hence
\begin{align*}
\Expect{|D_{0,e,i,j}'| + |D_{2,e,i,j}|} \cdot \frac{j-1}{(D')^{j-1}} &\le d_{H,i}(e) \cdot \binom{i-1}{i-j} \left(\frac{\varepsilon}{D}\right)^{i-j} \cdot \left( \frac{D'}{D} e^{-\varepsilon}\right)^{i-1} \cdot e^{3rg\gamma\varepsilon} \cdot \frac{j-1}{(D')^{j-1}} \\
&= d_{H,i}(e)\cdot \frac{i-1}{D^{i-1}} \cdot \binom{i-2}{j-2} \left(\varepsilon\cdot \frac{D'}{D}\right)^{i-j} \cdot e^{-(i-1)\varepsilon+3rg\gamma\varepsilon}.
\end{align*}
Thus summing over $j$ and using the Binomial Theorem, we have
\begin{align*}
\sum_{2\le j\le i} \Expect{|D_{0,e,i,j}'| + |D_{2,e,i,j}|} \cdot \frac{j-1}{(D')^{j-1}} &\le d_{H,i}(e) \cdot \frac{i-1}{D^{i-1}} \cdot \left(1+\varepsilon\cdot \frac{D'}{D}\right)^{i-2} \cdot e^{-(i-1)\varepsilon+3rg\gamma\varepsilon}\\
&\le d_{H,i}(e) \cdot \frac{i-1}{D^{i-1}} \cdot e^{(i-2)\varepsilon} \cdot e^{-(i-1)\varepsilon+3rg\gamma\varepsilon} \\
&= d_{H,i}(e) \cdot \frac{i-1}{D^{i-1}} \cdot e^{-\varepsilon+3rg\gamma\varepsilon} \\
&\le d_{H,i}(e) \cdot \frac{i-1}{D^{i-1}} \cdot \left(1-\frac{3}{4} \cdot \varepsilon\right),
\end{align*}
where for the last inequality we used that $3rg\gamma \le 1/8$ by assumption and that $e^{-(7/8)\varepsilon} \le 1 - (3/4)\varepsilon$ since $\varepsilon \le \gamma \le 1/8$ by assumption. Hence summing over $i$ and $j$ and using the linearity of expectation, we find
$$\Expect{w'(e)} = \sum_{i\ge 2} \sum_{2\le j\le i} \Expect{|D_{0,e,i,j}'| + |D_{2,e,i,j}|} \cdot \frac{j-1}{(D')^{j-1}} \le w \cdot \left(1-\frac{3}{4} \cdot \varepsilon\right).$$

\subsection{Concentrations}\label{ss:Conc}

For the purposes of applying the Lov\'asz Local Lemma as well as for exceptional outcomes used in the various concentrations, it is helpful to define an auxiliary graph for the various events as follows.

Let $J$ be the graph with $V(J) := [m] \cup V(G) \cup E(G)$ and 
\begin{align*}
E(J) := &\{iv: v\in X_i, v\in V(G), i\in [m]\} \cup \{ve: v\in e, v\in V(G),~e\in E(G)\} \\
&\cup \{ef: e,f\in E(G), e\cap f\ne \emptyset\} \cup \{ef: e,f\in E(G), ~\exists S\in E(H) \textrm{ with }e,f\in S\}.    
\end{align*}

Note that 
$$\Delta(J) \le \max \left\{2D^{1/r}, D(1+\gamma)+1, (D(1+\gamma)+1)r + \sum_{i\le g} D^i \log D\right\} \le D^{g+1},$$
where the last inequality follows since $D$ is large enough. Hence for each $x\in V(J)$ and nonnegative integer $d$, we have that
$$|y\in V(J): {\rm dist}_J(x,y)\le d\}| \le \sum_{0\le i \le d} \Delta(J)^i \le D^{(d+1)(g+1)},$$
where the last inequality follows since $D\ge 2$.

\subsubsection{Exceptional Outcomes}

For a vertex $v$ of $G$, let $S_v$ be the exceptional outcome of $v$ being incident with at least $\log^2 D$ edges of $E_0$. 

Note that $|N_G'(v) \cap E_0|$ is $(1,1)$-observable. Yet $$\Expect{|N_G'(v) \cap E_0|} = p|\{e\in E(G): v\in e\}| \le pD(1+\gamma) = \varepsilon(1+\gamma) \le 1,$$
since $\varepsilon \le \gamma \le 1/2$ by assumption. Hence by Talagrand's Inequality, 
$$\Prob{S_v} \le 4 e^{-(\log^2 D)/40}.$$

For an edge $e$ of $G$, let $T_e$ be the exceptional outcome of $e$ being in at least $\log^2 D$ otherwise completed configurations, that is
$$|S\in N_H'(e): S\setminus e \subseteq E_0| \ge \log^2 D.$$
Note from the calculations before for the equalizing coin flip, we have that 
$$\Expect{|S\in N_H'(e): S\setminus e \subseteq E_0|} \le \varepsilon w \le 1,$$
since $\varepsilon w \le \gamma \le 1$ by assumption. Since $H$ has girth at least three, it follows that this variable is $(g-1,1)$-observable. Hence by Talagrand's Inequality,
$$\Prob{T_e} \le 4 e^{-(\log^2 D)/(40(g-1))}.$$

We will use the exceptional outcomes above to bound the observation degree for an edge verifying $E_2$ and $E_0\cap F_0$; we will separately need the assumptions on the codegree of $G$ with $H$ and that $H$ has girth at least five to help bound the observation degree for $F_0$.

\subsubsection{Concentration of $|X_i\setminus V'|$}

Fix $i\in [m]$. %We prove $\Prob{A_i}$ is exponentially small in $D$ via Linear Talagrand's.

Note that $|X_i\cap V_1|$ is $(1,r)$-observable; namely we can verify a vertex $x\in X_i\cap V_1$ by showing an edge $e$ in $E_0$ incident to $x$ while an edge $e$ has observation degree at most $r$ since it is incident with at most $r$ vertices.

Let $B_i := \{ u\in V(G): {\rm dist}_J(i,u)\le 3\}$ and let $C_i := \{ f\in E(G): {\rm dist}_J(i,f)\le 3\}$. 
Note that 
$$|B_i|+|C_i| \le |\{y\in V(J): {\rm dist}_J(i,y)\le 3\}| \le D^{4(g+1)}.$$
Now we define a set of exceptional outcomes 
$$\Omega_i^* := \bigcup_{u\in B_i} S_u \cup \bigcup_{f\in C_i} T_f.$$ 
Hence 
$$\Prob{\Omega_i^*} \le D^{4(g+1)} \cdot e^{-(\log^2 D)/(40(g-1))}.$$

Now we have that $|X_i \cap V_2|$ is $(g,2r(r+g)\log^2 D)$-observable with respect to $\Omega_i^*$. Namely to verify that $x\in X_i\cap V_2$, we show an edge $e\in E_0$ incident with $x$ and then either an edge $f\in E_0$ incident to $e$, or that $\gamma_e=0$, or the set of at most $g-1$ edges of $G$ in an edge $S$ of $H$ with $e\in S$ and $S\setminus e \subseteq E_0$. While for $\omega \in \Omega\setminus \Omega_i^*$, the observation degree of a trial is at most $r+r^2\log^2 D + rg \log^2 D\le 2r(r+g)\log^2 D$ since $D$ is large enough.

Moreover, note that 
$$|X_i\cap V_2|\le |X_i\cap V_1| \le |X_i| \le 2\cdot D^{1/r}.$$ 
Also recall that 
$$\Expect{|X_i\setminus V'|} \ge |X_i|\cdot \left(\frac{3}{4}\right) \cdot \varepsilon.$$ Since $|X_i\setminus V'| = |X_i\cap V_1| - |X_i \cap V_2|$ and each of $|X_i\cap V_1|, |X_i\cap V_2|$ is $(g,2r(r+g)\log^{2} D)$-observable with respect to $\Omega_i^*$, we have by Theorem~\ref{exceptional talagrand's difference} with $t= |X_i|\cdot \varepsilon/4$, $m=2$ and $M=|X_i|$ that
\begin{align*}
\Prob{A_i} &\le 8 \exp\left( \frac{-\frac{\varepsilon^2}{16}\cdot  |X_i|}{8g(2r(r+g)\log^{2} D)(2)^2(5)}\right) + 4\cdot (2)\cdot \Prob{\Omega_i^*}\\
&\le e^{-\log^{1.5}D},
\end{align*}
where for the last inequality we also used that $\varepsilon \ge D^{-1/6r}$ by assumption, that $|X_i|\ge D^{1/r}$ and that $D$ is large enough.

\subsubsection{Concentration of $|D_v|$}

Let $B_v := \{u\in V(G): {\rm dist}_J(v,u) \le 3\} $ and $C_v := \{f\in E(G): {\rm dist}_J(v,f) \le 3\}$. Note that 
$$|B_v|+|C_v| \le |\{y\in V(J): {\rm dist}_J(v,y)\le 3\}| \le D^{4(g+1)}.$$
Now we define a set of exceptional outcomes
$$\Omega_v^* := \bigcup_{u\in B_v} S_u \cup \bigcup_{f\in C_v} T_f.$$
Hence
$$\Prob{\Omega_v^*} \le D^{4(g+1)} \cdot e^{-\log^2 D/ (40(g-1))}.$$
Let 
$$D_{1,v} := \{e\in E(G):v\in e, (e\setminus v)\cap V_1\ne \emptyset\}.$$
Note that 
$$|D_{0,v}'| = d_G(v)-|D_{1,v}|-|N_G'(v)\cap F_0| + |D_{1,v}\cap F_0|.$$
Now $d_G(v)$ is a constant (not a random variable) and $|D_{1,v}|$ is $(1,1)$-observable. 

Meanwhile, $|N_G'(v)\cap F_0|$ is $(g-1, 2^{2(g-1)}\cdot \log^{2(g-1)} D)$-observable; namely to verify $e\in N_G'(v)$ is in $F_0$, we show either $\gamma_e =0$ or a set of at most $g-1$ edges of $G$ in an edge $S$ of $H$ with $e\in S$ and $S\setminus e \subseteq E_0$. Meanwhile, the observation degree of a trial is at most 1 for coinflips and edges incident with $v$, while for an edge $f$ not incident with $v$ it is at most the codegree of $v$ and $f$ which is at most $2^{2(g-1)}\cdot \log^{2(g-1)} D$ by condition $(G+H)$.

Similarly $|D_{1,v}\cap F_0|$ is $(g, 2^{2(g-1)}\cdot \log^{2(g-1)} D)$-observable. Hence by Theorem~\ref{exceptional talagrand's difference} with $t= \gamma\varepsilon\cdot D'$, $m=3$ and $M=d_G(v)\le 2D \le 4D'$, we have that
\begin{align*}
\Prob{\Big| |D_{0,v}'|- \Expect{|D_{0,v}'|} \Big| \ge t} &\le 12\cdot \exp\left( \frac{-\gamma^2\varepsilon^2 \cdot D'}{8g(2^{2(g-1)}\cdot \log^{2(g-1)} D)(3)^2(17)}\right)\\
&\le \frac{1}{2} \cdot e^{-\log^{1.5}D},
\end{align*}
where for the last inequality we also used that $D' \ge D/2$ since $\varepsilon(r-1+w)\le 1/2$ by assumption, that $\gamma \ge \varepsilon \ge D^{-1/6}$ by assumption and that $D$ is large enough.

On the other hand, $D_{2,v}$ is $(g,\log^5 D)$-observable with respect to $\Omega_v^*$. Namely to verify that for an edge $e\in N'(v)$ we have $(e\setminus v)\cap V_2\ne \emptyset$, we show an edge $f\in E_0$ incident with some $u\in (e\setminus v)$ and then either an edge $f'\in E_0$ incident to $f$, or that $\gamma_f=0$, or the set $F$ of at most $g-1$ edges of $G$ in an edge $S$ of $H$ with $f\in S$ and $S\setminus f \subseteq E_0$. Meanwhile, for $\omega \in \Omega\setminus \Omega_i^*$, the observation degree of a trial is at most 
$$(1+4r\log^2 D) \cdot r(r+g)\log^2 D \le \log^5 D$$
where the inequality follows since $D$ is large enough. This upper bound for the observation degree follows as an edge $f''$ can be used as $f'$ for at most $r^2\log^2 D$ edges $f$ since $\omega\not\in \{S_u:u\in f''\}$ and in $F$ for at most $rg\log^2 D$ edges $f$ at most as $\omega\not\in T_{f''}$; and yet $f$ can be used at most $r\cdot (4\log^2 D)$ by condition (G2). 

Hence by Theorem~\ref{exceptional talagrand's observations} with $t= \gamma\varepsilon\cdot D'$, we have that
\begin{align*}
\Prob{\Big| |D_{2,v}|- \Expect{|D_{2,v}|} \Big| \ge t} &\le 4 \exp\left( \frac{-\gamma^2\varepsilon^2 \cdot D'}{8g(\log^{5} D)(17)}\right) + 4\Prob{\Omega_v^*}\\
&\le \frac{1}{2} \cdot e^{-\log^{1.5}D},
\end{align*}
where for the last inequality we also used that $D' \ge D/2$ since $\varepsilon(r-1+w)\le 1/2$ by assumption, that $\gamma \ge \varepsilon \ge D^{-1/6}$ by assumption and that $D$ is large enough.

%Recall that $\Expect{|D_v|} = D'(1 \pm  4r\gamma \varepsilon).$
It now follows that for $v\in V(G')$
\begin{align*}
\Prob{A_v} &= \Prob{|D_v| \ne D'(1\pm \gamma(1+6r\varepsilon))}\\
&\le \Prob{ |D_{0,v}'| \ne \Expect{|D_{0,v}'|} \pm D'\gamma\varepsilon} + \Prob{ |D_{2,v} \ne  \Expect{|D_{2,v}|} \pm D'\gamma\varepsilon }\\
&\le e^{-\log^{1.5}D}.
\end{align*}

\subsubsection{Concentration of $w'(e)$}
Let $B_e := \{u\in V(G): {\rm dist}_J(e,u) \le 3\} $ and $C_e := \{f\in E(G): {\rm dist}_J(e,f) \le 3\}$. Note that 
$$|B_e|+|C_e| \le |\{y\in V(J): {\rm dist}_J(e,y)\le 3\}| \le D^{4(g+1)}.$$
Now we define a set of exceptional outcomes
$$\Omega_e^* := \bigcup_{u\in B_e} S_u \cup \bigcup_{f\in C_e} T_f.$$
Hence
$$\Prob{\Omega_e^*} \le D^{4(g+1)} \cdot e^{-(\log^2 D)/ (40(g-1))}.$$
Let 
$$N'_H(e,i,j) := \{ S\in N'_H(e): |(S\setminus e)\cap E_0|=i-j \},$$
$$D_{1,e,i,j} := \{ S\in N'_H(e): (V_G(S\setminus E_0)\setminus e)\cap V_1 \ne \emptyset, |S|=i, |(S\setminus e)\cap E_0|=i-j \},$$
and
$$F_{e,i,j}:= \{ S\in N'_H(e): (S\setminus e) \cap F_{0,e} \ne \emptyset, |S|=i, |(S\setminus e)\cap E_0|=i-j \}$$
Note that 
$$|D_{0,e,i,j}'| = |\{ S\in N'_H(e): |(S\setminus e)\cap E_0|=i-j \}|-|D_{1,e,i,j}|-|F_{e,i,j}| + |D_{1,e,i,j}\cap F_{e,i,j}|.$$

Now $|N'_H(e,i,j)|$ is $(g-1,1)$-observable; namely to verify $S\in N_H'(e)$ satisfies $|(S\setminus e)\cap E_0|=i-j$, we show the at most $g-1$ trials consisting of whether $f\in S\setminus e$ is in $E_0$, while the observation degree is at most $1$ since every edge $f\ne e$ is in at most one edge of $N_H'(e)$ since $H$ has girth at least three.

Moreover, $|D_{1,e,i,j}|$ is $\left(g,~2^{2(g-1)}\cdot r\cdot \log^{2(g-1)} D\right)$-observable; namely, to verify $S\in N_H'(e)$ is in $D_{1,e,i,j}$, we show the at most $g-1$ trials consisting of whether $f\in S\setminus e$ is in $E_0$ as well as one edge $f'\in E_0$ with $f\cap (V_G(S\setminus E_0)\setminus e) \ne \emptyset$. Meanwhile the observation degree is at most $2^{2(g-1)}\cdot r\cdot \log^{2(g-1)} D$ since for every edge $f\ne e$, there are at most $2^{2(g-1)}\cdot r\cdot \log^{2(g-1)} D$ edges in $N_H'(e)$ such that $V_G(S\setminus e)\cap f \ne \emptyset$ by condition $(G+H)$.

Meanwhile, $|F_{e,i,j}|$ is $(2(g-1), 1)$-observable; namely to verify $S\in N_H'(e)$ is in $F_{e,i,j}$, we show the at most $g-1$ trials consisting of whether $f\in S\setminus e$ is in $E_0$, as well as either an edge $f\in S\setminus e$ with $\gamma_f=0$ or the at most $g-1$ edges of $T\setminus f$ for some $f\in S\setminus e$ where $T\in N_H'(f)\setminus S$ and $T\setminus f \subseteq E_0$. Meanwhile the observation degree of a trial is at most $1$ for an edge $f' \ne e$ as $f'$ is in at most one edge of $N_H'(e) \cup \bigcup_{S\in N_H'(e)} \bigcup_{f\in S\setminus e} N_H'(f)$ since $H$ has girth at least five.

Similarly, it follows that $|D_{1,e,i,j}\cap F_{e,i,j}|$ is $\left(2g-1, 1+2^{2(g-1)}\cdot r\cdot \log^{2(g-1)} D\right)$-observable. 

On the other hand, $D_{2,e,i,j}$ is $(2g-1,\log^{2g+1} D)$-observable with respect to $\Omega_e^*$. Namely to verify $S\in N_H'(e)$ is in $D_{2,e,i,j}$, we show the at most $g-1$ trials consisting of whether $f\in S\setminus e$ is in $E_0$, as well as an edge $f\in E_0$ incident with some $u\in V_G(S)\setminus e$ and then either an edge $f'\in E_0$ incident to $f$, or that $\gamma_f=0$, or the set $F$ of at most $g-1$ edges of $G$ in an edge $S'$ of $H$ with $f\in S'$ and $S'\setminus f \subseteq E_0$. While for $\omega \in \Omega\setminus \Omega_e^*$, the observation degree of a trial is at most 
$$(1+r\cdot 2^{2(g-1)}\cdot \log^{2(g-1)} D) \cdot (r+g)\log^2 D \le \log^{2g+1} D;$$ 
where the inequality follows since $D$ is large enough. This upper bound for the observation degree follows as an edge $f''$ can be used as $f'$ for at most $r\log^2 D$ edges $f$ since $\omega\not\in \{S_u:u\in f''\}$ and in $F$ for at most $g\log^2 D$ edges $f$ as $\omega\not\in T_{f''}$; and yet $f$ can be used at most $r\cdot 2^{2(g-1)}\cdot \log^{2(g-1)} D$ times by condition $(G+H)$.

Note that if a random variable is $(r,d)$-observable with respect to $\Omega_e^*$ then it is $(r',d')$-observable with respect to $\Omega_e^*$ for all $r'\ge r, d'\ge d$. Thus for all $i,j$ where $2\le j \le i \le g$, the above variables are $\left(2g-1,\log^{2g+1} D\right)$-observable with respect to $\Omega_e^*$. Hence if any of the above variables are scaled by a factor of $\frac{j-1}{(D')^{j-1}}$ as they are in $w'(e)$, we find using Proposition~\ref{prop:scaling} that the scaled variable is $\left(2g-1,\frac{j-1}{(D')^{j-1}} \cdot \log^{2g+1} D\right)$-observable with respect to $\Omega_e^*$ and hence is $\left(2g-1,\frac{\log^{2g+1} D}{D'}\right)$-observable with respect to $\Omega_e^*$ since $\frac{j-1}{(D')^{j-1}}\le \frac{1}{D'}$ for all $j\ge 2$ as $D'\ge 2$.

Thus $w'(e)$ is of the form $\sum_{k=1}^m \varepsilon_k X_k$ where $m\le 5g^2$, $\varepsilon_k \in \{-1,1\}$ for all $k\in [m]$ and each $X_k$ is a random variable that is $\left(2g-1,\frac{\log^{2g+1} D}{D'}\right)$-observable with respect to $\Omega_e^*$.

Finally note that for each $k\in [m]$, we have that 
$$\Expect{X_k} \le \max_{2\le j \le i \le g} \frac{j-1}{(D')^{j-1}} \cdot \Expect{|N_H'(e,i,j)|}.$$ 
Yet $$\Expect{|N_H'(e,i,j)|} = d_{H,i}(e) \binom{i-1}{j-1} p^{i-j}(1-p)^{j-1} \le \frac{D^{i-1}}{i-1} \cdot w \cdot 2^{i} \left(\frac{\varepsilon}{D}\right)^{i-j} \le \frac{D^{j-1}}{i-1} \cdot w \cdot 2^i,$$
where we used that $\varepsilon \le 1$, $0\le p\le 1$, $\binom{i-1}{j-1} \le 2^i$ and $d_{H_i}(e) \le \frac{D^{i-1}}{i-1} \cdot w$. 
Hence for all $2\le j \le i \le g$, we find that
$$ \frac{j-1}{(D')^{j-1}} \cdot \Expect{|N_H'(e,i,j)|} \le \left(\frac{D}{D'}\right)^{j-1} \cdot w \cdot 2^i \le 4^g \cdot w,$$
where we used that $D'\ge D/2$ since $\varepsilon(r-1+w)\le 1/2$ by assumption.
Let $M := 4^g \cdot w$. Thus for all $k\in [m]$, we have that
$$\Expect{X_k} \le M.$$

Hence by Theorem~\ref{exceptional talagrand's difference} with $t=\frac{\varepsilon}{4}\cdot w$, we find that
\begin{align*}
\Prob{\Big| |w'(e)|- \Expect{|w'(e)|} \Big| \ge t} &\le 4\cdot (5g^2)\cdot \exp\left( -\frac{ \left(\frac{\varepsilon^2}{4^2} \right)\cdot w}{8(2g-1)\left(\frac{\log^{2g+1} D}{D'}\right)(5g^2)^2(5 \cdot 4^g)}\right) + 4\cdot (5g^2)\cdot \Prob{\Omega_e^*}\\
&\le e^{-\log^{1.5}D},
\end{align*}
where we used that $\varepsilon \ge D^{-1/6}$ and $w\ge D^{-1/6}$ by assumption, that $D'\ge D/2$ since $\varepsilon(r-1+w)\le 1/2$ by assumption, and that $D$ is large enough.

\subsection{Applying the Local Lemma}\label{ss:LocalLemma}

Recall the definition of the graph $J$ from Section~\ref{ss:Conc}.

Note now that the event $A_i$ depends only on trials $(f\in E_0, \gamma_f)$ for edges $f$ of $G$ where ${\rm dist}_{J}(i,f) \le 3$. Similarly the event $A_v$ depends only on trials $(f\in E_0, \gamma_f)$ for edges $f$ of $G$ where ${\rm dist}_{J}(v,f) \le 3$ and the event $A_e$ depends only on trials $(f\in E_0, \gamma_f)$ for edges $f$ of $G$ where ${\rm dist}_{J}(e,f) \le 3$. Hence for each vertex $x$ of $J$, the event $A_x$ depends only on trials $(f\in E_0, \gamma_f)$ for edges $f$ of $G$ where ${\rm dist}_{J}(x,f) \le 3$.

For a vertex $x\in V(J)$, let 
$$I_x := \{A_y: {\rm dist}_J(x,y) \ge 7\}.$$
Hence for each $x\in V(J)$, we have that $A_x$ is mutually independent of the set of events $I_x$. Recall that for each $x\in V(J)$, we have that
$$|y\in V(J): {\rm dist}_J(x,y)\le 6\}| \le D^{7(g+1)}.$$

Yet from the concentration section above, we have that for each $A_x\in\mathcal{E}$
$$\Prob{A_x} \le e^{-\log^{1.5} D} \le \frac{1}{4D^{7(g+1)}},$$
since $D$ is large enough. Hence by the Lov\'asz Local Lemma, with positive probability none of the events in $\mathcal{E}$ happen. But in that instance then, all of (2), (3) and (4) hold as desired.

\section{Proof of Iterative Bipartite Lemma and Finishing Lemma}\label{s:Split}

In this section, we prove Lemmas~\ref{lem:IterativeBip} and~\ref{lem:Finish}. To prove Lemma~\ref{lem:IterativeBip}, we need a regularization lemma as follows.

\begin{lem}\label{lem:reg}
If $G=(A,B)$ is an $r$-uniform bipartite hypergraph where every vertex in $A$ has degree exactly $D_A$ and every vertex in $B$ has degree at most $D_B$ (where $D_B$ is an integer), then there exists an $r$-uniform bipartite hypergraph $G'=(A',B')$ where every vertex of $A'$ has degree $D_A$ in $G'$ and every vertex of $B'$ has degree $D_B$ in $G'$ and $G$ is an induced subhypergraph of $G'$ where $A\subseteq A'$ and $B\subseteq B'$.
\end{lem}
\begin{proof}
Let $\delta_G(B)$ denote the minimum degree in $G$ of vertices in $B$. We proceed by induction on $D_B-\delta_G(B)$. If $D_B=\delta_G(B)$, then $G$ is as desired. So we assume that $\delta_G(B) < D_B$.

Let $X=\{x_1,\ldots, x_m\}$ be the set of vertices of $B$ with degree exactly $\delta_G(B)$. Let $G_0=(A_0,B_0)$ where 
$$A_0:= (A\times [D_A\cdot (r-1)])\cup \{a_i:~i\in[m]\},~~B_0:=B\times [D_A\cdot (r-1)]$$ 
and 
\begin{align*}
E(G_0) =&\bigcup_{e\in E(G)} \Big\{ \{(v,i): v\in e\}: i\in[D_A\cdot (r-1)]\Big\}~\cup~\\
&\bigcup_{i\in [m]} \Big\{\{a_i\}\cup \{(x_i,(k-1)(r-1)+j):j\in [r-1]\}:  k\in[D_A] \Big\}.    
\end{align*}
That is, $G_0$ is obtained from the disjoint union of $D_A\cdot (r-1)$ copies of $G$ together with a set of new vertices $\{a_1,\ldots,a_{|X|}\}$ and a set of new edges on those vertices and the copies of $X$. Note now that every vertex in $A_0$ has degree exactly $D_A$ in $G_0$ by construction. Moreover, every vertex of $X\times [D_A\cdot (r-1)]$ has degree $\delta_G(B)+1$ in $G_0$ and hence every vertex in $B_0$ has degree at least $\delta_G(B)+1$ in $G_0$. That is $\delta_{G_0}(B_0) = \delta_G(B)+1$.

Thus $G_0=(A_0,B_0)$ is an $r$-uniform bipartite hypergraph where every vertex in $A_0$ has degree exactly $D_A$ and $\delta_{G_0}(B_0) > \delta_G(B)$. So by induction, there exists an $r$-uniform bipartite hypergraph $G'=(A',B')$ where every vertex of $A'$ has degree $D_A$ in $G'$ and every vertex of $B'$ has degree $D_B$ in $G'$ and $G_0$ is an induced subhypergraph of $G'$ where $A_0\subseteq A'$ and $B_0\subseteq B'$. But then $G'$ is as desired.
\end{proof}

We now explain how to modify the proof of Lemma~\ref{lem:Iterative} to prove Lemma~\ref{lem:IterativeBip}.

\begin{proof}[Proof of Lemma~\ref{lem:IterativeBip}]
We may assume without loss of generality (by deleting edges as necessary) that every vertex in $A$ has degree exactly $\lceil D_A \rceil$. %(Ceiling since not integer or just don't care?!) 
Then by Lemma~\ref{lem:reg}, we may assume without loss of generality that every vertex in $B$ has degree exactly $\lceil D \rceil$. %(Same Problem!!)
Note for the remainder of the proof, we drop the ceilings for $D$ and $D_A$ since the resulting error terms do not substantially affect the proofs.

Let $A':= A\cap V(G')$ and $B':= B\cap V(G')$. We then make the following changes to the proof of Lemma~\ref{lem:Iterative}:

\vskip.25in
\noindent {\bf Bad Events for (3a):} For each $a\in A$, let $A_a$ be the event that $|D_a| < D_A'$. Note that when $a\in A'$, then $d_{G'}(a) = |D_a|$.
\vskip.25in

\noindent {\bf Bad Events for (3b):} For each $b\in B$, let $A_b$ be the event that $|D_b| > D'$. Note that when $b\in B'$, then $d_{G'}(b) = |D_b|$.
\vskip.25in

\noindent {\bf Basic Calculations:}
Since we have removed outcome (2), it is no longer necessary to calculate $\Prob{v\in V_0}$ or $\Prob{v\in V_1}$ or $\Prob{v\in V_1\setminus V_2}$.

For the calculation of $\Prob{v\in V_2}$, there is a slight difference as follows: For each edge $e\in E(G)$,
$$\Prob{e\in E_2} \le ((r-1)D+D_A)\cdot p^2 \le \frac{r \cdot D_A \cdot \varepsilon^2}{D^2},$$
and hence
$$\Prob{e\in E_2\cup (E_0\cap F_0)} \le \frac{\varepsilon^2}{D} \cdot \left(r \cdot \frac{D_A}{D} + w\right).$$
Thus by the union bound (and using $d_G(v) \le D_A$)
\begin{align*}
\Prob{v\in V_2} &\le  \Prob{N'_G(v)\cap (E_2\cup (E_0\cap F_0)) \ne \emptyset} \\
&\le d_G(v)\cdot \frac{\varepsilon^2}{D} \cdot \left(r \cdot \frac{D_A}{D}+w\right)
 \\
&\le \varepsilon^2 \cdot \frac{D_A}{D} \cdot \left(r\cdot \frac{D_A}{D}+w\right).
\end{align*}

\vskip.15in

\noindent {\bf Expectation of $d_{G'}(v)$ for $v\in A$:}
\vskip.15in

Since we only desire a lower bound on $d_{G'}(v)$ here, we only calculate a lower bound on $D_{0,v}$ and $D_{0,v}'$ as follows. Here 
$$|B_e|\le (r-1)D.$$
Thus
\begin{align*}
\Prob{e\in D_{0,v}'} &= (1-\varepsilon w) \cdot (1-p)^{|B_e|}\\
&\ge e^{-\varepsilon w - \varepsilon^2 w^2} \cdot e^{-p|B_e| - p^2 |B_e|} \\
&\ge e^{-\varepsilon(r-1+ w) - \varepsilon^2 (w^2 +r-1)}\\
&\ge \frac{D_A'}{D_A} \cdot e^{\gamma\varepsilon/2}.
\end{align*}
where we used that $\varepsilon(w^2+2(r-1)) \le \gamma/2$ by assumption. Hence by linearity of expectation, we have that
$$\Expect{|D_{0,v}'|} \ge D_A' \cdot e^{\gamma\varepsilon/2}.$$

\vskip.15in

\noindent {\bf Expectation of $d_{G'}(v)$ for $v\in B$:}
\vskip.15in
Since we only desire an upper bound on $d_{G'}(v)$ here, we only calculate an upper bound on $D_{0,v}$, $D_{0,v}'$ and $D_{2,v}$ here as follows. Here 
$$|B_e|\ge (r-2)D + D_A - \binom{r-1}{2}(4\log^2 D).$$
Thus
\begin{align*}
\Prob{e\in D_{0,v}'} &= (1-\varepsilon w) \cdot (1-p)^{|B_e|}\\
&\le e^{-\varepsilon w} \cdot e^{-p|B_e|}\\
&\le e^{-\varepsilon\left(r-2 + \frac{D_A}{D} + w - \frac{\gamma}{2}\right)}\\
&= \frac{D'}{D} \cdot e^{-\gamma\varepsilon/2},
\end{align*}
where for the second to last inequality we used that $\binom{r-1}{2}\cdot \frac{4\log^2 D}{D} \le \gamma/2$ since $\gamma \ge \varepsilon \ge D^{-1/6r}$ by assumption and $D$ is large enough.
Hence by linearity of expectation, we have that
$$\Expect{|D_{0,v}'|} \le D' \cdot e^{-\gamma\varepsilon/2}.$$

Meanwhile
$$\Prob{e\in D_{2,v}} = \Prob{ (e\setminus v)\cap V_2 \ne \emptyset} \le (r-1) \varepsilon^2 \cdot \frac{D_A}{D} \cdot \left(r\cdot \frac{D_A}{D}+w\right) \le \frac{\gamma \varepsilon}{16},$$
since $\frac{\gamma}{16} \ge (r-1)\cdot \frac{D_A}{D} \cdot \left(r\cdot \frac{D_A}{D}+w\right) \cdot \varepsilon$ by assumption. By linearity of expectation, we have
$$\Expect{|D_{2,v}|} \le d_G(v) \cdot \frac{\gamma \varepsilon}{16} \le D \cdot \frac{\gamma \varepsilon}{16} \le D' \cdot \frac{\gamma\varepsilon}{8},$$
since $D' \ge D/2$ as $\varepsilon\left(r-2 + \frac{D_A}{D}+w\right)\le 1/2$ by assumption.

Combining, we have that
$$\Expect{|D_v|} \le \Expect{|D_{0,v}'|} + \Expect{|D_{2,v}|} \le D'\cdot e^{-\gamma\varepsilon/4},$$
where we used that $e^{-\gamma\varepsilon/2} \ge 1/2$ since $\gamma \varepsilon \le \gamma \le 1$ by assumption, and also that $1+\gamma\varepsilon/4 \le e^{\gamma\varepsilon/4}$.

\vskip.15in

\noindent {\bf Expectation of $w'(e)$, Concentrations, Applying the Local Lemma:}
\vskip.15in

In these sections, there are a few small discrepancies in the calculations where a $D_A$ should appear instead of a $D$, or a $D_A'$ should appear instead of a $D'$. However, these changes do not substantially affect the proof and so are left to the reader. We also note that for concentrating $|D_v|$, we should use $t=\frac{\gamma\varepsilon}{2} \cdot D_A'$ for $v\in A$ and $t=\frac{\gamma\varepsilon}{4} \cdot D'$ for $v\in B$.
\end{proof}

We now prove Lemma~\ref{lem:Finish}.

\begin{proof}[Proof of Lemma~\ref{lem:Finish}]
We may assume without loss of generality (by deleting edges as necessary) that every vertex of $A$ has degree exactly $300rD$.

For each vertex $a$ in $A$, independently choose an edge $e_a$ incident with $a$ uniformly at random. Hence for each $f\in E(G)$ with $a\in f$, we have $\Prob{e_a=f} = \frac{1}{300rD}$.

Let $\mathcal{T}$ be the set of unordered pairs $f_1,f_2 \in E(G)$ where $f_1\cap f_2 \ne \emptyset$ and $f_1\cap A \ne f_2\cap A$. For each $T\in \mathcal{T}$ where $T=\{f_1,f_2\}$, let $C_{T}$ be the bad event that $e_{f_1\cap A}=f_1 $ and $e_{f_2\cap A}=f_2$. For each edge $S \in E(H)$, let $C_S$ be the bad event that for each edge $e\in S$ and $\{a\}=e\cap A$, we have $e_{a}=e$.
 
Hence for all $T\in \mathcal{T}$, we have $$\Prob{C_{T}} = \frac{1}{(300rD)^2}.$$
Similarly, for each $S$ in $E(H)$,
$$\Prob{C_S} = \frac{1}{(300rD)^{|S|}}.$$

Let $J=(A,\mathcal{T}\cup E(H))$ be the bipartite graph 
$$E(J) := \{aS: a\in A, S\in \mathcal{T}\cup E(H), a\in V(S)\}.$$ 
Define $\mathcal{S}_i := \{S\in \mathcal{T}\cup E(H), |S|=i\}$. Let $\mathcal{E}$ be the set of all bad events.

For each $S\in \mathcal{T}\cup E(H)$, let $\mathcal{E}_{S} :=   \{C_{S'}: d_J(S, S') \le 2\}$ and let $\mathcal{E}_{S,i} := \{C_{S'}: d_J(S,S')\le 2, |S'|=i\}$. Note that  $C_{S}$ is mutually independent of $\mathcal{E} \setminus \mathcal{E}_{S}$.

Note that for each $S\in \mathcal{T}\cup E(H)$, we have $d_J(S) = |S|$. Meanwhile for each $a\in A$ and $i\in \{2,\ldots, g\}$, we have that 
\begin{align*}
|N_J(a)\cap \mathcal{S}_i| &\le 300rD \cdot \left(rD+w_D(H)\cdot D^{i-1}\right)\\
&\le 600\cdot r^2\cdot D^i,
\end{align*}
where for the last inequality we used that $w_D(H)\le 1$. Hence for each $S\in \mathcal{T}\cup E(H)$, we find that
$$|\mathcal{E}_{S,i}| \le |S|\cdot 600\cdot r^2\cdot D^i.$$

For each positive integer $i$, define
$$w(i):= \frac{1}{1200\cdot r^2\cdot i^2 \cdot D^{i}} $$
Then for each $S\in \mathcal{T}\cup E(H)$, define $x_S := w(|S|)$.
Recall that for $x\in [0,1/2]$, $1-x \ge e^{-x-x^2}\ge e^{-2x}$. Then for all $S\in \mathcal{T}\cup E(H)$, we have that
\begin{align*}
&x_S \cdot \prod_{d_J(S,S')\le 2} (1-x_{S'}) \ge x_S \cdot \prod_{d_J(S,S')\le 2} e^{-2x_{S'}} \ge x_S \cdot e^{- \sum_{d_J(S,S')\le 2} 2x_{S'}} \ge x_S \cdot e^{-\sum_{i\ge 2} 2\cdot |\mathcal{E}_{S,i}| \cdot w(i)}\\
&\ge x_S \cdot e^{-\sum_{i\ge 2} 2\cdot |S|\cdot 600\cdot r^2\cdot D^i\cdot w(i)} \ge x_S \cdot e^{-|S|\cdot \sum_{i\ge 2} \frac{1}{i^2}} \ge x_s\cdot e^{-|S|\cdot \left(\frac{\pi^2}{6}-1\right)} \ge  x_S \cdot \frac{1}{2^{|S|}} \\
&\ge \frac{1}{1200\cdot r^2\cdot |S|^2 \cdot D^{|S|}}\cdot \frac{1}{2^{|S|}} \ge \frac{1}{(300\cdot r\cdot D)^{|S|}} \ge \Prob{C_S}, 
\end{align*}
where we used that $\sum_{i\ge 2} \frac{1}{i^2} \le \frac{\pi^2}{6}-1$, that $\left(\frac{\pi^2}{6}-1\right) \cdot \frac{1}{\ln 2} \le 0.94 \le 1$, and for the last line we used that $|S|\ge 2$, that $|S|^2 \le 4^{|S|}$ for all $|S|\ge 2$, and that $1200 \le (300/8)^2$. Now by the general version of the Lov\'asz Local Lemma, it follows that none of the events $\mathcal{E}$ happen with positive probability, that is, $\{e_a: a\in A\}$ is an $H$-avoiding $A$-perfect matching of $G$ as desired.
\end{proof}

\section{Proof of Codegree Sparsification Lemma}\label{s:Generalizations}

In this section, we prove Theorem~\ref{lem:Sparsifying}. First we need some definitions and a preliminary lemma as follows.

\begin{definition}
Let $H$ be a hypergraph. Let $T = e_1v_1e_2v_2$ be a $2$-cycle of $H$ such that $e_1\setminus e_2, e_2\setminus e_1 \ne \emptyset$. We say $T$ is \emph{good} if $|e_1\setminus e_2|\ge 2$ or $|e_2\setminus e_1|\ge 2$, and otherwise we say $T$ is \emph{bad}.   
\end{definition}

\begin{definition}
Let $H$ and $T$ be hypergraphs and $S\subseteq V(H)$. We let $N_T(H,S)$ denote the number of copies of $T$ in $H$ that contain $S$, and we let $\Delta_T(H,S) := |N_T(H,S)|$. 
\end{definition}

Recall that an $i$-cycle $e_1,v_1,\ldots, e_i,v_i$ in a hypergraph is \emph{loose} if $e_j\cap e_{(j+1)\mod i} = \{v_j\}$ and $e_j\cap e_k = \emptyset$ for all $j\in [i]$ and $k-j \not\equiv -1,0,1 \mod i$.

\begin{lem}\label{lem:NumberOfCycles}
Let $\beta > 0$ be a real. If $H$ is a $g$-bounded hypergraph where every edge has size at least $2$,  $\Delta_{k,\ell}(H) \le D^{k-\ell-\beta}$ for each $2\le \ell < k < g$, the maximum common $2$-degree of $H$ is at most $D^{1-\beta}$ and $w_D(H)\le \log D$, then for each good $2$-cycle, loose triangle or loose $4$-cycle $T$ and each nonempty $S\subseteq V(H)$ with $|S|\le v(T)-1$, we have 
\begin{align*}
\Delta_{T}(H,S) \le  (4g)^{v(T)}\cdot D^{v(T)-|S|-\beta} \cdot \log^{3} D.
\end{align*}
\end{lem}
\begin{proof}
Fix a mapping $\phi: S\rightarrow V(T)$ and let $S'=\phi(S)$. We prove the stronger statement that the number of copies of $T$ in $H$ (i.e.~homomorphisms $\phi'$ of $T$ into $V(H)$) that respect $\phi$ (i.e.~where $\phi'(\phi(s))=s$ for all $s\in S$) is at most
$$(4g)^{v(T)-|S|}\cdot D^{v(T)-|S|-\beta} \cdot \log^{e(T)-e(T[S'])-1} D.$$
To prove this stronger statement, suppose not. We let $S$ be a counterexample such that $v(T)-|S|$ is minimized.

Now choose $e\in E(T)$ such that $e\setminus S'\ne \emptyset$ and $e\cap S'\ne \emptyset$ and subject to that $|e\setminus S'|$ is minimized, and subject to that $|e\cap S'|$ is maximized. Note that such $e$ exists since $S'\ne \emptyset$, $S'\subsetneq V(T)$ and $T$ is connected. Let $S_0' := S'\cup e$ and $k := |e|$.

First suppose that $S_0' \ne V(T)$. Since $w_D(H)\le \log D$, we have that $\Delta_{k}(H) \le D^{k-1}\cdot \log D$. Thus there exists at most $D^{k-1}\cdot \log D$ edges $f$ of $H$ containing $\phi^{-1}(S'\cap e)$. For each such $f$, let $S_f = S\cup f$. By the minimality of $v(T)-|S|$, we find that for each mapping $\phi_f: S_f\rightarrow S_0'$, there exists at most $$(4g)^{v(T)-|S_f|}\cdot D^{v(T)-|S_f|-\beta} \cdot \log^{e(T)-e(T[S_f])-1} D$$
copies of $T$ in $H$ that respect $\phi_f$. Moreover, the number of choices of $\phi_f$ that respect $\phi$ is at most $v(T)^{|S_f|-|S|} \le (4g)^{|S_f|-|S|}$ since $v(T)\le 4g$. Since $e(T[S_f]) \ge e(T[S])+1$, the formula for $S$ now follows, a contradiction.

So we may assume that $S_0' = V(T)$. Next suppose that $|e\cap S'|\ge 2$. Since $e\setminus S' \ne \emptyset$, this implies that $k=|e|\ge 3$. Let $\ell := |e\cap S'|$. Since $\Delta_{k,\ell}(H)\le D^{k-\ell-\beta}$, we have that there exists at most $D^{k-\ell-\beta}$ edges $f$ of $H$ containing $\phi^{-1}(S'\cap e)$. For each such $f$, the number of copies of $T$ in $H$ that respect $\phi$ is at most $v(T)^{v(T)-|S|} \le (4g)^{v(T)-|S|}$ since $v(T)\le 4g$. Since $k-\ell = v(T)-|S|$, the formula for $S$ now follows, a contradiction.

So we may assume that $|e\cap S'|=1$. Since $T$ is a cycle, we have that there exists $v\in e\setminus S'$ and $e' \in E(T)\setminus \{e\}$ such that $v\in e'$. 

First suppose that $T$ is a loose triangle or loose $4$-cycle. Since $T$ is loose, it follows that $e'\setminus v \subseteq e'\setminus e$. Since $V(T)=S'\cup e$, we have that $e'\setminus e \subseteq S'$. Thus $e'\setminus v \subseteq S'$. Hence $e' \setminus S' = \{v\}$. By the choice of $e$, it follows that $|e\setminus S'| \le |e'\setminus S'| = 1$. Hence $e\setminus S' = \{v\}$. Since $|e\cap S'|=1$, we find that $|e|=2$. By the choice of $e$, it follows that $|e'\cap S'|=1$ and hence $|e'|=2$. 

Since the maximum common $2$-degree of $H$ is at most $D^{1-\beta}$, we have that there exists at most $D^{1-\beta}$ vertices $w$ of $H$ such that there exist edges $f_1,f_2$ of $H$ of size two where $f_1=\{w,\phi^{-1}(S'\cap e)\}$ and $f_2=\{w,\phi^{-1}(S'\cap e')\}$. For each such $w$, the number of copies of $T$ in $H$ that respect $\phi$ is at most $1\le v(T)^{v(T)-|S|} \le (4g)^{v(T)-|S|}$ since $v(T)\le 4g$. Since $v(T)-|S|=1$, the formula for $S$ now follows, a contradiction.

So we may assume that $T$ is a good $2$-cycle. 
 Since $V(T)=S'\cup e$, we have that  $e'\setminus e \subseteq S'$. Since $v\in e'\setminus S'$, we have that $e'\setminus S'\ne \emptyset$. By the choice of $e$, we find that $|e'\setminus S'| \ge |e\setminus S'|$. It follows that $e\setminus e' \subseteq S'$. Hence $V(T)\setminus S' \subseteq (e\cap e')\setminus S'$ and thus $|e\setminus S'|= |e'\setminus S'|$. Then by the choice of $e$,  we have that $|e\cap S'| \ge |e'\cap S'|$.  Hence $1=|e\cap S'|\ge |e'\cap S'|$. But then $T$ is not a good $2$-cycle, a contradiction. This proves the stronger statement.

The lemma now follows since there are at most $v(T)^{|S|}\le (4g)^{|S|}$ choices for $\phi$ and \\ $e(T)-e(T[S'])-1 \le e(T)-1\le 3$.
\end{proof}

We are now prepared to prove Lemma~\ref{lem:Sparsifying}.

\begin{proof}[Proof of Lemma~\ref{lem:Sparsifying}]
We may assume without loss of generality that there are not distinct $S,T\in E(H)$ where $S\subseteq T$.

Note that since $G$ has maximum degree at most $2D$ by assumption and that $\Delta_{k,\ell}(H)\le D^{k-\ell-\beta}$ for all $2\le \ell < k \le g$, it follows that, for all $i$ with $3\le i \le g$, the maximum $i$-codegree of $G$ with $H$ is at most $2D\cdot D^{i-2-\beta} = 2\cdot D^{i-1-\beta}$. By assumption, the maximum $2$-codegree of $G$ with $H$ is at most $D^{1-\beta}$ and hence for all $i$ with $2\le i \le g$, the maximum $i$-codegree of $G$ with $H$ is at most $2\cdot D^{i-1-\beta}$.

Let $p=D^{\beta/4g -1}$. Choose each edge of $G$ independently with probability $p$. Let $E_0$ be the set of chosen edges. Let $G_0$ be the hypergraph with $V(G_0)=V(G)$ and $E(G_0)=E_0$ and let $H_0$ be the configuration hypergraph of $G_0$ where $V(H_0)=E(G_0)$ and $E(H_0) = \{S\in E(H): S\subseteq E_0\}$.

We now define bad events for each of the desired outcomes and calculate their probabilities as follows.

\vskip.15in
\noindent{\bf Bad Events for (1):} For each $v\in V(G)$, define a bad event $A_v$ where
$$d_{G_0}(v) \ne p\cdot d_G(v) \pm \frac{1}{2} \left(p\cdot d_G(v) \cdot D^{-\beta/18g} + D^{\beta/6g}\right).$$
Note that $\Expect{d_{G_0}(v)} = p\cdot d_G(v)$ while $d_{G_0}(v)$ is $(1,1)$-observable. Let $t=\frac{1}{2} \left(p\cdot d_G(v) \cdot D^{-\beta/18g} + D^{\beta/6g}\right)$. Thus by Talagrand's Inequality with $t$ as above,  %that is (pD)^{2/3}
we have that 
$$\Prob{A_v} \le 4 e^{-D^{\beta/18g} /40} \le e^{-\log^2 D},$$ 
where we used the fact that $D$ is large enough.

\vskip.15in
\noindent{\bf Bad Events for (2):} For each $u,v\in V(G)$ such that $N_G'(u)\cap N_G'(v)\ne \emptyset$, define a bad event $A_{u,v}$ where 
$$|\{ e\in E(G_0): u,v\in e\}| > \log^2(pD).$$
Note
$$\Expect{|\{ e\in E(G_0): u,v\in e\}|} = p \cdot |N_G'(u)\cap N_G'(v)|\le p\cdot D^{1-\beta} \le 1$$ while $|\{ e\in E(G_0): u,v\in e\}|$ is $(1,1)$-observable. Thus by Talagrand's Inequality with $t= \log^2(pD)-1 \ge \log^2(pD)/2$,
$$\Prob{A_{u,v}} \le 4e^{-\log^2(pD)/160} \le e^{-\log^{1.5} D},$$
where we used the fact that $D$ is large enough.

\vskip.15in
\noindent{\bf Exceptional outcomes for (3) and (4):} For each $S\subseteq V(H)$ with $|S|\ge 1$ and $|S|< i\le g$, let $F_{S,i} := \{T\in E(H):~S\subseteq T,~|T|=i,~T\setminus S \subseteq E(G_0)\}$. For each $S$ with $|S|\ge 2$ and $F_{S,i}\ne \emptyset$, define an exceptional outcome $B_{S,i}$ where 
$$|F_{S,i}| > \log^{2(i-|S|)}(pD).$$
Since $\Delta_i(H,S) \le \Delta_{i,|S|}(H) \le D^{i-|S|-\beta}$ for all $S$ and $i$ with $2\le |S| < i \le g$ by assumption, we find that for all $S$ and $i$ with $2\le |S| < i \le g$
$$\Expect{|F_{S,i}|} = p^{i-|S|}\cdot \Delta_i(H,S) \le (pD)^{i-|S|} \cdot D^{-\beta}\le  1,$$
where we used that $(pD)^{i-|S|} \le (pD)^g \le D^{\beta/4}$.

For each $S\subseteq V(H)$ with $|S|\ge 1$ and $|S|< i\le g$, define a set of exceptional outcomes 
$$\Omega_{S,i}^* = \bigcup_{T\supsetneq S} B_{T,i}.$$

\begin{claim}\label{claim:Exceptional}
For each $S\subseteq V(H)$ with $|S|\ge 1$ and $|S|< i\le g$, $|F_{S,i}|$ is $\left(i-|S|,~\log^{2(i-|S|-1)} (pD)\right)$-observable with respect to $\Omega_{S,i}^*$, and for $|S|\ge 2$ with $F_{S,i}\ne \emptyset$ we have that $$\Prob{B_{S,i}} \le e^{-\log^{1.75} D} + 4\Prob{\Omega_{S,i}^*} \le e^{-\log^{1.6} D}.$$
\end{claim}
\begin{proof}
We proceed by induction on $i-|S|$. First suppose that $i=|S|+1$. Since $H$ has no repeated edges, $|F_{S,i}|$ is $(1,1)$-observable as desired. 

Since $\Expect{|F_{S,i}|}\le 1$, by Talagrand's Inequality with $t= \log^2(pD)-1 \ge \log^2(pD)/2$, we find that
$$\Prob{B_{S,i}} \le 4\cdot e^{-\log^2(pD)/160} \le e^{-\log^{1.75} D},$$
as desired, where we used the fact that $D$ is large enough.

So we may assume that $i > |S|+1$. Let $\omega \in \Omega\setminus \Omega_{S,i}^*$. For any edge $e$ of $G$ not in $S$, we have that $B_{S\cup \{e\},i} \subseteq \Omega_{S,i}^*$. Hence in $\omega$, we have that $|F_{S\cup \{e\},i}| \le \log^{2(i-|S|-1)}(pD)$. It follows that $|F_{S,i}|$ is $\left(i-|S|,~\log^{2(i-|S|-1)} (pD)\right)$-observable with respect to $\Omega_{S,i}^*$.

Since $\Expect{|F_{S,i}|}\le 1$, by Linear Talagrand's Inequality with $t= \log^{2(i-|S|)}(pD)-1 \ge \log^{2(i-|S|)}(pD)/2$, we find that
$$\Prob{B_{S,i}} \le 4\cdot {\rm exp}\left(-\frac{\log^{2(i-|S|)}(pD)}{160g\cdot \log^{2(i-|S|-1)}(pD)}\right) + 4\Prob{\Omega_{S,i}^*}\le 4\cdot e^{-\log^{1.75} D}+ 4\Prob{\Omega_{S,i}^*}  \le e^{-\log^{1.6} D},$$
%(explain about $\Prob{\Omega_{S,i}^*}$??)
as desired, where we used the fact that $D$ is large enough.
\end{proof}

\vskip.15in
\noindent{\bf Bad Events for (3):} For each $e\in V(H)$, let
$$w(e):= \sum_{2\le i \le g} \frac{i-1}{(pD)^{i-1}} \cdot |F_{\{e\},i}|$$ and then
define a bad event $A_e$ where 
$$w(e) > w_{D}(H)+1.$$
Note that when $e\in E_0$, then $w_{pD}(H_0,e) = w(e)$.

Note $\Expect{w(e)} = w_D(H,e)$. Define a set of exceptional outcomes
$$\Omega_e^* = \bigcup_{2\le i\le g} \Omega_{\{e\},i}^*.$$
By Claim~\ref{claim:Exceptional}, for all $2\le i \le g$, we find that $|F_{\{e\},i}|$ is $\left(i-1,~\log^{2(i-2)}(pD)\right)$-observable with respect to $\Omega_e^*$. By Proposition~\ref{prop:scaling}, it follows that $\frac{i-1}{(pD)^{i-1}} \cdot |F_{\{e\},i}|$ is $\left(i-1,~ \frac{i-1}{(pD)^{i-1}} \cdot \log^{2(i-2)}(pD)\right)$-observable with respect to $\Omega_e^*$. Since $pD = D^{\beta/4g} \ge 2$, we have that $\frac{i-1}{(pD)^{i-1}} \cdot |F_{\{e\},i}|$ is $\left(g-1,~\frac{\log^{2(g-2)}(pD)}{D^{\beta/4g}}\right)$-observable.

Thus by Theorem~\ref{exceptional talagrand's difference} with $t= 1$, $m=g-1$, and $M=w_D(H)$, we find that 
$$\Prob{A_{e}} \le 4(g-1) \cdot {\rm exp}\left(-\frac{1}{8(g-1)\cdot \frac{\log^{2(g-2)}(pD)}{D^{\beta/4g}}\cdot (g-1)^2 \cdot (4\cdot w_D(H)+1)}\right) + 4(g-1)\cdot \Prob{\Omega_e^*} \le e^{-\log^{1.5} D},$$
where we used the fact that $D$ is large enough and Claim~\ref{claim:Exceptional} to upper bound $\Prob{\Omega_{e}^*}$.

\vskip.15in
\noindent{\bf Bad Events for (4):} For each $e\in V(H)$ and $v\in V(G)$ where there exists $S\in E(H)$ with $e\in S$ and $v\in V(S\setminus e)$, let for all $i$ with $2\le i \le g$, 
$$D_{e,v,i}:= \{S\in N_{H}(e): v\in V(S), (S\setminus e)\subseteq E_0, |S|=i\},$$ and
then define a bad event $A_{e,v}$ where 
$$\sum_{2\le i \le g} |D_{e,v,i}| > \frac{1}{g} \cdot \log^{2(g-1)}(pD).$$
When $e\in E(G_0)$, then the codegree of $e$ with $v$ is equal to $\sum_{2\le i \le g} |D_{e,v,i}|$. Since the $i$-codegree of $v$ with $e$ is at most $2\cdot D^{i-1-\beta}$ for all $i\in \{2,\ldots,g\}$ and using the fact that $(pD)^g \le D^{\beta}$, we have that
$$\Expect{|D_{e,v,i}|} \le p^{i-1}\cdot (2\cdot D^{i-1-\beta}) \le 2.$$ 

Define a set of exceptional outcomes
$$\Omega_e^* = \bigcup_{2\le i\le g} \Omega_{\{e\},i}^*.$$

For any edge $f\ne e$ of $G$ with $N_H'(e)\cap N_H'(f)\ne \emptyset$, we have that $B_{\{e,f\},i} \subseteq  \Omega_{\{e\},i}^* \subseteq \Omega_{e}^*$. Hence in $\omega$, we have that $|F_{\{e,f\},i}| \le \log^{2(i-2)}(pD)$. It follows that $|D_{e,v,i}|$ is $\left(i-1,~\log^{2(i-2)}(pD)\right)$-observable with respect to $\Omega_{e}^*$.

Thus by Theorem~\ref{exceptional talagrand's difference} with $t=\frac{1}{g}\cdot\log^{2(g-1)}(pD)-2(g-1)$, $m=g-1$, and $M=2$, we find that 
$$\Prob{A_{e,v}} \le 4(g-1) \cdot {\rm exp}\left(-\frac{\frac{1}{g}\cdot\log^{2(g-1)}(pD)-2(g-1)}{8(g-1)\cdot \log^{2(g-2)}(pD)\cdot (g-1)^2(9)}\right) + 4(g-2)\cdot \Prob{\Omega_e^*} \le e^{-\log^{1.55} D},$$
where we used the fact that $D$ is large enough and Claim~\ref{claim:Exceptional} to upper bound $\Prob{\Omega_{e}^*}$.

\vskip.15in
\noindent{\bf Exceptional outcomes for (5):} Let $\mathcal{T}_{{\rm good}}$ be the set of good $2$-cycles, loose triangles and loose $4$-cycles whose edges have size at least $2$ and at most $g$.

For each $T\in\mathcal{T}_{{\rm good}}$ and each nonempty $S\subseteq V(H)$ and each $|S|< v(T)$ with $|N_T(H,S)|\ne 0$, let
$$F_{S,T} := \{T'\in N_T(H,S): T'\setminus S \subseteq E_0\},$$
and define an exceptional outcome $B_{S,T}$ where 
$$|F_{S,T}| > \log^{2(v(T)-|S|)}(pD).$$
By Lemma~\ref{lem:NumberOfCycles}, we find that $|N_T(H,S)| \le (4g)^{4g} \cdot D^{v(T)-|S|-\beta}\cdot \log^3 D$. Thus we find by linearity of expectation that
$$\Expect{|F_{S,T}|} = p^{v(T)-|S|}\cdot |N_T(H,S)| \le (4g)^{4g} \cdot (pD)^{v(T)-|S|} \cdot D^{-\beta}\cdot \log^3 D\le  1,$$
since $v(T) \le 4g-4$ and $D$ is large enough.

For each $T\in\mathcal{T}_{{\rm good}}$ and each nonempty $S\subseteq V(H)$ and each $|S|< v(T)$ with $|N_T(H,S)|\ne 0$, define a set of exceptional outcomes 
$$\Omega_{S,T}^* = \bigcup_{S'\supsetneq S} B_{S',T}.$$

\begin{claim}\label{claim:Exceptional2}
For each $T\in\mathcal{T}_{{\rm good}}$ and each nonempty $S\subseteq V(H)$ and each $|S|< v(T)$, $|F_{S,T}|$ is $\left(v(T)-|S|,~\log^{2(v(T)-|S|-1)}(pD)\right)$-observable with respect to $\Omega_{S,T}^*$ and 
$$\Prob{B_{S,T}} \le 4\cdot e^{-\log^{1.75} D} + 4\Prob{\Omega_{S,T}^*} \le e^{-\log^{1.6} D}.$$
\end{claim}
\begin{proof}
We proceed by induction on $v(T)-|S|$. First suppose that $v(T)=|S|+1$. Since $H$ has no repeated edges, it follows that $|F_{S,T}|$ is $(1,1)$-observable as desired. 

Since $\Expect{|F_{S,T}|}\le 1$, by Talagrand's Inequality with $t= \log^2(pD)-1 \ge \log^2(pD)/2$, we find that
$$\Prob{B_{S,T}} \le 4e^{-\log^2(pD)/160} \le e^{-\log^{1.75} D},$$
as desired, where we used the fact that $D$ is large enough.

So we may assume that $v(T) > |S|+1$. Let $\omega \in \Omega\setminus \Omega_{S,T}^*$. For any edge $e$ of $G$ not in $S$, we have that $B_{S\cup \{e\},T} \subseteq \Omega_{S,T}^*$. Hence in $\omega$, we have that $|F_{S\cup \{e\},T}| \le \log^{2(v(T)-|S|-1)}(pD)$. It follows that $|F_{S,T}|$ is $\left(v(T)-|S|,~\log^{2(v(T)-|S|-1)}(pD)\right)$-observable with respect to $\Omega_{S,T}^*$.

Since $\Expect{|F_{S,T}|}\le 1$, by Linear Talagrand's Inequality with $t= \log^{2(v(T)-|S|)}(pD)-1 \ge \log^{2(v(T)-|S|)}(pD)/2$, we find that
$$\Prob{B_{S,T}} \le 4\cdot {\rm exp}\left(-\frac{\log^{2(v(T)-|S|)}(pD)}{160(4g)\cdot \log^{2(v(T)-|S|-1)}(pD)}\right) + 4\Prob{\Omega_{S,T}^*}\le 4\cdot e^{-\log^{1.75} D}+ 4\Prob{\Omega_{S,T}^*}  \le e^{-\log^{1.6} D},$$
%(explain about $\Prob{\Omega_{S,T}^*}$?!)
as desired, where we used the fact that $D$ is large enough and induction to upper bound $\Prob{\Omega_{S,T}^*}$. 
\end{proof}

\vskip.15in
\noindent{\bf Bad Events for (5):} Let $\mathcal{T}_{{\rm bad}}$ be the set of bad $2$-cycles whose edges have size at least $2$ and size at most $g$. Note that since there do not exist distinct $S,T\in E(H)$ where $S\subseteq T$, it follows that all the edges in the $2$-cycles in $\mathcal{T}_{{\rm bad}}$ have size at least $3$. Let $\mathcal{T} = \mathcal{T}_{{\rm good}}\cup \mathcal{T}_{{\rm bad}}$.

For each $v\in V(G)$ and $T\in\mathcal{T}$, 
let 
$$T_v:= \{f\in N_{G_0}'(v): N_T(H_0,\{f\})\ne \emptyset\}.$$
Let
$$E_{v,T} := \{e\in E(G): \exists f\in N_{G}'(v) \textrm{ such that } N_T(H,\{e,f\})\ne\emptyset \}.$$
For each $e\in E_{v,T}$, let
$$T_{v,e}:= \{f\in N_{G_0}'(v): \exists T'\in N_T(H,\{e,f\}) \textrm{ with } T'\setminus \{e\} \subseteq E_0\}.$$

Define a bad event $A_{v,T}$ where 
$$|T_v| > p\cdot d_G(v) \cdot D^{-\beta/6g} + D^{\beta/12g}.$$

When $T\in\mathcal{T}_{{\rm good}}$, we have by Lemma~\ref{lem:NumberOfCycles} that for each $e\in E(G)$, $|N_T(H,\{e\})| \le (4g)^{v(T)} \cdot D^{v(T)-1-\beta} \cdot \log^3 D$. On the other hand when $T\in\mathcal{T}_{{ bad}}$, we have that for each $e\in E(G)$, $|N_T(H,\{e\})|\le g \cdot D^{v(T)-1-\beta} \cdot \log D$ (since there are at most $D^{|S_1|-1}\cdot \log D$ choices for an edge $S_1$ of $T$ containing $e$ and then at most $g\cdot D^{|S_2|-(|S_1|-1)-\beta}$ choices of the edge $S_2$ of $T$ that contains all but one vertex of $S_1$). Either way, we find that for each $e\in E(G)$, we have that
$$|N_T(H,\{e\})| \le (4g)^{4g} \cdot D^{v(T)-1-\beta} \cdot \log^3 D.$$

But then by the union bound
\begin{align*}
\Prob{ e\in E(G_0) \cap N_T(H_0,e)\ne 0} &\le p^{v(T)} \cdot (4g)^{4g} \cdot D^{v(T)-1-\beta} \cdot \log^3 D \\
&\le (4g)^{4g} \cdot \log^3(D) \cdot p (pD)^{v(T)-1}\cdot D^{-\beta}.
\end{align*}
Thus by linearity of expectation, we find that
$$\Expect{|T_v|} \le d_G(v) \cdot (4g)^{4g} \cdot \log^3(D) \cdot p \cdot (pD)^{v(T)-1}\cdot D^{-\beta} \le p\cdot d_G(v) \cdot D^{-\beta/4g}$$
since $v(T)\le 4g-4$ and $D$ is large enough.

For each $e\in E_{v,T}$, define an exceptional outcome $B_{e,v,T}$
where
$$|T_{v,e}| > \log^{8g-2}(pD)$$
and then define a set of exceptional outcomes
$$\Omega_{v,T}^* = \bigcup_{e\in E_{v,T}} B_{e,v,T}.$$

\begin{claim}\label{claim:Exceptional3}
For each $e\in E_{v,T}$, $\Prob{B_{e,v,T}} \le 2\cdot e^{-\log^{1.55} D}$.
\end{claim}
\begin{proof}
First suppose $T\in \mathcal{T}_{{\rm good}}$. Note $|T_{v,e}| \le |F_{\{e\},T}|$. Hence $\Prob{B_{e,v,T}} \le \Prob{B_{\{e\},T}}$ which is at most $e^{-\log^{1.6} D}$ by Claim~\ref{claim:Exceptional}.

So we may assume that $T\in \mathcal{T}_{{\rm bad}}$. Let $T_1\subseteq T_{v,e}$ be where $e$ and $f$ are in the same edge of some copy of $T$. Let $T_2\subseteq T_{v,e}$ be where $e$ and $f$ are in different edges of some copy of $T$. Note that $T_{e,v}\subseteq T_1\cup T_2$.

Note that $T_1 \subseteq \bigcup_{3\le i \le g} D_{e,v,i}$. Yet $\Prob{A_{e,v}} \le e^{-\log^{1.55}} D$ and hence $\Prob{|T_1|\ge \log^{2(g-1)}(pD)} \le e^{-\log^{1.55} D}$.

Meanwhile, $\Expect{|T_2|} \le |N_T(H,\{e\})|\cdot p^{v(T)-1} \le 1$. Let $\Omega_{e,v,T}^* = \bigcup_{S\in N_H'(e)} \bigcup_{f\in S} A_{f,v}$. Thus $|T_2|$ is $\left(2g,~ \log^{2(g-1)}(pD)\right)$-observable with respect to $\Omega_{e,v,T}^*$. By Talagrand's Inequality with $t=\log^{2g}(pD)-1 \ge \log^{2g}(pD)/2$, we find that
$$\Prob{|T_2|\ge \log^{2g}(pD)} \le 4\cdot {\rm exp}\left(-\frac{\log^{2g}(pD)}{160(2(g-1))\cdot \log^{2(g-1)}(pD)}\right) + 4\Prob{\Omega_{e,v,T}^*}\le e^{-\log^{1.55} D},$$
since $D$ is large enough.

Thus $\Prob{B_{v,e,T}}\le \Prob{|T_1|+|T_2| \ge \log^{8g}(pD)} \le \Prob{|T_1|\ge \log^{2(g-1)}(pD)} + \Prob{|T_2|\ge \log^{2g}(pD)} \le 2\cdot e^{-\log^{1.55} D}$.

\end{proof}

Note that $|T_v|$ is $\left(4g,\log^{8g-2}(pD)\right)$-observable with respect to $\Omega_{v,T}^*$ by definition. 

Recall that $\Expect{|T_v|}\le p\cdot d_G(v) \cdot D^{-\beta/4g}$. By Talagrand's Inequality with $t= p\cdot d_G(v) \cdot D^{-\beta/4g} + D^{\beta/12g}$, we find that
$$\Prob{A_{v,T}} \le 4\cdot {\rm exp}\left(-\frac{p\cdot d_G(v) \cdot D^{-\beta/4g} + D^{\beta/12g}}{160(4g)\cdot \log^{8g-2}(pD)}\right) + 4\Prob{\Omega_{v,T}^*}\le e^{-\log^{1.5} D},$$
where we used that $D$ is large enough and Claim~\ref{claim:Exceptional3} to upper bound $\Prob{\Omega_{v,T}^*}$.

\vskip.15in
\noindent{\bf Applying the Local Lemma:}
Let $J$ be the graph with $V(J) := V(G) \cup E(G)$ and 
\begin{align*}
E(J) := &\{ve: v\in e, v\in V(G), e\in E(G)\} \\
&\cup \{ef: e,f\in E(G), e\cap f\ne \emptyset\} \cup \{ef: e,f\in E(G), \exists S\in E(H) \textrm{ with }e,f\in S\}.    
\end{align*}

Note now that the event $A_v$ depends only on trials $(f\in E_0)$ for edges $f$ of $G$ where ${\rm dist}_{J}(v,f) = 1$. Similarly the event $A_{u,v}$ (where ${\rm dist}_J(u,v)=2$) depends only on trials $(f\in E_0)$ for edges $f$ of $G$ where ${\rm dist}_{J}(v,f) = 1$; the event $A_e$ depends only on trials $(f\in E_0)$ for edges $f$ of $G$ where ${\rm dist}_{J}(e,f) \le 1$. The event $A_{e,v}$ (where ${\rm dist}_J(e,v)\le 2$) depends only on trials $(f\in E_0)$ for edges $f$ of $G$ where ${\rm dist}_{J}(v,f) \le 2$; the event $C_v$ depends only on trials $(f\in E_0)$ for edges $f$ of $G$ where ${\rm dist}_{J}(v,f) \le 3$.

Note that since $G$ has maximum degree at most $2D$, we find that 
$$\Delta(J) \le \max \left\{2D+1,~2Dr + \sum_{i\le g} D^i \log D\right\} \le D^{g+1},$$
where the last inequality follows since $D$ is large enough.

Yet for each event $A_x\in \mathcal{E}$, we have $$\Prob{A_x} \le e^{-\log^{1.5} D} \le \frac{1}{4D^{7(g+1)}},$$
since $D$ is large enough. Hence by the Lov\'asz Local Lemma, with positive probability none of the events in $\mathcal{E}$ happen. This implies that all of the outcomes (2)-(4) hold for $G_0$ and $H_0$ as desired. Furthermore, $d_{G_0}(v) =  p\cdot d_{G}(v) \pm \frac{1}{2} \left(p\cdot d_G(v) \cdot D^{-\beta/18g} + D^{\beta/6g}\right)$ for each $v\in V(G)$.

Finally define $G'$ to be the hypergraph obtained from $G_0$ by deleting all edges in cycles in $H_0$ in $\mathcal{T}$. Let $H'$ be the configuration hypergraph of $G'$ where $H' = H_0\setminus (E(G_0)\setminus E(G'))$. Since outcomes (2)-(4) held for $G_0$ and $H_0$, it follows they also hold for $G'$ and $H'$. Moreover, (5) now holds for $G'$ and $H'$ by construction. 

Finally, for each $v\in V(G)$ and $T\in \mathcal{T}$, since $A_{v,T}$ does not hold, it follows that $|T_v|\le p\cdot d_G(v) \cdot D^{-\beta/6g} + D^{\beta/12g}$. Meanwhile, note that there are at most $g^4$ $2$-cycles (good or bad) with edges of size at most $g$, at most $g^4$ loose $3$-cycles with edges of size at most $g$, and at most $g^4$ loose $4$-cycles with edges of size at most $g$. Hence $|\mathcal{T}|\le 3\cdot g^4$. Hence 
for each $v\in V(G)$,
\begin{align*}
d_{G'}(v) &= d_{G_0}(v) \pm 3\cdot g^4 \left(p\cdot d_G(v) \cdot D^{-\beta/6g} + D^{\beta/12g}\right)\\
&= p\cdot d_{G}(v) \pm \frac{1}{2} \left(p\cdot d_G(v) \cdot D^{-\beta/18g} + D^{\beta/6g}\right) \pm 3\cdot g^4 \left(p\cdot d_G(v) \cdot D^{-\beta/6g} + D^{\beta/12g}\right)\\
&= p\cdot d_G(v) ( 1 \pm D^{-\beta/18g}) \pm D^{\beta/6g},    
\end{align*}  
since $D$ is large enough, and hence (1) holds for $G'$.

Hence all of outcomes (1)-(5) hold for $G'$ and $H'$ as desired.
\end{proof}

\section{Rainbow Matching Proofs}\label{s:Rainbow}
Recall that for a hypergraph $G$ whose edges are colored by a not necessarily proper coloring $\phi$, define a new hypergraph, the \emph{rainbow hypergraph of $G$ with respect to $\phi$} as Rainbow$(G,\phi):=(A,B)$ where $A=\bigcup_{e\in E(G)} \phi(e)$ is the set of colors and $B=V(G)$ is the set of vertices and then extend every edge $e$ of $G$ to include its color $\phi(e)$, that is $E({\rm Rainbow}(G,\phi)) = \{ e\cup \phi(e): e\in E(G)\}$. Then a rainbow matching of $G$ is precisely a matching in ${\rm Rainbow}(G,\phi)$.  

In this section, we prove Theorems~\ref{thm:SparseGrinblat},~\ref{thm:CKAverage} and~\ref{thm:CKBipartite}, but first a proposition.

\begin{proposition}\label{prop:RandomDecrease}
For all integers $r\ge 2$, there exists an integer $D_0$ such that the following holds for all $D\ge D_0$:

Let $G$ be an $r$-bounded hypergraph such that $d_G(v) \le D$ for all $v\in V(G)$. If $w: E(G)\rightarrow [0,1]$ is a weight function on the edges of $G$, then there exists a subhypergraph $G'$ of $G$ with $V(G')=V(G)$ where for all $v\in V(G')$,
$$d_{G'}(v) = \left(\sum_{e\in N'_G(v)} w(e)\right) \pm D^{2/3}.$$
\end{proposition}
\begin{proof}
Let $G'$ be the subhypergraph of $G$ with $V(G')=V(G)$ obtained by choosing each edge of $G$ independently with probability $w(e)$. For each $v\in V(G)$, let $A_v$ be the bad event that $d_{G'}(v) \ne \sum_{e\in N'_G(v)} w(e) \pm D^{2/3}$. By linearity of expectation, we have that $\Expect{d_{G'}(v)} = \sum_{e\in N_G'(v)} w(e)$. Note that $d_{G'}(v)$ is $(1,1)$-observable.  Hence by Talagrand's Inequality with $t = D^{2/3}$, we have that
$$\Prob{A_v} \le 4 e^{-D^{1/3}/40}.$$

Let $\mathcal{E}$ be the set of all bad events. For each $v\in V(G)$, let $\mathcal{E}_v := \{A_u: u\in V(G),~\textrm{dist}_G(u,v) \le 1\}$. Note that $A_v$ is mutually independent of $\mathcal{E}\setminus \mathcal{E}_v$. Since $G$ is $r$-bounded, we have that $|\mathcal{E}_v| \le rD$. Hence by the Lov\'asz Local Lemma, we have that none of the events in $\mathcal{E}$ happen with positive probability, in which case $G'$ is as desired.
\end{proof}

We are now ready to prove Theorem~\ref{thm:SparseGrinblat}.

\begin{proof}[Proof of Theorem~\ref{thm:SparseGrinblat}]
Let $\phi$ be the edge coloring of $G$. Let $C$ denote the set of colors in $\phi$. Let $H= {\rm Rainbow}(G,\phi)$.

We may assume without loss of generality that each color $c\in C$ is the disjoint union of cliques whose number of vertices is at least $r$ and at most $2r-1$ (by partitioning larger cliques into smaller cliques and deleting edges between parts of the partition). We may also assume without loss of generality that each color spans at most $(r+1)D$ vertices. 

Note that every vertex has degree in $G$ (and hence in $H$) at most $D\cdot \binom{2r-2}{r-1}$. Hence every vertex of $G$ has degree in $H$ at most $D':= 2^{2r-2}\cdot D$.

For each $q\in \{r,\ldots, 2r-1\}$, let $E_q$ denote the set of edges $e$ of $G$ where the clique of the color of $e$ containing $e$ contains exactly $q$ vertices. For each $e\in E_q$, let 
$$w(e) := \frac{1}{\binom{q-1}{r-1}}.$$
By Proposition~\ref{prop:RandomDecrease} applied to $H$ with $w$, it follows that there exists a subhypergraph $H'$ of $H$ with $V(H')=V(H)$ and for all $v\in V(H')$,
$$d_{H'}(v) = \left(\sum_{e\in N'_H(v)} w(e)\right) \pm D^{3/4},$$
where we used that $(D')^{2/3} \le D^{3/4}$ as $D$ is large enough.

For each $c\in C$, the number of vertices spanned by $C$ in $H'$ is precisely
$$\sum_{q=r}^{2r-1} |E_q\cap N_H'(c)|\cdot \frac{q}{\binom{q}{r}},$$
which is at least $rD(1+D^{-\alpha})$ by assumption. Hence,
$$\sum_{e\in N'_H(c)} w(e) = \sum_{q=r}^{2r-1} |E_q\cap N_H'(c)|\cdot \frac{1}{\binom{q-1}{r-1}} \ge D(1+D^{-\alpha}).$$

On the other hand, for each $v\in V(G)$, since each vertex is incident with exactly $\binom{q-1}{r-1}$ edges of a color when contained in a clique of $q$ vertices of that color, we have that 
$$\sum_{q=r}^{2r-1} |E_q\cap N_H'(v)|\cdot \frac{1}{\binom{q-1}{r-1}} \le D.$$
Hence
$$\sum_{e\in N'_H(v)} w(e) = \sum_{q=r}^{2r-1} |E_q\cap N_H'(v)|\cdot \frac{1}{\binom{q-1}{r-1}} \le D.$$

By choosing $\alpha$ small enough, it now follows from Theorem~\ref{thm:KahnBipartite} applied to $H'$ that there exists a full rainbow matching of $G$ and indeed even a set of $D$ disjoint full rainbow matchings of $G$.
\end{proof}

We now prove Theorem~\ref{thm:CKAverage}.

\begin{proof}[Proof of Theorem~\ref{thm:CKAverage}]
We proceed by induction on $r$. Let $\phi$ be the edge coloring of $G$ and let $C$ be the set of colors used in $\phi$. We may assume without loss of generality that $\varepsilon = \frac{1}{2^k}$ for some positive integer $k$. Note then that $|C|=(2+\varepsilon)q$ and hence $(2+\varepsilon)q$ is integral. We may assume without loss of generality (by deleting edges as necessary) that $e(G)= (2+\varepsilon)qr$.  

Note that $\Delta(G)\le (2+\varepsilon)q$. Let $X$ be the set of vertices of $G$ with degree at least $\left(2+\frac{\varepsilon}{2}\right)q$. Let $G'=G(X,V(G)\setminus X)$, that is the bipartite subgraph of $G$ consisting of edges of $G$ with exactly one end in $X$.
Now let $\alpha > 0$ be the real given by Theorem~\ref{thm:KahnBipartite} for $r=3$ and $\beta = 1/2$ say. 

First suppose that $e(G\setminus X) \ge 3\cdot (2^{k+2}+1)\cdot rq^{1-\alpha}$. Let
$$a_1 := 2^{k+2}+1,~~a_2:=2^{k+2}+2.$$
Let $G_0$ be obtained from $G$ by repeating every edge of $G$ incident with at least one vertex of $X$ exactly $a_1$ times and every other edge of $G$ exactly $a_2$ times. %Note that $a_1 = 2^{k+1} (2+\varepsilon/2)$ and $a_2 = 2^{k+1}\cdot (2+\varepsilon)$.
Let 
$$D:=(2^{k+2}+1)\cdot (2+\varepsilon)q.$$ 
Then every vertex in $X$ has degree in $G_0$ at most 
$$a_1 \cdot (2+\varepsilon)q = (2^{k+2}+1) \cdot (2+\varepsilon)q=D$$ 
while every vertex in $V(G)\setminus X$ has degree in $G_0$ at most 
$$a_2 \cdot (2+\varepsilon/2)q = (2^{k+2}+2) \cdot (2+\varepsilon/2)q = D.$$
Thus $\Delta(G_0) \le D$. 

Note that $G_0$ has multiplicity at most $(2^{k+2}+2)\le D^{1-\beta}$ since $q$ is large enough. Meanwhile, 
$$e(G_0) = (2^{k+2}+1)\cdot e(G) + e(G\setminus X).$$ 
Let $H:={\rm Rainbow}(G_0,\phi)$. Since $G$ is simple and $\phi$ is a proper edge coloring of $G$, it follows that $H$ has codegrees at most $D^{1-\beta}$ since $q$ is large enough. 

Thus by Theorem~\ref{thm:KahnKahn} applied to $H$, we have that $\chi(L(H))\le D(1+D^{-\alpha})$ and hence there exists a matching of $H$ of size at least $\frac{e(H)}{D(1+D^{-\alpha})}$. Hence $G_0$ (and hence also $G$) has a rainbow matching $M$ of size $\frac{e(G_0)}{D(1+D^{-\alpha})}$. Note that
$$(2^{k+2}+1)\cdot e(G) = (2^{k+2}+1)\cdot (2+\varepsilon)qr = D\cdot r,$$
and yet 
$$e(G\setminus X) \ge 3\cdot (2^{k+2}+1)\cdot rq^{1-\alpha} \ge D^{1-\alpha}\cdot r.$$
Thus we find that 
$$e(M) \ge \frac{e(G_0)}{D(1+D^{-\alpha})} \ge r,$$ 
as desired.

So we may assume that $e(G\setminus X) \le 3\cdot (2^{k+2}+1) \cdot rq^{1-\alpha}.$ Now by induction, there exists a matching $M$ in $G\setminus X$ of size at least $\frac{e(G\setminus X)}{(2+\varepsilon)q}$. Let $G'$ be obtained from $G\setminus V(M)$ by deleting all edges colored with colors of $M$. 

Since $e(M)\ge \frac{e(G\setminus X)}{(2+\varepsilon)q}$, we have that 
$$e(G)-e(G\setminus X) \ge (r-e(M))\cdot (2+\varepsilon)q.$$ 
Since $X$ has maximum degree $(2+\varepsilon)q$, it follows that 
$$e(G)-e(G\setminus X)\le |X|\cdot (2+\varepsilon)q$$ and hence 
$$|X| \ge r-e(M).$$ 
Note that $e(M) \le e(G\setminus X) \le 3\cdot (2^{k+2}+1) \cdot rq^{1-\alpha} \le (\varepsilon/12) \cdot q$ since $q$ is large enough. Thus every vertex $x$ in $X$ has degree in $G'$ at least
$$(2+\varepsilon/2)\cdot q - e(M) - v(M) \ge (2+\varepsilon/2)\cdot q - 3\cdot e(M) \ge (2+\varepsilon/4)\cdot q$$ 
where we used that $v(M)=2\cdot e(M)$ for a matching.

Now let $X'$ be a subset of $X$ of size exactly $r-e(M)$. Let $G'':=G'(X',V(G')\setminus X')$ and let $H':= {\rm Rainbow}(G'',\phi)$. Since $|X'| \le r-e(M) \le r \le q$, it follows that every vertex $x$ of $X'$ has degree in $G''$ at least $d_{G'}(x)-|X'| \ge (2+\varepsilon/4)q - q \ge (1+\varepsilon/4)q$. 
Note that every vertex of $C$ has degree at most $q$ in $H'$ since no color appears more than $q$ times in $\phi$. Now apply Theorem~\ref{thm:KahnBipartite} to $H'$ with the partition $(X', (V(G')\setminus X') \cup C)$; note crucially this differs from the original partition as given in the rainbow definition but the new partition is still a bipartite hypergraph since in fact $H'$ is tripartite. Thus it follows that there exists an $X'$-perfect matching in $H'$ and hence a rainbow matching $M'$ of $G''$ that saturates $X'$. But then $M\cup M'$ is a rainbow matching of $G$ of size $|M|+|X'| = e(M) + (r-e(M)) = r$ as desired.
\end{proof}

We now prove Theorem~\ref{thm:CKBipartite}.

\begin{proof}[Proof of Theorem~\ref{thm:CKBipartite}]
We proceed by induction on $r$. Let $\phi$ be the edge coloring of $G$ and let $C$ be the set of colors used in $\phi$. We may assume without loss of generality that $\varepsilon = \frac{1}{2^k}$ for some positive integer $k$. Note then that $|C|=(1+\varepsilon)q$ and hence $(1+\varepsilon)q$ is integral.  We may assume without loss of generality (by deleting edges as necessary) that $e(G)= (1+\varepsilon)qr$. 

Note that $\Delta(G)\le (1+\varepsilon)q$. Let $X$ be the set of vertices of $G$ with degree at least $(1+3\varepsilon/4)q$. Let $Y$ be the set of vertices in $X$ that have at most $(\varepsilon/8)q$ neighbors in $X$. Let $G':=G(Y,V(G)\setminus Y)$, that is the bipartite subgraph of $G$ consisting of edges with exactly one end in $Y$. Now let $\alpha$ be the real given by Theorem~\ref{thm:KahnBipartite} for $r=3$ and $\beta = 1/2$ say. 

First suppose that $e(G\setminus X) \ge 2^{k+4}\cdot rq^{1-\alpha}$. Let
$$a_1 := 2^{k+2}+3,~~a_2:=2^{k+2}+4.$$
Let $G_0$ be obtained from $G$ by repeating every edge of $G$ incident with at least one vertex of $X$ exactly $a_1$ times and every other edge of $G$ exactly $a_2$ times. Note that $a_1 = 2^{k+2} \cdot (1+3\varepsilon/4)$ and $a_2 = 2^{k+2}\cdot (1+\varepsilon)$. Let 
$$D:=(2^{k+2}+3)\cdot (1+\varepsilon)q.$$ 
Then every vertex in $X$ has degree in $G_0$ at most 
$$ a_1 \cdot (1+\varepsilon)q = (2^{k+2}+3)\cdot (1+\varepsilon)q = D$$ 
while every vertex in $V(G)\setminus X$ has degree in $G_0$ at most $$a_2 \cdot (1+3\varepsilon/4)q = (2^{k+2}+4)\cdot (1+3\varepsilon/4)q = D.$$ 
Thus $\Delta(G_0) \le D$. 

Note that $G_0$ has multiplicity at most $(2^{k+2}+4)\le D^{1-\beta}$ since $q$ is large enough. Meanwhile, 
$$e(G_0) = (2^{k+2}+3)\cdot e(G) + e(G\setminus X).$$ 
Let $H:= {\rm Rainbow}(G_0,\phi)$. Since $G$ is simple and $\phi$ is a proper edge coloring of $G$, it follows that $H$ has codegrees at most $D^{1-\beta}$.

Thus by Theorem~\ref{thm:KahnKahn} applied to $H$, we have that $\chi(L(H))\le D(1+D^{-\alpha})$ and hence there exists a matching of $H$ of size at least $\frac{e(H)}{D(1+D^{-\alpha})}$. Hence $G_0$ (and hence also $G$) has a rainbow matching $M$ of size at least $\frac{e(G_0)}{D(1+D^{-\alpha})}$.  Note that
$$(2^{k+2}+3)\cdot e(G) = (2^{k+2}+3)\cdot (1+\varepsilon)qr = D\cdot r,$$
and yet 
$$e(G\setminus X) \ge2^{k+4}\cdot rq^{1-\alpha} \ge D^{1-\alpha} \cdot r.$$
Thus we find that 
$$e(M) \ge \frac{e(G_0)}{D(1+D^{-\alpha})} \ge \frac{(2^{k+2}+3)\cdot e(G) + e(G\setminus X)}{D(1+D^{-\alpha})} \ge \frac{D\cdot r + D^{1-\alpha}\cdot r}{D(1+D^{-\alpha})}\ge r,$$ 
as desired.

So we may assume that $e(G\setminus X) \le 2^{k+4}\cdot rq^{1-\alpha}$. Next suppose that $|X| \ge r + 2^{2k+5}\cdot rq^{-\alpha}$. Let $A_1,A_2$ be the parts of the bipartition of $G$. Since there are $(1+\varepsilon)q$ colors each appearing at most $q$ times, we find that $e(G)\le (1+\varepsilon)q^2$. Hence for each $i\in\{1,2\}$, we have that 
$$|X\cap A_i| \le \frac{e(G)}{(1+3\varepsilon/4)q}  \le \frac{(1+\varepsilon)q}{1+3\varepsilon/4} \le (1+\varepsilon/4)\cdot q.$$
For each vertex $x$ of $X$, let $S_x$ be a subset of $N_G(x)\setminus X$ of size exactly $(\varepsilon/2)q$. Note this is possible since $d_G(x) \ge (1+3\varepsilon/4)q$ and $|N_G(x)\cap X| \le \max_{i\in [2]} |X\cap A_i| \le  (1+\varepsilon/4)q$ from above and the fact $G$ is bipartite.  Let $S:=\bigcup_{x\in X} \{ ux: u\in S_x\}$. Note that $|S| = |X|\cdot (\varepsilon/2)q$ since the sets $\{ ux: u\in S_x\}$ are pairwise disjoint. 

Let
$$a_3 := 2^{k+2}+1,~~a_4:=2^{k+2}+2.$$
Let $G_1$ be obtained from $G$ by repeating every edge in $E(G)\setminus (S\cup E(G\setminus X))$ exactly $a_3$ times and every edge in $S\cup E(G\setminus X)$ exactly $a_4$ times. Note that $a_3 = 2^{k+2} (1+\varepsilon/4)$ and $a_4 = 2^{k+2}(1+\varepsilon/2)$. Let 
$$D_1:=(2^{k+2}+2)\cdot (1+3\varepsilon/4)q = \frac{(2^{k+2}+2)(2^{k+2}+3)}{2^{k+2}} \cdot q.$$ 
Then every vertex of $X$ has degree in $G_1$ at most $$a_3\cdot (1+\varepsilon/2)q + a_4 \cdot (\varepsilon/2)q = (2^{k+2}+1)\cdot (1+\varepsilon/2)q + (2^{k+2}+2)\cdot (\varepsilon/2)q = D_1.$$ 
On the other hand every vertex in $V(G)\setminus X$ has degree in $G_1$ at most $$a_4 \cdot (1+3\varepsilon/4)q = (2^{k+2}+2)\cdot (1+3\varepsilon/4)q = D_1.$$ Thus $\Delta(G_1) \le D_1$. 

Note that $G_1$ has multiplicity at most $(2^{k+2}+2)\le D_1^{1-\beta}$ since $q$ is large enough. Meanwhile, 
$$e(G_1) = (2^{k+2}+1)\cdot e(G) + |S| + e(G\setminus X).$$ 
Let $H_1:= {\rm Rainbow}(G_1,\phi)$. Since $G$ is simple and $\phi$ is a proper edge coloring of $G$, it follows that $H_1$ has codegrees at most $D_1^{1-\beta}$.

Thus by Theorem~\ref{thm:KahnKahn} applied to $H_1$, we have that $\chi(H_1)\le D_1(1+D_1^{-\alpha})$ and hence there exists a matching of $H$ of size at least $\frac{e(H)}{D_1(1+D_1^{-\alpha})}$. Hence $G_1$ (and hence also $G$) has a rainbow matching $M$ of size at least $\frac{e(G_1)}{D_1(1+D_1^{-\alpha})}$. Note that
\begin{align*}
(2^{k+2}+1)\cdot e(G) &= (2^{k+2}+1)\cdot (1+\varepsilon)qr = \frac{(2^{k+2}+1)(2^{k+2}+4)}{2^{k+2}} \cdot qr\\
&= \left(D_1-\frac{2}{2^{k+2}}\cdot q\right)\cdot r = \left(D_1-\frac{\varepsilon}{2}\cdot q\right) \cdot r,
\end{align*}
and yet since $|X| \ge r + 2^{2k+5} \cdot rq^{-\alpha} \ge r + \frac{4(2^{k+2}+2)}{\varepsilon}\cdot rq^{-\alpha}$, we have that
$$|S| = \frac{\varepsilon}{2} \cdot q\cdot |X| \ge \frac{\varepsilon}{2} \cdot qr +  2(2^{k+2}+2)\cdot rq^{1-\alpha} \ge \frac{\varepsilon}{2} \cdot qr + D_1^{1-\alpha} \cdot r.$$
Thus we find that 
$$e(G_1) \ge \left(D_1-\frac{\varepsilon}{2}\cdot q\right) \cdot r + \frac{\varepsilon}{2} \cdot qr + D_1^{1-\alpha} \cdot r =(D_1+D_1^{1-\alpha})\cdot r,$$
and hence
$$e(M) \ge \frac{e(G_1)}{D_1(1+D_1^{-\alpha})} \ge r,$$ 
as desired.

So we may assume that $|X| \le r + 2^{2k+5}\cdot rq^{-\alpha}$. Note that
$$e(G)-e(G\setminus X) + e(G[X]) = \sum_{x\in X} d_G(x) \le |X| \cdot (1+\varepsilon)q.$$ 
 Recall that we have assumed for this entire subcase that $e(G\setminus X) \le 2^{k+4}\cdot rq^{1-\alpha}$. Hence we calculate that 
 \begin{align*}
e(G[X]) &\le e(G\setminus X) - e(G) + |X| \cdot (1+\varepsilon)q\\
&\le 2^{k+4}\cdot rq^{1-\alpha} - (1+\varepsilon)rq + \left( r + 2^{2k+5}\cdot rq^{-\alpha}\right) \cdot (1+\varepsilon)q\\
&\le 2^{2k+6} \cdot rq^{1-\alpha}.
\end{align*}
 Since $X\setminus Y$ is the set of vertices of $X$ with at least $(\varepsilon/8) q$ neighbors in $X$, we find that
$$ |X\setminus Y| \le \frac{2\cdot e(G[X])}{(\varepsilon/8)q} \le 2^{3k+10} \cdot rq^{-\alpha}.$$
Hence 
\begin{align*}
e(G\setminus Y) &\le e(G\setminus X) + |X\setminus Y| \cdot (1+\varepsilon)q\\
&\le 2^{k+4}\cdot rq^{1-\alpha} + 2^{3k+10} rq^{-\alpha} \cdot 2q \\
&\le 2^{3k+12}\cdot rq^{1-\alpha}\\
&\le \frac{\varepsilon}{8} \cdot q,
\end{align*}
where for the last inequality we used that $q$ is large enough.

Now by induction, there exists a matching $M$ in $G\setminus Y$ of size at least $\frac{e(G\setminus Y)}{(1+\varepsilon)q}$. Let $G'$ be obtained from $G\setminus V(M)$ by deleting all edges colored with colors of $M$. Note that 
$$e(M) \le e(G\setminus Y) \le  \frac{\varepsilon}{8} \cdot q.$$
Since $e(M)\ge \frac{e(G\setminus Y)}{(1+\varepsilon)q}$, we have that $$e(G)-e(G\setminus Y) \ge (r-e(M))\cdot (1+\varepsilon)q.$$ 
Since $Y$ has maximum degree $(1+\varepsilon)q$, it follows that $$e(G)-e(G\setminus Y)\le |Y|\cdot (1+\varepsilon)q$$ and hence 
$$|Y| \ge r-e(M).$$

Let $G''=G'(Y,V(G')\setminus Y)$. Recall by the definition of $Y$ that every vertex in $Y$ has at most $(\varepsilon/8)q$ neighbors in $X$. Thus every vertex $y$ in $Y$ still has degree in $G''$ at least
$$(1+3\varepsilon/4)q - e(M) - v(M) - (\varepsilon/8)q \ge (1+3\varepsilon/4)q - 3e(M) - (\varepsilon/8)q \ge (1+\varepsilon/4)\cdot q,$$
where we used that $v(M)=2\cdot e(M)$ as $M$ is a matching.

Let $H':= {\rm Rainbow}(G'',\phi)$. Note that every vertex of $C$ has degree at most $q$ in $H'$ since no color appears more than $q$ times in $\phi$. Now apply Theorem~\ref{thm:KahnBipartite} to $H'$ with the partition $(Y,(V(G')\setminus Y)\cup C)$; note crucially this differs from the original partition as given in the rainbow definition but the new partition is still a bipartite hypergraph since in fact $H'$ is tripartite. Thus it follows that there exists a $Y$-perfect matching in $H'$ and hence a rainbow matching $M'$ of $G''$ that saturates $Y$. But then $M\cup M'$ is a rainbow matching of $G$ of size $|M|+|Y| \ge e(M) + (r-e(M)) = r$ as desired.
\end{proof}

We now derive Theorem~\ref{thm:MCCBipartite} from Theorem~\ref{thm:CKBipartite} as follows.

\begin{proof}[Proof of Theorem~\ref{thm:MCCBipartite}]
Let $\varepsilon' = \frac{\varepsilon}{2}$. Let $q_0$ be the value of Theorem~\ref{thm:CKBipartite} for $\varepsilon'$. We set $d_0 = q_0$.

Let $G=(X,Y)$. We may assume without loss of generality that $|X|\ge |Y|$ and hence that $$|X|\ge \frac{v(G)}{2} \ge \left(1+\frac{\varepsilon}{2}\right)d.$$ 
Let $C$ be the set of colors used in the proper edge coloring $\phi$ of $G$. Define a bipartite graph $H = (C,Y)$ with 
$$E(H) = \{ \{c,y\}:~c\in C,~y\in Y,~\exists~xy\in E(G) \text{ with }\phi(xy)=c\}.$$
Note that since $\phi$ is a proper coloring, $H$ is a simple graph. Now we let $\phi'$ be the edge coloring of $H$ defined by letting $\phi'(cy)=x$ for all $cy\in E(H)$. Thus the set of colors used in $\phi'$ is $X$. Since $G$ is a simple graph and $\phi$ is proper, it follows that $\phi'$ is a proper edge coloring of $H$. 

Let $X'$ be a subset of $X$ of size $\left(1+\frac{\varepsilon}{2}\right)d$. For each $x\in X'$, let $E_x$ be a subset of $E(H)$ colored by $x$ of size exactly $d$. Note that $E_x$ exists for all $x\in X'$ since $G$ has minimum degree $d$. Let $H' = \bigcup_{x\in X'} E_x$. Note that $H'$ is a subgraph of $H$ and thus $\phi'$ is also a proper edge coloring of $H'$ with $|X'|= (1+\varepsilon')d$ colors. Moreover, $H'$ has exactly $|X'|\cdot d = (1+\varepsilon')d^2$ edges.

By Theorem~\ref{thm:CKBipartite} with $q=r=d$, we have that $H'$ has rainbow matching $M$ of size $d$. Let 
$$M' := \{ xy \in E(G):~cy\in M,~y\in Y,~x=\phi'(cy)\}.$$ 
Hence $M'$ is a rainbow matching of $G$ of size $d$ as desired.
\end{proof}

\section*{Acknowledgments}
The authors would like to thank David Munh\'{a} Correia for pointing out Theorem~\ref{thm:MCCBipartite} to us and how it follows from our Theorem~\ref{thm:CKBipartite} as a corollary as well as the anonymous reviewers for helpful suggestions. 

\bibliographystyle{plain}
\bibliography{mdelcourt}

\appendix
\section{Proof of Theorem \ref{exceptional talagrand's}}\label{s:Appendix}

In this appendix, we prove Theorem~\ref{exceptional talagrand's}. First we need the following lemma.

\begin{lem}\label{lem:CertScaling}
If $X$ is $b$-certifiable with respect to $\Omega^*$, then for every $s \ge 0$ and $\omega\in\Omega\setminus\Omega^*$ such that $X(\omega) \geq s$, there exists a $b$-certificate for $X, \omega, s$, and $\Omega^*$,
\end{lem}
\begin{proof}
Since $X$ is $b$-certifiable with respect to $\Omega^*$, there exists a $b$-certificate for $X,\omega, s':=X(\omega)$ and $\Omega^*$ by definition, that is there exists $I\subseteq \{1,\ldots,n\}$ and a vector $(c_i': i\in I)$ with $\sum_{i\in I} (c_{i}')^2 \le bs'$ such that for all $I'\subseteq I$, we have that
  \begin{equation*}
    X(\omega') \geq s' - \sum_{i\in I'} c_i',
  \end{equation*}
  for all $\omega' = (\omega'_1, \dots, \omega'_n)\in\Omega\setminus\Omega^*$  such that $\omega_i=\omega_i'$ for all $i\in I\setminus I'$.

Note that $s'=X(\omega)\ge s$ by assumption. Now for all $i\in I$, we define $c_i := c_i' \cdot \frac{s}{s'}$. We claim that $I$ and $(c_i:i\in I)$ is a $b$-certificate for $X,\omega,s$ and $\Omega^*$ as desired. To see this, note that
$$\sum_{i\in I} c_i^2 = \left(\frac{s}{s'}\right)^2 \sum_{i\in I} (c_i')^2 \le \left(\frac{s}{s'}\right)^2 bs' \le bs,$$
where the last inequality follows since $s\le s'$. Furthermore, for all $\omega' = (\omega'_1, \dots, \omega'_n)\in\Omega\setminus\Omega^*$  such that $\omega_i=\omega_i'$ for all $i\in I\setminus I'$, we have that
$$X(\omega') \ge s' - \sum_{i\in I'} c_i' = \frac{s'}{s} \left( s- \sum_{i\in I'} c_i \right) \ge s- \sum_{i\in I'} c_i,$$
where the last inequality follows since $s'\ge s$. This completes the claim and hence the proof of the lemma.
\end{proof}

In order to prove Theorem~\ref{exceptional talagrand's}, we prove the following theorem which yields concentration around the median under the same conditions.
\begin{thm}\label{exceptional talagrand's with median}
  If $X$ is $b$-certifiable with respect to $\Omega^*$, then for any $t > 0$,
  \begin{equation*}
    \Prob{|X - \Med(X)| > t} \leq 4\exp\left({-\frac{t^2}{4b(\Med(X) + t)}}\right) + 4\Prob{\Omega^*}
  \end{equation*}
\end{thm}

We then prove that the expectation and median are close as in the following lemma.
\begin{lem}\label{expectation close to median}
If $X$ is $b$-certifiable with respect to $\Omega^*$ and $M = \sup X$, then
\begin{equation*}
|\Expect{X} - \Med(X)| \leq 48\sqrt{b\Expect{X}} +  64b + 4M\Prob{\Omega^*}.
\end{equation*}
\end{lem}
\begin{proof}
Let $Y = X + \Expect{X}$.  Note that $\Expect{Y} - \Med(Y) = \Expect{X} - \Med(X)$, $\Med(Y) \geq \Expect{X} > 0$, and $\Expect{Y} = 2\Expect{X}$.
Note also that by Jensen's inequality applied to the absolute value function, we find that
\begin{equation*}
|\Expect{Y} - \Med(Y)| \leq \Expect{|Y - \Med(Y)|}.
\end{equation*}

Let $L = \lfloor M/(\sqrt{b\Med(Y)})\rfloor$, and note that $|Y - \Med(Y)| \leq  (L + 1)\sqrt{b\Med(Y)}$.  By partitioning the possible values of $|Y - \Med(Y)|$ into intervals of length $\sqrt{b\Med(Y)}$, we get

\begin{align*}
\Expect{|Y - \Med(Y)|} &\leq \begin{aligned}[t]
	\sum_{\ell=0}^L \sqrt{b\Med(Y)}(\ell + 1) &\left(\Prob{|Y - \Med(Y)| \geq \ell \sqrt{b\Med(Y)}}\right.\\
	&\left. - \Prob{|Y - \Med(Y)| \geq (\ell + 1) \sqrt{b\Med(Y)}}\right).\end{aligned}\\
&= \sum_{\ell=0}^L \sqrt{b\Med(Y)}\left(\Prob{|Y - \Med(Y)| \geq \ell \sqrt{b\Med(Y)}}\right).
\end{align*}
By applying Theorem~\ref{exceptional talagrand's with median} with $t=\ell \sqrt{b\Med(Y)}$ to every summand,
\begin{equation*}
  \Expect{|Y - \Med(Y)|} \leq 4\sqrt{b\Med(Y)}\sum_{\ell=0}^L \left(\exp\left({-\frac{\ell^2b\Med(Y)}{4b(\Med(Y) + \ell \sqrt{b\Med(Y)})}}\right) + \Prob{\Omega^*}\right).
\end{equation*}
Note that for each $\ell\in\{0, \dots, L\}$,
\begin{multline*}
  \exp\left(-{\frac{\ell^2b\Med(Y)}{4b(\Med(Y) + \ell \sqrt{b\Med(Y)})}}\right) \leq \exp\left(-{\frac{\ell^2b\Med(Y)}{8b\max\{\Med(Y), \ell \sqrt{b\Med(Y)}\}}}\right)\\
  \leq \exp\left(-{\frac{\ell^2b\Med(Y)}{8b\Med(Y)}}\right) + \exp\left(-{\frac{\ell^2b\Med(Y)}{8b\ell\sqrt{b\Med(Y)}\}}}\right) = \exp\left(-\ell^2/8\right) + \exp\left(-\frac{\ell\sqrt{\Med(Y)}}{8\sqrt{b}}\right).
\end{multline*}
Note also that
\begin{equation*}
  4\sqrt{b\Med(Y)}\sum_{\ell=0}^L\Prob{\Omega^*} \leq 4M\Prob{\Omega^*}.
\end{equation*}
Therefore
\begin{equation*}
  \Expect{|Y - \Med(Y)|} \leq 4\sqrt{b\Med(Y)}\sum_{\ell=0}^\infty \left(\exp\left({-\ell^2/8}\right) + \exp\left({-\frac{\ell\sqrt{\Med(Y)}}{8\sqrt{b}}}\right)\right) + 4M\Prob{\Omega^*}.
\end{equation*}
Note that $\sum_{\ell=0}^\infty e^{-\ell x} = \frac{1}{1 - e^{-x}}$.  Note also that $\frac{x}{2} \leq 1 - e^{-x}$ if $x < \frac{3}{2}$.   Since $\frac{1}{1 - e^{-x}} < 2$ when $x\geq \frac{3}{2}$, $\frac{1}{1 - e^{-x}} \leq \max\{2, \frac{2}{x}\}$.  Therefore
\begin{equation*}
  \sum_{\ell=0}^\infty \exp\left(-\frac{\ell\sqrt{\Med(Y)}}{8\sqrt{b}}\right) \leq \max\left\{2, \frac{16\sqrt{b}}{\sqrt{\Med(Y)}}\right\}.
\end{equation*}

Note that $\sum_{\ell=0}^\infty e^{-\ell^2/ 8} < 4$.  Therefore
\begin{equation*}
  \Expect{|Y - \Med(Y)|} \leq 4\sqrt{b\Med(Y)}\left(4 + \max\left\{2, \frac{16\sqrt{b}}{\sqrt{\Med(Y)}}\right\}\right) + 4M\Prob{\Omega^*}.  
\end{equation*}
Since the maximum of two numbers is at most their sum,
\begin{equation*}
  \Expect{|Y - \Med(Y)|} \leq 24\sqrt{b\Med(Y)} + 64b + 4M\Prob{\Omega^*}.
\end{equation*}
Since $\Med(Y) \leq 2\Expect{Y} \leq 4\Expect{X}$ by Markov's inequality,
\begin{equation*}
\Expect{Y - \Med(Y)|} \leq 48\sqrt{b\Expect{X}} +  64b + 4M\Prob{\Omega^*},
\end{equation*}
as desired.
\end{proof}

Now we can prove Theorem~\ref{exceptional talagrand's} assuming Theorem~\ref{exceptional talagrand's with median}.
\begin{proof}[Proof of Theorem \ref{exceptional talagrand's}]
Since $t > 96\sqrt{b\Expect{X}} +  128b + 8M\Prob{\Omega^*},$
\begin{equation}\label{t over 2 bound}
\frac{t}{2} > 48\sqrt{b\Expect{X}} +  64b + 4M\Prob{\Omega^*}.
\end{equation}
By applying Lemma \ref{expectation close to median} and then \eqref{t over 2 bound},
\begin{equation*}
\Prob{|X - \Expect{X}| > t} \leq \Prob{|X - \Med(X)| > \frac{t}{2}}.
\end{equation*}
Since $\Med(X) \leq 2\Expect{X}$, Theorem~\ref{exceptional talagrand's with median} implies that
\begin{align*}
\Prob{|X - \Med(X)| > \frac{t}{2}} &\leq 4\exp\left({-\frac{(t/2)^2}{4b(2\Expect{X} + (t/2))}}\right) + 4\Prob{\Omega^*},\\
 &=4\exp\left({-\frac{t^2}{2b(4\Expect{X} + t)}}\right) + 4\Prob{\Omega^*}
\end{align*}
as desired.
\end{proof}

It remains to prove Theorem~\ref{exceptional talagrand's with median}.

Let $((\Omega_i, \Sigma_i, \mathbb P_i))_{i=1}^n$ be probability spaces and $(\Omega, \Sigma, \mathbb P)$ their product space.  For a set $A\subseteq \Omega$ and event $\omega\in\Omega$, let
\begin{equation}
d(\omega, A) = \sup_{||\alpha||=1}\left\{\tau : \sum_{i:\omega_i\neq\omega'_i}\alpha_i\geq\tau \rm{\ for\ all\ }\omega'\in A\right\}.
\end{equation}

We use the original version of Talagrand's Inequality.
\begin{thm}[Talagrand's Inequality \cite{T95}]\label{OG Tala}
If $A,B\subseteq \Omega$ are measurable sets such that for all $\omega\in B$, $d(\omega, A) \geq \tau$, then
$$\Prob{A}\Prob{B}\leq e^{\frac{-\tau^2}{4}}.$$
\end{thm}

We can now prove Theorem~\ref{exceptional talagrand's with median}.
\begin{proof}[Proof of Theorem~\ref{exceptional talagrand's with median}]
It suffices to show that
\begin{equation}\label{one sided estimation}
\Prob{X\leq \Med(X) - t} \leq 2\exp\left({-\frac{t^2}{4b(\Med(X) + t)}}\right) + 2\Prob{\Omega^*}
\end{equation}
and
\begin{equation}\label{other one sided estimation}
\Prob{X\geq \Med(X) + t} \leq 2\exp\left({-\frac{t^2}{4b(\Med(X) + t)}}\right) + 2\Prob{\Omega^*}.
\end{equation}
Let 
\begin{align*}
&A = \{\omega\in \Omega\backslash \Omega^* : X(\omega) \geq \Med(X) + t\}, \text{ and}\\
&B = \{\omega\in \Omega\backslash \Omega^* : X(\omega) \leq \Med(X) \}.
\end{align*}

We need to show the following.
\begin{claim}\label{A and B satisfy OG Tala}
For all $\omega\in A$, $d(\omega, B) \geq \frac{t}{\sqrt{b(\Med(X) + t)}}$.
\end{claim}
\begin{proofclaim}
Since $X$ is $b$-certifiable, we have by Lemma~\ref{lem:CertScaling} that there exists a $b$-certificate for $X, \omega, \Med(X) + t,$ and $\Omega^*$, that is a subset $I\subseteq\{1,\ldots,n\}$ and a vector $(c_i:i\in I)$. 

Let $\omega'\in B$.  Let $I' = \{i\in I:  \omega_i \ne \omega'_i\}$. By definition of $b$-certificate since $X(\omega') \leq X(\omega)-t$, we have that 
$$\sum_{i\in I'} c_i \ge t.$$
Let $r = \sqrt{\sum_{i\in I}c_i^2}$. By the definition of $b$-certificate, we have that $\sum_{i\in I}c_i^2 \le b(\Med(X)+t)$ and hence $r\le \sqrt{b(\Med(X)+t)}$.

Now we set $\alpha = (\alpha_i: i\in[n])$ where $\alpha_i = c_i/ r$ if $i\in I$ and $\alpha_i = 0$ otherwise. Thus 
$$\sum_{i\in I'} c_i/r \ge t/r \ge \frac{t}{\sqrt{b(\Med(X) + t)}}.$$
Hence $d(\omega,B) \ge \frac{t}{\sqrt{b(\Med(X) + t)}}$ as desired.
\end{proofclaim}

Now \eqref{other one sided estimation} follows from Claim \ref{A and B satisfy OG Tala} and Theorem~\ref{OG Tala}.  The proof of~\eqref{one sided estimation} is similar, so we omit it.
\end{proof}

\end{document}